\documentclass{ps}
\usepackage{amsmath,amsfonts}
\usepackage{color}
\usepackage[applemac]{inputenc}
\usepackage{mathrsfs}
\usepackage{hyperref}
\usepackage{amssymb}
\usepackage{latexsym,graphicx}

\def \O{\mathcal{O}}

\def \1{\mbox{\textbf{1}}}

\def\leftB{[\![}
\def\rightB{]\!]}

\def\L{{\cal L}}
\newcommand \A[1]{{\bf (#1)}}

\def\O{{\cal{O}}}

\def\F{{\cal F}}

\def\bint#1^#2{\displaystyle{\int_{#1}^{#2}}}
\def\bsum#1^#2{\displaystyle{\sum_{#1}^{#2}}}

\def\xdt_#1{X_#1(\Delta t)}

\newtheorem{THM}{Theorem}[section]

\newtheorem{PROP}{Proposition}[section]

\newtheorem{LEMME}{Lemma}[section]
\newtheorem{REM}{Remark}[section]

\newcommand{\R}{\mathbb{R}}
\newcommand{\N}{\mathbb{N}}
\renewcommand{\P}{\mathbb{P}}
\newcommand{\E}{\mathbb{E}}
\renewcommand{\L}{\mathbb{L}}

\newcommand{\delimleft}[2]{\ifcase #1\or
    \bigl#2\or %
    \Bigl#2\or %
    \biggl#2\or %
    \Biggl#2\or %
    \left#2\fi}
\newcommand{\delimright}[2]{\ifcase #1\or
    \bigr#2\or %
    \Bigr#2\or %
    \biggr#2\or %
    \Biggr#2\or %
    \right#2\fi}

\DeclareMathOperator*{\Argmin}{Argmin}

\begin{document}
\title{Multi-level stochastic approximation algorithms}
\author{N. Frikha}\address{LPMA, Universit\'e Paris Diderot, Bt. Sophie Germain, 5 Rue Thomas Mann, 75205 Paris, Cedex 13, frikha@math.univ-paris-diderot.fr, \url{http://www.proba.jussieu.fr/pageperso/frikha/}}
\date{\today}
\begin{abstract} This paper studies multi-level stochastic approximation algorithms. Our aim is to extend the scope of the multilevel Monte Carlo method recently introduced by Giles \cite{Giles:08} to the framework of stochastic optimization by means of stochastic approximation algorithm. We first introduce and study a two-level method, also referred as statistical Romberg stochastic approximation algorithm. Then, its extension to multi-level is proposed. We prove a central limit theorem for both methods and describe the possible optimal choices of step size sequence. Numerical results confirm the theoretical analysis and show a significant reduction in the initial computational cost.
\end{abstract}
\maketitle

\section{Introduction}
In this paper we propose and analyze a multi-level paradigm for stochastic optimization problem by means of stochastic approximation schemes. The multi-level Monte Carlo method introduced by Heinrich \cite{heinrich2001} and popularized in numerical probability by \cite{Keb:2005} and \cite{Giles:08} allows to significantly increase the computational efficiency of the expectation of an $\R$-valued non-simulatable random variable $Y$ that can only be strongly approximated by a sequence $(Y^{n})_{n\geq1}$ of easily simulatable random variables (all defined on the same probability space) as the \emph{bias parameter} $n$ goes to infinity with a \emph{weak error} or \emph{bias} $\E[Y]-\E[Y^{n}]$ of order $n^{-\alpha}$, $\alpha>0$. Let us be more specific. In this context, the standard Monte Carlo method uses the statistical estimator $M^{-1} \times \sum_{j=1}^{M} Y^{n,j}$ where the $(Y^{n,j})_{j\in \leftB 1,M\rightB}$ are $M$ independent copies of $Y^{n}$. Given the order of the weak error, a natural question is to find the optimal choice of the sample size $M$ to achieve a global error. If the weak error is of order $n^{-\alpha}$ then for a total error of order $n^{-\alpha}$ ($\alpha \in [1/2,1]$), the minimal computation necessary for the standard Monte Carlo algorithm is obtained for $M=n^{2\alpha}$, see \cite{duff:glyn:1995}. So, if the computational cost required to simulate one sample of $U^{n}$ is of order $n$ then the optimal computational cost of the Monte Carlo method is $C_{MC}=C\times n^{2\alpha+1}$, for a positive constant $C>0$.

In order to reduce the complexity of the computation, the principle of the multi-level Monte Carlo method introduced by Giles \cite{Giles:08} as a generalization of Kebaier's approach \cite{Keb:2005} consists in using the telescopic sum
$$
\E[Y^{m^{L}}] = \E[Y^{1}] + \sum_{\ell=1}^{L} \E[Y^{m^{\ell}} - Y^{^{m^{\ell-1}}}],
$$

\noindent where $m\in \N^{*}\backslash{\left\{1\right\}}$ satisfies $m^{L}=n$. For each level $\ell \in \left\{1,\cdots,L\right\}$ the numerical computation of $\E[Y^{m^{\ell}}-Y^{m^{\ell-1}}]$ is achieved by the standard Monte Carlo method with $N_{\ell}$ independent samples of $(Y^{m^{\ell-1}}, Y^{m^{\ell}})$. An important point is that the random sample $Y^{m^{\ell}}$ and $Y^{m^{\ell-1}}$ are perfectly correlated. Then the expectation $\E[Y^{n}]$ is approximated by the following multi-level estimator
$$
\frac{1}{N_0} \sum_{j=1}^{N_0} Y^{1,j} + \sum_{\ell=1}^{L} \frac{1}{N_\ell} \sum_{j=1}^{N_{\ell}}\left( Y^{m^{\ell},j} - Y^{m^{\ell-1},j} \right),
$$

\noindent where for each level $\ell$, $(Y^{m^{\ell},j})_{j\in \leftB1, N_{\ell}\rightB}$ is a sequence of i.i.d. random variables with the same law as $Y^{m^{\ell}}$. 

Based on an analysis of the variance, Giles \cite{Giles:08} proposed an optimal choice for the sequence $(N_{\ell})_{1 \leq \ell \leq L}$ which minimizes the total complexity of the algorithm. More recently, Ben Alaya and Kebaier \cite{ben:keb:12} proposed a different analysis to obtain the optimal choice of the parameters relying on a Lindeberg Feller central limit theorem (CLT) for the multi-level Monte Carlo algorithm. To obtain a global error of order $n^{-\alpha}$, both approaches allow to achieve a complexity of order $n^{2\alpha} (\log n)^2$ if the $L^{2}(\P)$ strong approximation rate $\E|U^{n}-U|^2]$ of $U$ by $U^{n}$ is of order $1/n$. Hence, the multi-level Monte Carlo method is significantly more effective than the crude Monte Carlo and the statistical Romberg methods. Originally introduced for the computation of expectations involving stochastic differential equation (SDE), it has been widely applied to various problems of numerical probability, see Giles \cite{Giles:08b}, Dereich \cite{Dereich11}, Giles, Higham and Mao \cite{Giles:09} among others. We refer the interested reader to the webpage: \url{http://people.maths.ox.ac.uk/gilesm/mlmc_community.html}.

In the present paper, we are interested in broadening the scope of the multi-level Monte Carlo method to the framework of stochastic approximation (SA) algorithm. Introduced by Robbins and Monro \cite{Robbins1951}, these recursive simulation based algorithms appear as effective and widely used procedures to solve inverse problems. To be more specific, their aim is to find a zero of a continuous function $h:\R^{d} \rightarrow \R^{d}$ which is unknown to the experimenter but can only be estimated through experiments. Successfully and widely investigated from both a theoretical and applied point of view since this seminal work, such procedures are now commonly used in various contexts such as convex optimization since minimizing a function amounts to finding a zero of its gradient. In the general Robbins-Monro procedure, the function $h$ writes $h(\theta):=\E[H(\theta,U)]$ where $H:\R^d\times \R^q \rightarrow \R^d$ and $U$ is an $\R^q$-valued random vector. To estimate the zero of $h$, they proposed the algorithm
\begin{equation}
\label{RMorigine}
\theta_{p+1} = \theta_p - \gamma_{p+1} H(\theta_p,U^{p+1}), \ \ p\geq0
\end{equation}

\noindent where $(U^{p})_{p\geq1}$ is an i.i.d. sequence of copies of $U$ defined on a probability space $(\Omega, \mathcal{F}, \mathbb{P})$, $\theta_0$ is independent of the innovation of the algorithm with $\E|\theta_0|^2<+\infty$ and $\gamma=(\gamma_{p})_{p \geq 1}$ is a sequence of non-negative deterministic and decreasing steps satisfying the assumption
\begin{equation}
\label{STEP}
\sum_{p \geq 1} \gamma_{p} = + \infty, \ \ \mbox{and} \ \ \sum_{p\geq1} \gamma_{p}^{2} < + \infty.
\end{equation} 

\noindent When the function $h$ is the gradient of a convex potential, the recursive procedure \eqref{RMorigine} is a stochastic gradient algorithm. Indeed, replacing $H(\theta_p, U^{p+1})$ by $h(\theta_p)$ in \eqref{RMorigine} leads to the usual deterministic descent gradient procedure. When $h(\theta)=k(\theta)-\ell$, $\theta \in \mathbb{R}$, where $k$ is a monotone function, say increasing, which writes $k(\theta)=\E[K(\theta,U)]$, $K: \mathbb{R} \times \mathbb{R}^{q} \rightarrow \mathbb{R}$ being a Borel function and $\ell$ a given desired level, then setting $H=K-\ell$, the recursive procedure \eqref{RMorigine} aims to compute the value $\bar{\theta}$ such that $k(\bar{\theta})=\ell$. 

As in the case of the Monte Carlo method described above, the random vector $U$ is not directly simulatable (at a reasonable cost) but can only be approximated by another sequence of easily simulatable random vectors $((U^{n})^{p})_{p\geq1}$, which strongly approximates $U$ as $n\rightarrow +\infty$ with a standard weak discretization error (or bias) $\E[f(U)]-\E[f(U^{n})]$ of order $n^{-\alpha}$ for a specific class of functions $f\in \mathcal{C}$. The computational cost required to simulate one sample of $U^{n}$ is of order $n$ that is $Cost(U^{n})=K\times n$ for some positive constant $K$. One standard situation corresponds to the case of a discretization of an SDE by means of an Euler-Maruyama scheme with $n$ time steps.  

Some typical applications are the computations of the implied volatility or the implied correlation which both boil down to finding the zero of a function which writes as an expectation. Computing the Value-at-Risk and the Conditional Value-at-Risk of a financial portfolio when the dynamics of the underlying assets are given by an SDE also appears as an inverse problem for which a SA scheme may be devised, see e.g. \cite{bar:fri:pag:09, bar:fri:pag:09:2}. The risk minimization of a financial portfolio by means of SA has been investigated in \cite{bar:fri:pag:10, fri:14}. For more applications and a complete overview in the theory of stochastic approximation, the reader may refer to \cite{Duflo1996}, \cite{Kushner2003} and \cite{ben:met:pri}.

The important point here is that the function $h$ is generally neither known nor computable (at least at reasonable cost) and since the random variable $U$ cannot be simulated, estimating $\theta^*$ using the recursive scheme \eqref{RMorigine} is not possible. Therefore, two steps are needed to compute $\theta^*$:
\begin{trivlist}
\item[-] the first step consists in approximating the zero $\theta^*$ of $h$ by the zero $\theta^{*,n}$ of $h^{n}$ defined by $h^{n}(\theta):=\E[H(\theta, U^{n})]$, $\theta \in \R^d$. It induces \emph{an implicit weak error} which writes 
$$
\mathcal{E}_{D}(n):= \theta^{*}-\theta^{*,n}.
$$ 

\noindent Let us note that $\theta^{*,n}$ appears as a proxy of $\theta^*$ and one would naturally expect that $\theta^{*,n}\rightarrow \theta^{*}$ as the bias parameter $n$ tends to infinity.


\item[-] the second step consists in approximating $\theta^{*,n}$ by $M\in \N^{*}$ steps of the following SA scheme
\begin{equation}
\label{RM}
\theta^{n}_{p+1} = \theta^{n}_{p} - \gamma_{p+1} H(\theta^{n}_{p}, (U^{n})^{p+1}), \ p \in \leftB 0, M-1\rightB,
\end{equation}

\noindent where $((U^{n})^{p})_{p\in \leftB1,M\rightB}$ is an i.i.d. sequence of random variables with the same law as $U^{n}$, $\theta^{n}_0$ is independent of the innovation of the algorithm with $\sup_{n\geq1}\E[|\theta^{n}_0|^2]<+\infty$ and $\gamma=(\gamma_{p})_{p \geq 1}$ is a sequence of non-negative deterministic and decreasing steps satisfying \eqref{STEP}.
This induces a \emph{statistical error} which writes 
$$
\mathcal{E}_S(n, M, \gamma):= \theta^{*,n} - \theta^{n}_{M}.
$$
\end{trivlist}

The global error between $\theta^{*}$, the quantity to estimate, and its implementable approximation $\theta^{n}_M$ can be decomposed as follows:
\begin{align*}
\mathcal{E}_{glob}(n, M, \gamma) & = \theta^{*} - \theta^{*,n} + \theta^{*,n}- \theta^{n}_M\\
& := \mathcal{E}_{D}(n) + \mathcal{E}_S(n, M, \gamma).
\end{align*} 

The first step of our analysis consists in investigating the behavior of the \emph{implicit weak error} $\mathcal{E}_{D}(n)$. Under mild assumptions on the functions $h$ and $h^{n}$, namely the local uniform convergence of $(h^{n})_{n\geq1}$ towards $h$ and a mean reverting assumption of $h$ and $h^{n}$, we prove that $\lim_n \mathcal{E}_{D}(n) = 0$. We next show that under additional assumption, namely the local uniform convergence of $(Dh^{n})_{n\geq1}$ towards $Dh$ and the non-singularity of $Dh(\theta^*)$, the rate of convergence of the \emph{standard weak error} $h^{n}(\theta)-h(\theta)$, for a fixed $\theta \in \R^d$, transfers to the \emph{implicit weak error} $\mathcal{E}_{D}(n) = \theta^*-\theta^{*,n}$.  

Regarding the \emph{statistical error} $\mathcal{E}_S(n, M, \gamma):= \theta^{*,n}- \theta^{n}_M$, it is well-known that under standard assumptions, i.e. a mean reverting assumption on $h^{n}$ and a growth control of the $L^{2}(\P)$-norm of the noise of the algorithm, the Robbins-Monro theorem guarantees that $\lim_{M}\mathcal{E}_S(n, M, \gamma)=0$ for each fixed $n\in \N^{*}$, see Theorem \ref{RM:THM} below. Moreover, under mild technical conditions, a CLT holds at rate $\gamma^{-1/2}(M)$, that is, for each fixed $n\in\N^*$, $\gamma^{-1/2}(M) \mathcal{E}_S(n, M, \gamma)$ converges in distribution to a normally distributed random variable with mean zero and finite covariance matrix, see Theorem \ref{CLT_SA} below. The reader may also refer to \cite{fri:men:12, fat:fri:13} for some recent developments on non-asymptotic deviation bounds for the statistical error. In particular if we set $\gamma(p)=\gamma_0/p$, $\gamma_0 >0$, $p\geq1$, the weak convergence rate is $\sqrt{M}$ provided that $ 2\mathcal{R}e(\lambda_{min}) \gamma_0 > 1$ where $\lambda_{min}$ denotes the eigenvalue of $Dh(\theta^*)$ with the smallest real part. However, this local condition on the Jacobian matrix of $h$ at the equilibrium is difficult to handle in practical situation. 

To circumvent such a difficulty, it is fairly well-known that the key idea is to carefully smooth the trajectories of a converging SA algorithm by averaging according to the \emph{Ruppert \& Polyak averaging principle}, see e.g. \cite{Ruppert1991, Polyak1992}. It consists in devising the original SA algorithm \eqref{RM} with a slow decreasing step and to simultaneously compute the empirical mean $(\bar{\theta}^{n}_{p})_{p\geq1}$ (which $a.s.$ converges to $\theta^{*,n}$) of the sequence $(\theta^{n}_{p})_{p\geq0}$ by setting
\begin{align}
\label{RM_AV_int}
\bar{\theta}^{n}_{p} & = \frac{\theta^{n}_0 + \theta^{n}_1+ \cdots + \theta^{n}_p}{p+1} = \bar{\theta}^{n}_{p-1} - \frac{1}{p+1}\left(\bar{\theta}^{n}_{p-1}-\theta^{n}_{p}\right).
\end{align}

The statistical error now writes $\mathcal{E}_S(n, M, \gamma):= \theta^{*,n} - \bar{\theta}^{n}_M$ and under mild assumptions a CLT holds at rate $\sqrt{M}$ without any stringent condition on $\gamma_0$.

Given the order of the implicit weak error and a step sequence $\gamma$ satisfying \eqref{STEP} a natural question is to find the optimal balance between the value of $n$ and the number $M$ of steps in \eqref{RM} in order to achieve a given global error. This problem was originally investigated in \cite{duff:glyn:1995} for the standard Monte Carlo method. The error between $\theta^*$ and the approximation $\theta^{n}_M$ writes $\theta^{n}_{M} - \theta^{*} = \theta^{n}_{M} - \theta^{*,n} + \theta^{*,n} - \theta^{*}$ suggesting to select $M=\gamma^{-1}(1/n^{2\alpha})$, where $\gamma^{-1}$ is the inverse function of $\gamma$, when the \emph{weak error} is of order $n^{-\alpha}$. However, due to the non-linearity of the SA algorithm \eqref{RM}, the methodology developed in \cite{duff:glyn:1995} does not apply in our context. The key tool to tackle this question consists in linearizing the dynamic of $(\theta^{n}_p)_{p \in \leftB1,M\rightB}$ around its target $\theta^{*,n}$, quantifying the contribution of the non linearities in the space variable $\theta^{n}_{p}$ and the innovations and finally exploiting stability arguments from SA schemes. Optimizing with respect to the usual choice of the step sequence, the minimal computational cost (to achieve an error of order $n^{-\alpha}$) given by $C_{\text{SA}} = K \times n \times \gamma^{-1}(1/n^{2\alpha})$ is reached by setting $\gamma(p)=\gamma_0/p$, $p\geq1$, provided that the constant $\gamma_0$ satisfies a stringent condition involving $h^{n}$, leading to a complexity of order $n^{2\alpha+1}$. Considering the empirical mean sequence $(\bar{\theta}^{n}_{p})_{p \in \leftB1,n^{2\alpha} \rightB}$ instead of the crude SA estimate also allows to reach the optimal complexity for free without any condition on $\gamma_0$. 

To increase the computational efficiency for the estimation of $\theta^*$ by means of SA algorithm, we investigate in a second part multi-level SA algorithms. The first one is a two-level method, also referred as the statistical Romberg SA method. It consists in approximating the unique zero $\theta^*$ of $h$ by $\Theta^{sr}_n = \theta^{n^{\beta}}_{M_1} + \theta^{n}_{M_2} - \theta^{n^{\beta}}_{M_2}$, $\beta \in (0,1)$. The couple $(\theta^{n}_{M_2}, \theta^{n^{\beta}}_{M_2})$ is computed using $M_2$ independent copies of $(U^{n},U^{2n})$. Moreover the random samples used to obtain $\theta^{n^{\beta}}_{M_1}$ are independent of those used for the computation of $(\theta^{n}_{M_2}, \theta^{n^{\beta}}_{M_2})$. For an implicit weak error of order $n^{-\alpha}$, we prove a CLT for the sequence $(\Theta^{sr}_n)_{n\geq1}$ through which we are able to optimally set $M_1$, $M_2$ and $\beta$ with respect to $n$ and the step sequence $\gamma$. The intuitive idea is that when $n$ is large, $(\theta^{n}_{p})_{p\in \leftB0,M_2\rightB}$ and $(\theta^{n^{\beta}}_{p})_{p\in \leftB0,M_2\rightB}$ are close to the SA scheme $(\theta_{p})_{p\in \leftB0,M_2\rightB}$ devised with the innovation variables $(U^{p})_{p\geq1}$ so that the correction term writes $ \theta^{n}_{M_2} -\theta_{M_2} - (\theta^{n^{\beta}}_{M_2}- \theta_{M_2})$. Then we quantify the two main contributions in this decomposition, namely the one due to the non linearity in the space variables $(\theta^{n^{\beta}}_p,\theta^{n}_p, \theta_p)_{p \in \leftB0,M_2\rightB}$ and the one due to the non linearity in the innovation variables $(U^{n^{\beta},p}, U^{n,p},  U^{p})_{p\geq1}$. Under mild smoothness assumption on the function $H$, the weak rate of convergence is ruled by the non linearity in the innovation variables for which we use the weak convergence of the normalized error $n^{\rho}(U^{n}-U)$, $\rho \in (0,1/2]$. The optimal choice of the step sequence is again $\gamma_p = \gamma_0/p$, $p\geq1$ and induces a complexity for the procedure given by $C_{\text{SA-SR}}=K\times n^{2\alpha + 1/(1+\rho)}$, provided that $\gamma_0$ satisfies again a condition involving $h^{n}$ which is difficult to handle in practice. By considering the empirical mean sequence $\bar{\Theta}^{sr}_n = \bar{\theta}^{n^{\beta}}_{M_3} + \bar{\theta}^{n}_{M_4} - \bar{\theta}^{n^{\beta}}_{M_4},$ where $(\bar{\theta}^{n^{\beta}}_{p})_{p\in \leftB 0, M_3 \rightB}$ and $(\bar{\theta}^{n}_{p},\bar{\theta}^{n^{\beta}}_{p})_{p\in \leftB 0, M_4 \rightB}$ are respectively the empirical means of the sequences $(\theta^{n^{\beta}}_{p})_{p\in \leftB 0,M_3 \rightB}$ and $(\theta^{n}_{p},\theta^{n^{\beta}}_{p})_{p \in \leftB 0, M_4 \rightB}$ devised with the same slow decreasing step sequence, this optimal complexity is reached for free by setting $M_3=n^{2\alpha}$, $M_4=n^{2\alpha-1/(1+\rho)}$ without any condition on $\gamma_0$.

Moreover, we generalize this approach to the case of multi-level SA method. In the spirit of \cite{Giles:08} for Monte Carlo path simulation, the multi-level SA scheme estimates $\theta^{*,n}$ by computing the quantity $\Theta^{ml}_{n} = \theta^{1}_{M_0} + \sum_{\ell=1}^{L} \theta^{m^{\ell}}_{M_\ell} - \theta^{m^{\ell -1}}_{M_\ell}$ where for every $\ell$, the couple $(\theta^{m^{\ell}}_{M_\ell}, \theta^{m^{\ell -1}}_{M_\ell})$ is obtained using $M_\ell$ independent copies of $(U^{m^{\ell-1}}, U^{m^{\ell}})$. Here again to establish a CLT for this estimator (in the spirit of \cite{ben:keb:12} for the Monte Carlo path simulation), our analysis follows the lines of the methodology developed so far. The optimal computational cost to achieve an accuracy of order $1/n$ is reached by setting $M_0 = \gamma^{-1}(1/n^{2})$, $M_{\ell}=\gamma^{-1}(m^{\ell}\log(m)/(n^2 \log(n) (m-1)))$, $\ell=1, \cdots, L$ in the case $\rho=1/2$. Once again the step sequence $\gamma(p)=\gamma_0/p$, $p\geq1$, is optimal among the usual choices of step sequence and it induces a complexity for the procedure given by $C_{\text{SA-ML}} = K \times n^{2} (\log(n))^2$. We thus recover the rates as in the multi-level Monte Carlo path simulation for SDE obtained in \cite{Giles:08} and \cite{ben:keb:12}.

The paper is organized as follows. In the next section we state our main results and list the assumptions. Section \ref{proof:sec} is devoted to the proofs. In Section \ref{num:res:sec} numerical results are presented to confirm the theoretical analysis. Finally, Section \ref{technical:res:sec} is devoted to technical results which are useful throughout the paper.


\section{Main results}\label{gen:frame:sec}
In the present paper, we make no attempt to provide an exhaustive discussion related to convergence results of SA schemes. We refer the interested readers to \cite{Duflo1996}, \cite{Kushner2003} and \cite{ben:met:pri} among others for developments and a more complete overview in SA theory. In the next section, we first recall some basic facts concerning stable convergence (following the notations of \cite{Jac:Prot:98}) and list classical results of SA theory. 

\subsection{Preliminaries}
%
%
%
%

For a sequence of $E$-valued ($E$ being a Polish space) random variables $(X_n)_{n\geq1}$ defined on a probability space $(\Omega, \mathcal{F}, \P)$, we say that $(X_n)_{n\geq1}$ converges in law stably to $X$ defined on an extension $(\tilde{\Omega}, \tilde{\mathcal{F}}, \tilde{\P})$ of $(\Omega, \mathcal{F}, \P)$ and write $X_n\overset{stably }{\Longrightarrow} X$, if for all bounded random variable $U$ defined on $(\Omega, \mathcal{F}, \P)$  and for all $h: E \rightarrow \R$ bounded continuous, one has
$$
\E[U h(X_n)] \rightarrow \tilde{\E}[Uh(X)], \ \ n\rightarrow + \infty.
$$

This convergence is obviously stronger than convergence in law that we denote by ``$\Longrightarrow$''. Stable convergence was introduced in \cite{ren:63} and notably investigated in \cite{ald:eagl:78}. The following lemma is a basic result on stable convergence that will be useful throughout the paper. We refer to \cite{Jac:Prot:98}, Lemma 2.1 for a proof. Here, $E$ and $F$ will denote two Polish spaces. We consider a sequence $(X_n)_{n\geq1}$ of $E$-valued random variable defined on $(\Omega, \F)$.
\begin{LEMME}
\label{lemme:conv:stab}
Let $(Y_n)_{n\geq1}$ be a sequence of $F$-valued random variable defined on $(\Omega, \F)$ satisfying
$$
 Y_n \overset{\P}{\longrightarrow} Y
$$

\noindent where $Y$ is defined on $(\Omega, \F)$. If $X_n \overset{stably}{\Longrightarrow} X$ where $X$ is defined on an extension of $(\Omega, \F)$ then, we have
$$
(X_n,Y_n) \overset{stably }{\Longrightarrow} (X,Y).
$$ 

Let us note that this result remains valid when $Y_n=Y$, for all $n\geq1$
\end{LEMME}

We illustrate this notion by the Euler-Maruyama discretization scheme of a diffusion process $X$ solution of an SDE. The following results will be useful in the sequel in order to illustrate multi-level SA methods. We first introduce some notations, namely for $x\in \R^q$
$$
f(x) = \begin{pmatrix}
b_1(x) & \sigma_{11}(x) & \cdots & \sigma_{1q'}(x) \\
b_2(x) & \sigma_{21}(x) & \cdots & \sigma_{2 q'}(x) \\
\vdots      & \vdots                    & \cdots  & \vdots \\
b_q(x) & \sigma_{q1}(x) & \cdots & \sigma_{q q'}(x)
\end{pmatrix}
$$

\noindent and $dY_t = (dt \  dW^{1}_t \  \cdots \  dW^{q'}_t)^{T}$ where $b:\R^q\rightarrow \R^q$, $\sigma:\R^q\rightarrow \R^q\times \R^{q'}$. Here as below $u^T$ denotes the transpose of the vector $u$. The dynamic of $X$ will be written in the compact form
$$
\forall t \in [0,T], \ X_t = x + \int_0^t f(X_s) dY_s
$$

\noindent with its Euler-Maruyama scheme with time step $\Delta=T/n$ 
\begin{equation*}
X^{n}_t = x+ \int_0^t f(X^{n}_{\phi_n(s)}) dY_s.
\end{equation*}

We introduce the following smoothness assumption on the coefficients:
\begin{trivlist}
\item[\A{HS}] The coefficients $b, \sigma$ are uniformly Lipschitz continuous.
\item[\A{HD}] The coefficients $b, \sigma$ satisfy \A{HS} and are continuously differentiable. 
\end{trivlist}

The following result is due to \cite{Jac:Prot:98}, Theorem 3.2 p.276 and Theorem 5.5, p.293.

\begin{THM}\label{weak:conv:euler}Assume that \A{HD} holds. Then, the process $V^{n}:=X^{n}-X$ satisfies
$$
\sqrt{\frac{n}{T}} V^{n} \overset{stably }{ \Longrightarrow} V, \ \ as \ \ n\rightarrow + \infty
$$

\noindent the process $V$ being defined by $V_0=0$ and
\begin{equation}
dV^{i}_t = \sum_{j=1}^{q'+1} \sum_{k=1}^{q} f^{'ij}_k(X_t) \left[ V^{k}_t dY^{j}_t - \sum_{\ell=1}^{q'+1} f^{k\ell}(X_t) dZ^{\ell j}_t\right]
\label{PROC:U}
\end{equation}

\noindent where $f^{'ij}_k$ is the $k$th partial derivative of $f^{ij}$ and 
\begin{align*}
\forall (i,j) \in \leftB 2, q'+1 \rightB \times \leftB 2, q'+1\rightB, \ Z^{ij}_t & = \frac{1}{\sqrt{2}} \sum_{1 \leq k, \ell \leq q} \int_0^t \sigma^{ik}(X_s) \sigma^{j \ell}(X_s) dB^{k \ell}_s, \\
\forall j \in \leftB1, q'+1 \rightB, \  Z^{1j} & = 0, \\
\forall i \in \leftB1, q'+1 \rightB, Z^{i1}    &  = 0,
\end{align*}

\noindent where $B$ is a standard $(q')^2$-dimensional Brownian motion defined on an extension $(\tilde{\Omega} , \tilde{\F},(\tilde{\F}_t)_{t\ge 0},\tilde{\P})$ of $(\Omega,\F,(\F_t)_{t\ge 0},\P)$ and independent of $W$.
\end{THM}

We will also use the following result which is due to \cite{ben:keb:12}, Theorem 4. 
\begin{THM}\label{weak:conv:euler:2step}Let $m\in \N^{*}\backslash{\left\{1\right\}}$. Assume that \A{HD} holds. Then, we have
$$
\sqrt{\frac{m^{\ell}}{(m-1)T}} (X^{m^{\ell}}-X^{m^{\ell-1}}) \overset{stably }{ \Longrightarrow} V, \ \ as \ \ \ell\rightarrow + \infty.
$$
\end{THM}

We now turn our attention to SA. There are various theorems that guarantee the $a.s.$ and/or $L^{p}$ convergence of SA algorithms. We provide below a general result in order to derive the $a.s.$ convergence of such procedures. It is also known as \emph{Robbins-Monro Theorem} and covers most situations (see the remark below).

\begin{THM}[Robbins-Monro Theorem]\label{RM:THM}
Let $H: \R^d \times \R^q \rightarrow \R^d$ a Borel function and $U$ a $\R^q$-valued random vector with law $\mu$. Define 
$$
\forall \theta \in \R^d, \ h(\theta) = \E[H(\theta,U)],
$$

\noindent and denote by $\theta^{*}$ the (unique) solution to $h(\theta) = 0$. Suppose that $h$ is a continuous function that satisfies the mean-reverting assumption
\begin{equation}
\forall \theta \in \R^d, \theta \neq \theta^{*}, \ \ \langle \theta-\theta^{*} , h(\theta) \rangle>0.
\label{MeanRevert_Ass}
\end{equation}

\noindent Let $\gamma=(\gamma_p)_{p\geq1}$ be a sequence of gain parameters satisfying \eqref{STEP}. Suppose that 
\begin{equation}
\forall \theta \in \R^d, \ \ \E|H(\theta,U)|^2 \leq C(1+ |\theta-\theta^{*}|^2)
\label{Growth_Cond}
\end{equation}

Let $(U_p)_{p\geq1}$ be an i.i.d. sequence of random vectors with common law $\mu$ and $\theta_0$ a random vector independent of $(U_p)_{p\geq 1}$ satisfying $\E |\theta_0|^2< +\infty$. Then, the recursive procedure defined by
\begin{equation}
\theta_{p+1} = \theta_{p} - \gamma_{p+1} H(\theta_p,U_{p+1}), \ p \geq0
\label{SA_Alg}
\end{equation}

\noindent satisfies 
$$
\theta_p \overset{a.s.}{\longrightarrow} \theta^{*},  \ as \ p\rightarrow +\infty.
$$
\end{THM}

Let us point out that the Robbins-Monro theorem also covers the framework of stochastic gradient algorithm. Indeed, if the function $h$ is the gradient of a convex potential $L$, namely $h=\nabla L$ where $L\in \mathcal{C}^{1}(\R^d,\R_+)$, that satisfies: $\nabla L$ is Lipschitz, $|\nabla L|^2 \leq C(1+L)$ and $\lim_{|\theta|\rightarrow + \infty} L(\theta) = +\infty$ then, $\Argmin  L$ is non-empty and according to the following standard lemma $\theta \mapsto \frac{1}{2}|\theta-\theta^{*}|^2$ is a Lyapunov function so that the sequence $(\theta_n)_{n\geq1}$ defined by \eqref{SA_Alg} converges $a.s.$ to $\theta^*$. 

\begin{LEMME} Let $L\in \mathcal{C}^{1}(\R^d,\R_+)$ be a convex function, then
$$
\forall \theta, \theta' \in \R^d,  \ \ \langle \nabla L(\theta) - \nabla L (\theta'), \theta-\theta' \rangle \geq 0.
$$ 

Moreover, if $\Argmin L$ is non-empty, then one has
$$
\forall \theta \in \R^d \backslash \Argmin L, \forall \theta^{*} \in \Argmin L, \ \ \langle \nabla L(\theta), \theta-\theta^{*} \rangle >0.
$$
\end{LEMME}

Now, we provide a result on the weak rate of convergence of SA algorithm. In standard situations, it is well-known that a stochastic algorithm $(\theta_p)_{p\geq1}$ converges to its target at a rate $\gamma^{-1/2}_p$. We also refer to \cite{fri:men:12, fat:fri:13} for some recent developments on non-asymptotic deviation bounds. More precisely, the sequence $(\gamma^{-1/2}_p(\theta_p - \theta^*))_{p\geq1}$ converges in distribution to some normal distribution with a covariance matrix based on $\E[ H(\theta^*,U)H(\theta^*,U)^{T}]$ where $U$ is the noise of the algorithm. The following result is due to \cite{Pell:98} (see also \cite{Duflo1996}, p.161 Theorem 4.III.5) and has the advantage to be local, in the sense that a CLT holds on the set of convergence of the algorithm to an equilibrium which makes possible a straightforward application to multi-target algorithms.

\begin{THM}\label{CLT_SA}Let $\theta^*\in \left\{ h=0 \right\}$. Suppose that $h$ is twice continuously differentiable in a neighborhood of $\theta^*$ and that $Dh(\theta^*)$ is a stable $d\times d$ matrix, $i.e.$ all its eigenvalues have strictly positive real parts. Assume that the function $H$ satisfies the following regularity and growth control property
$$
\theta \mapsto \E\left[ H(\theta,U)H(\theta,U)^{T}\right]  \ \mbox{is continuous on}\  \R^d, \ \ \exists \varepsilon>0 \  s.t. \ \theta \mapsto \E\left[|H(\theta,U)|^{2+\varepsilon}\right] \ \mbox{is locally bounded on} \ \R^d.
$$

Assume that the noise of the algorithm is not degenerated at the equilibrium, that is $\Gamma(\theta^*) :=  \E\left[H(\theta^*,U)H(\theta^*,U)^{T}\right] $ is a positive definite deterministic matrix.

The step sequence of the procedure \eqref{SA_Alg} is given by $\gamma_p = \gamma(p)$, $p\geq1$, where $\gamma$ is a positive function defined on $[0, + \infty[$ decreasing to zero. We assume that $\gamma$ satisfies one of the following assumptions:
\begin{itemize}

\item $\gamma$ varies regularly with exponent $(-a)$, $a \in [0,1)$, that is, for any $x>0$, $\lim_{t\rightarrow +\infty}\gamma(tx)/\gamma(t)=x^{-a}$. In this case, set $\zeta=0$.

\item for $t\geq1$, $\gamma(t)=\gamma_0/t$ and $\gamma_0$ satisfies $2\mathcal{R}e(\lambda_{min}) \gamma_0 > 1$, where $\lambda_{min}$ denotes the eigenvalue of $Dh(\theta^*)$ with the lowest real part. In this case, set $\zeta = 1/(2\gamma_0)$.
\end{itemize} 

Then, on the event $\left\{ \theta_p \rightarrow \theta^*\right\}$, one has
$$
\gamma(p)^{-1/2}\left( \theta_p - \theta^*\right)  \Longrightarrow  \mathcal{N}\left(0, \Sigma^{*}\right) 
$$

\noindent where $\Sigma^*:=\int_0^{\infty} \exp\left(-s(Dh(\theta^*)- \zeta I_d)\right)^{T} \Gamma(\theta^*) \exp\left(-s(Dh(\theta^*)-\zeta I_d)\right) ds$.
\end{THM}

\begin{REM} In SA theory it is also said that $-Dh(\theta^*)$ is a Hurwitz matrix, that is all its eigenvalue has strictly negative real part. 
The assumption on the step sequence $(\gamma_{n})_{n\geq1}$ is quite general and includes polynomial step sequences. In practical situation, the above theorem is often applied to the usual gain $\gamma_p=\gamma(p) = \gamma_0 p^{-a}$, with $1/2 < a \leq 1$, which notably satisfies \eqref{STEP}.

\end{REM}

Hence we clearly see that the optimal weak rate of convergence is achieved by choosing $\gamma_p = \gamma_0 /p$ with $2 \mathcal{R} e(\lambda_{min}) \gamma_0 >1$. However the main drawback with this choice is that the constraint on $\gamma_0$ is difficult to handle in practical implementation. Moreover it is well-known that in this case the asymptotic covariance matrix is not optimal, see e.g. \cite{Duflo1996} or \cite{ben:met:pri} among others.

As mentioned in the introduction, a solution consists in devising the original SA algorithm \eqref{SA_Alg} with a slow decreasing step $\gamma=(\gamma_{p})_{p\geq1}$, where $\gamma$ varies regularly with exponent $(-a)$, $a \in (1/2,1)$ and to simultaneously compute the empirical mean $(\bar{\theta}_{p})_{p\geq1}$ of the sequence $(\theta_{p})_{p\geq0}$ by setting
\begin{align}
\label{RM_AV}
\bar{\theta}_{p} & = \frac{\theta_0 + \theta_1+ \cdots + \theta_p}{p+1} = \bar{\theta}_{p-1} - \frac{1}{p+1}\left(\bar{\theta}_{p-1}-\theta_{p}\right).
\end{align}

The following result states the weak rate of convergence for the sequence $(\bar{\theta}_{p})_{p\geq1}$. In particular, it shows that the optimal weak rate of convergence and the optimal asymptotic covariance matrix can be obtained without any condition on $\gamma_0$.  For a proof, the reader may refer to \cite{Duflo1996}, p.169.

\begin{THM}\label{rip:polyak:thm} Let $\theta^*\in \left\{ h=0 \right\}$. Suppose that $h$ is twice continuously differentiable in a neighborhood of $\theta^*$ and that $Dh(\theta^*)$ is a stable $d\times d$ matrix, $i.e.$ all its eigenvalues have positive real parts. Assume that the function $H$ satisfies the following regularity and growth control property
$$
\theta \mapsto \E\left[H(\theta,U)H(\theta,U)^{T}\right]  \ \mbox{is continuous on}\  \R^d, \ \ \exists b>0 \  s.t. \ \theta \mapsto \E\left[|H(\theta,U)|^{2+b}\right] \ \mbox{is locally bounded on} \ \R^d.
$$

Assume that the noise of the algorithm is not degenerated at the equilibrium, that is $\Gamma(\theta^*) :=  \E\left[H(\theta^*,U)H(\theta^*,U)^{T}\right]$ is a positive definite deterministic matrix.

The step sequence of the procedure \eqref{SA_Alg} is given by $\gamma_p = \gamma(p)$, $p\geq1$, where $\gamma$ varies regularly with exponent $(-a)$, $a \in (1/2,1)$. Then, on the event $\left\{ \theta_p \rightarrow \theta^*\right\}$, one has
$$
\sqrt{p} \left( \bar{\theta}_p - \theta^* \right) \Longrightarrow  \mathcal{N}\left(0,  Dh(\theta^*)^{-1} \Gamma(\theta^*) (Dh(\theta^*)^{-1})^{T} \right).
$$

\end{THM}

\subsection{Main assumptions}
We list here the required assumptions in our framework to derive our asymptotic results and make some remarks. 
\begin{trivlist}
\item[\A{HWR1}] There exists $\rho \in (0,1/2]$, 
$$
n^{\rho}(U^{n}-U) \overset{stably }{ \Longrightarrow} V, \ \ as \ \ n\rightarrow + \infty
$$

\noindent where $V$ is an $\R^{q}$-valued random variable eventually defined on an extension $(\tilde{\Omega} , \tilde{\F}, \tilde{\P})$ of $(\Omega,\F,\P)$.

\item[\A{HWR2}] There exists $\rho \in (0,1/2]$, 
$$
m^{\ell \rho}(U^{m^{\ell}} - U^{m^{\ell-1}}) \overset{stably }{ \Longrightarrow} V^{m} \ \ as \ \ \ell \rightarrow + \infty
$$

\noindent where $V^{m}$ is an $\R^{q}$-valued random variable eventually defined on an extension $(\tilde{\Omega} , \tilde{\F}, \tilde{\P})$ of $(\Omega,\F,\P)$.

\item[\A{HSR}] There exists $\delta>0$,
$$
\sup_{n\geq1}\E\left[|n^{\rho}(U^{n}-U)|^{2+\delta}\right] < +\infty.
$$

\item[\A{HR}] There exists $b\in (0,1]$, 
$$
\sup_{n \in \N^{*}, (\theta, \theta') \in (\R^d)^2} \frac{\E[|H(\theta, U^{n})-H(\theta', U^{n})|^2]}{|\theta-\theta'|^{2b}}< +\infty.
$$

\item[\A{HDH}] For all $\theta \in \R^d$, $\P(U \notin \mathcal{D}_{H,\theta}) =0$ with $\mathcal{D}_{H,\theta}:=\left\{x \in \R^q: x\mapsto H(\theta,x) \mbox{ is differentiable at } x \right\}$.

\item[\A{HLH}] For all $(\theta, \theta', x) \in (\R^d)^2 \times \R^q, \ |H(\theta,x)-H(\theta',x)| \leq C (1+ |x|^{r}) |\theta-\theta'|$, for  some  $C, r>0$.

\item[\A{HI}] There exists $\delta>0$ such that for all $R>0$, we have $\sup_{\left\{\theta: |\theta|\leq R, \ n\in\N^{*}\right\}}\E[|H(\theta,U^{n})|^{2+\delta}]<+\infty$. The sequence $
(\theta \mapsto \E[H(\theta,U^{n}) H(\theta, U^{n})^{T}])_{n\geq1}$ converges locally uniformly towards $\theta \mapsto \E[H(\theta, U) H(\theta, U)^{T}]$. The function $\theta \mapsto \E[H(\theta, U) H(\theta, U)^{T}]$ is continuous and $ \E[H(\theta^{*}, U) H(\theta^{*}, U)^{T}]$ is a positive deterministic matrix.

\item[\A{HMR}] There exists $\underline{\lambda}>0$ such that $\forall n \geq1$
$$
\forall \theta \in \R^d, \ \langle  \theta-\theta^{*,n}, h^{n}(\theta) \rangle  \geq \underline{\lambda} |\theta-\theta^{*,n}|^2.
$$

\end{trivlist}

We will denote $\lambda_{m}$ the lowest real part of the eigenvalues of $Dh(\theta^*)$. We will assume that the step sequence is given by $\gamma_p = \gamma(p)$, $p\geq1$, where $\gamma$ is a positive function defined on $[0, + \infty[$ decreasing to zero and satisfying one of the following assumptions:
\begin{trivlist}
\item[\A{HS1}] $\gamma$ varies regularly with exponent $(-a)$, $a \in [0,1)$, that is, for any $x>0$, $\lim_{t\rightarrow +\infty}\gamma(tx)/\gamma(t)=x^{-a}$.

\item[\A{HS2}] for $t\geq1$, $\gamma(t)=\gamma_0/t$ and $\gamma_0$ satisfies $
2 \underline{\lambda} \gamma_0 > 1$.
 
\end{trivlist}

\begin{REM} Assumption \A{HR} is trivially satisfied when $\theta \mapsto H(\theta,x)$ is Hlder-continuous with modulus having polynomial growth in $x$. However, it is also satisfied when $H$ is less regular. For instance, it holds for $H(\theta,x)=\mbox{\bf{1}}_{\left\{ x \leq \theta \right\}}$ under the additional assumption that $U^{n}$ has a bounded density (uniformly in $n$). 
\end{REM}

\begin{REM}\label{eigenvalue:rem} Assumption \A{HMR} already appears in \cite{Duflo1996} and \cite{ben:met:pri}, see also \cite{fri:men:12} and \cite{fat:fri:13} in another context. It allows to control the $L^2$-norm $\E[|\theta^{n}_p-\theta^{*,n}|^2]$ with respect to the step $\gamma(p)$ uniformly in $n$, see LemmaÊ\ref{sstrongerror:tech:lemme} in Section \ref{technical:res:sec}. As discussed in \cite{Kushner2003}, Chapter 10, Section 5, if one considers the projected version of the algorithm \eqref{RM} on a bounded convex set $D$ (for instance an hyperrectangle $\Pi_{i=1}^{d} [a_i,b_i]$) containing $\theta^{*,n}$, $\forall n\geq1$, as very often happens from a practical point of view, this assumption can be localized on $D$, that is it holds on $D$ instead of $\R^d$. In this case, a sufficient condition is $\inf_{\theta \in D, n\in\N^*} \lambda_{min}((Dh^{n}(\theta)+Dh^{n}(\theta)^{T})/2)>0$, where $\lambda_{min}(A)$ denotes the lowest eigenvalue of the matrix $A$. 

We also want to point out that if it is satisfied then one has $\lambda_m \geq\underline{\lambda}$. Indeed, writing $h^{n}(\theta) =  \int_{0}^{1}Dh^{n}(t \theta + (1-t) \theta^{*,n}) (\theta-\theta^{*,n}) dt$, for all $\theta \in \R^d$, we clearly have
\begin{align*}
\langle  \theta-\theta^{*,n}, h^{n}(\theta) \rangle & = \int_{0}^{1}  \langle \theta - \theta^{*,n}, \frac{Dh^{n}(t \theta + (1-t) \theta^{*,n}) + Dh^{n}(t \theta + (1-t) \theta^{*,n})^{T}}{2} (\theta-\theta^{*,n}) \rangle dt \\
& \geq \underline{\lambda} |\theta-\theta^{*,n}|^2.
\end{align*}
Using the local uniform convergence of $(Dh^{n})_{n\geq1}$ and the convergence of $(\theta^{*,n})_{n\geq1}$ toward $\theta^*$, by passing to the limit $n\rightarrow + \infty$ in the above inequality, we obtain
$$
\forall \theta \in K, \ \int_{0}^{1} \langle \theta - \theta^{*}, \frac{Dh(t \theta + (1-t) \theta^{*}) + Dh(t \theta + (1-t) \theta^{*})^{T}}{2} (\theta-\theta^{*}) \rangle dt  \geq \underline{\lambda} |\theta-\theta^{*}|^2
$$ 

\noindent where $K$ is a compact set such that $\theta^*+ u_{m} \in K$, $u_m$ being the eigenvector associated to the eigenvalue of $Dh(\theta^*)$ with the lowest real part. Hence, selecting $\theta = \theta^*+ \varepsilon u_{m} $ in the previous inequality and passing to the limit $\varepsilon \rightarrow 0$, we get $\lambda_m \geq \underline{\lambda}$.
\end{REM}

\begin{REM}Assumptions \A{HWR1}, \A{HWR2} and $\A{HSR}$ allow to establish a CLT for the multi-level SA estimators presented in sections \ref{pres:mlv:sec} and \ref{ml:sec}. They include the case of the value at time $T$ of an SDE, namely $U=X_T$ approximated by its continuous Euler-Maruyama scheme $U^{n}=X^{n}_T$ with $n$ steps. Under \A{HD} one has $\rho=1/2$. Moreover, $U$ may depend on the whole path of an SDE. For instance, one may have $U=L_T$ the local time at level $0$ of a one-dimensional continuous and adapted diffusion process and the approximations may be given by
$$
U^{n}= \sum_{i=1}^{[nt]} f\left(u_n X_{\frac{i-1}{n}},\sqrt{n}\left( X_{\frac{i}{n}} - X_{\frac{i-1}{n}} \right)\right).
$$

Then under some assumptions on the function $f$ and the coefficients $b, \ \sigma$, the weak and strong rate of convergence is $\rho=1/4$, see \cite{Jacod98} for more details. Let us note that we do not know what happens when $\rho>1/2$ which includes the case of higher order schemes for discretization schemes of SDE.

\end{REM}

\subsection{On the implicit weak error}\label{implicit:discret:sec}

As already observed the approximation of $\theta^*$ solution of $h(\theta)=\E[H(\theta,U)]=0$ is affected by two errors: the \emph{implicit discretization error} and the \emph{statistical error}. Our first results concern the convergence of $\theta^{*,n}$ toward $\theta^*$ and its convergence rate as $n\rightarrow + \infty$.

\begin{THM}
\label{thm:conv:disc}
For all $n \in \N^*$, assume that $h$ and $h^n$ satisfy the mean reverting assumption \eqref{MeanRevert_Ass} of Theorem \ref{RM:THM}. Moreover, suppose that $(h^{n})_{n\geq1}$ converges locally uniformly towards $h$. Then, one has
$$
\theta^{*,n} \rightarrow \theta^* \ \ \mbox{as} \ \ n\rightarrow + \infty.
$$

Moreover, suppose that $h$ and $h^n$, $n\geq1$, are continuously differentiable and that $Dh(\theta^*)$ is non-singular. Assume that $(Dh^{n})_{n\geq1}$ converges locally uniformly to $Dh$. If there exists $\alpha \in \R^{*}$ such that
$$
\forall \theta \in \R^d, \  \lim_{n\rightarrow + \infty} n^{\alpha} (h^n(\theta) - h(\theta)) = \mathcal{E}(h,\alpha, \theta),
$$

\noindent then, one has
$$
 \lim_{n\rightarrow + \infty} n^{\alpha} (\theta^{*,n} - \theta^{*}) = -Dh^{-1}(\theta^*) \mathcal{E}(h,\alpha, \theta^*).
$$

\end{THM}

\subsection{On the optimal tradeoff between the implicit error and the statistical error}\label{opt:tradeoff:sec}
Given the order of the implicit weak error, a natural question is to find the optimal balance between the value of $n$ in the approximation of $U$ and the number $M$ of steps in \eqref{RM} for the computation of $\theta^{*,n}$ in order to achieve a given global error $\varepsilon$. 

\begin{THM}\label{globalTCL}Suppose that the assumptions of Theorem \ref{thm:conv:disc} are satisfied and that $h$ satisfies the assumptions of Theorem \ref{CLT_SA}. Assume that \A{HR}, \A{HI} and \A{HMR} hold and that $h^{n}$ is twice continuously differentiable with $Dh^{n}$ Lipschitz continuous uniformly in $n$. 
If \A{HS1} or \A{HS2} is satisfied then one has
$$
n^{\alpha}\left( \theta^{n}_{\gamma^{-1}(1/n^{2\alpha})} - \theta^* \right) \Longrightarrow  -Dh^{-1}(\theta^*) \mathcal{E}(h,\alpha, \theta^*) + \mathcal{N}\left(0, \Sigma^{*}\right),
$$  

\noindent where 
\begin{equation}
\label{cov:matrix:opt:alloc}
\Sigma^{*}:=\int_0^{\infty} \exp\left(-s(Dh(\theta^*)- \zeta I_d)\right)^{T} \E[H(\theta^*,U)H(\theta^*,U)^{T}] \exp\left(-s(Dh(\theta^*)-\zeta I_d)\right) ds
\end{equation}

\noindent with $\zeta=0$ if \A{HS1} holds and $\zeta=1/2\gamma_0$ if \A{HS2} holds.

\end{THM}
\begin{LEMME}\label{lem:conv:opt:tradeoff} Let $\delta>0$. Under the assumptions of Theorem \ref{globalTCL}, one has
$$
n^{\alpha} \left(  \theta^{n^{\delta}}_{\gamma^{-1}(1/n^{2\alpha})} - \theta^{*,n^{\delta}} \right)  \Longrightarrow \mathcal{N}(0,\Sigma^*), \ \ n\rightarrow + \infty.
$$

\end{LEMME}

\begin{proof}[Proof of Theorem \ref{globalTCL}] We decompose the error as follows:
$$
\theta^{n}_{\gamma^{-1}(1/n^{2\alpha})} - \theta^* = \theta^{n}_{\gamma^{-1}(1/n^{2\alpha})} - \theta^{*,n}  + \theta^{*,n} - \theta^{*}
$$

\noindent and analyze each term of the above sum. By Lemma \ref{lem:conv:opt:tradeoff}, we have
$$
n^{\alpha} \left( \theta^{n}_{\gamma^{-1}(1/n^{2\alpha})} - \theta^{*,n} \right)  \Longrightarrow  \mathcal{N}\left(0, \Sigma^{*}\right) 
$$

\noindent and using Theorem \ref{thm:conv:disc}, we also obtain
$$
n^{\alpha} (\theta^{*,n} - \theta^{*}) \rightarrow -Dh^{-1}(\theta^*) \mathcal{E}(h,\alpha, \theta^*).
$$


\end{proof}

The result of Theorem \ref{globalTCL} could be construed as follows. For a total error of order $1/n^{\alpha}$, it is necessary to achieve at least $M=\gamma^{-1}(1/n^{2\alpha})$ steps of the SA scheme defined by \eqref{RM}. Hence, in this case the complexity (or computational cost) of the algorithm is given by
\begin{align}
C_{SA}(\gamma) = C \times n \times \gamma^{-1}(1/n^{2\alpha}), 
\label{Compl:SA}
\end{align}

\noindent where $C$ is some positive constant. We now investigate the impact of the step sequence $(\gamma_n)_{n\geq1}$ on the complexity by considering the two following basic step sequences:
\begin{itemize}

\item if we choose $\gamma(p) = \gamma_0/p$ with $2 \underline{\lambda} \gamma_0 >1$, then $C_{SA} = C \times n^{2\alpha + 1 }$.

\item if we choose $\gamma(p) = \gamma_0/p^{\rho}$, $\frac12 < \rho <1$ then $C_{SA} = C \times n^{2\alpha/\rho +1}$.

\end{itemize}

Hence we clearly see that the minimal complexity is achieved by choosing $\gamma_p = \gamma_0 /p$ with $2 \underline{\lambda} \gamma_0 >1$. In this latter case, we see that the computational cost is similar to the one achieved by the classical Monte Carlo algorithm for the computation of $\E_x[f(X_T)]$. However the main drawback with this choice of step sequence comes from the constraint on $\gamma_0$. Next result shows that the optimal complexity can be reached for free through the smoothing of the procedure \eqref{RM} according to the Ruppert \& Polyak averaging principle.

\begin{THM}\label{globalTCLRupPol} Suppose that the assumptions of Theorem \ref{thm:conv:disc} are satisfied and that $h$ satisfies the assumptions of Theorem \ref{CLT_SA}. Assume that \A{HR}, \A{HI} and \A{HMR} hold and that $h^{n}$ is twice continuously differentiable with $Dh^{n}$ Lipschitz continuous uniformly in $n$. Define the empirical mean sequence $(\bar{\theta}^{n}_p)_{p\geq1}$ of the sequence $(\theta^{n}_p)_{p\geq1}$ by setting
$$
\bar{\theta}^{n}_{p}  = \frac{\theta_0 + \theta^{n}_1+ \cdots + \theta^{n}_p}{p+1} = \bar{\theta}^{n}_{p-1} - \frac{1}{p+1}\left(\bar{\theta}^{n}_{p-1}-\theta^{n}_{p}\right),
$$

\noindent where the step sequence $\gamma=(\gamma_p)_{p\geq1}$ satisfies \A{HS1} with $a \in (1/2,1)$.
Then, one has
$$
n^{\alpha}\left(\bar{\theta}^{n}_{n^{2\alpha}} - \theta^* \right) \Longrightarrow  -Dh^{-1}(\theta^*) \mathcal{E}(h,\alpha, \theta^*) + \mathcal{N}\left(0, Dh(\theta^*)^{-1} \E[H(\theta^*,U)H(\theta^*,U)^{T}] (Dh(\theta^*)^{-1})^{T} \right),
$$

\end{THM}

  
\begin{LEMME}\label{lem:conv:mean} 
Let $\delta>0$. Under the assumptions of Theorem \ref{globalTCLRupPol}, one has
$$
n^{\alpha} \left( \bar{\theta}^{n^{\delta}}_{n^{2\alpha}} - \theta^{*,n^{\delta}} \right)   \Longrightarrow   \mathcal{N}\left(0, Dh(\theta^*)^{-1} \E[H(\theta^*,U)H(\theta^*,U)^{T}] (Dh(\theta^*)^{-1})^{T}\right), \ \ n\rightarrow + \infty.
$$

\end{LEMME}

\begin{proof}[ Proof of Theorem \ref{globalTCLRupPol}] Similarly to the proof of Theorem \ref{globalTCL} we decompose the error as follows:
$$
\bar{\theta}^{n}_{n^{2\alpha}} - \theta^* = \bar{\theta}^{n}_{n^{2\alpha}} - \theta^{*,n}  + \theta^{*,n} - \theta^*.
$$

Applying successively Theorem \ref{thm:conv:disc} and Lemma \ref{lem:conv:mean}, we obtain
$$
n^{\alpha} \left(\bar{\theta}^{n}_{n^{2\alpha}} - \theta^*  \right) \Longrightarrow  -Dh^{-1}(\theta^*) \mathcal{E}(h,\alpha, \theta^*) + \mathcal{N}\left(0, \Sigma^{*}\right).
$$

\end{proof}

The result of Theorem \ref{globalTCLRupPol} shows that for a total error of order $1/n^{\alpha}$, it is necessary to achieve at least $M=n^{2\alpha}$ steps of the SA scheme defined by \eqref{RM} with step sequence satisfying \A{HS1} and to simultaneously compute its empirical mean, which represents a negligible part of the total cost. As a consequence, we see that in this case the complexity of the algorithm is given by
\begin{align*}
C_{\text{SA-RP}}(\gamma) = C \times n^{2\alpha+1}.
\end{align*}

Therefore, the optimal complexity is reached for free without any condition on $\gamma_0$ thanks to the Ruppert \& Polyak averaging principle.


\subsection{The statistical Romberg stochastic approximation method}\label{pres:mlv:sec}
 
In this section we present a two-level SA scheme that will be also referred as the statistical Romberg SA method which allows to minimize the complexity of the SA algorithm $(\theta^{n}_{p})_{p \in \leftB 0, \gamma^{-1}(1/n^{2\alpha})\rightB}$ for the numerical computation of $\theta^{*}$ solution to $h(\theta)=\E[H(\theta,U)]=0$. It is clearly apparent that
$$
\theta^{*,n} = \theta^{*,n^{\beta}} + \theta^{*,n} - \theta^{*,n^{\beta}}, \ \beta \in (0,1).
$$

The statistical Romberg SA scheme independently estimates each of the solutions appearing on the right-hand side in a way that minimizes the computational complexity. Let $\theta^{n^{\beta}}_{M_1}$ be an estimator of $\theta^{*,n^{\beta}}$ using $M_1$ independent samples of $U^{n^{\beta}}$ and $\theta^{n}_{M_2} - \theta^{n^{\beta}}_{M_2}$ be an estimator of $\theta^{*,n}-\theta^{*,n^{\beta}}$ using $M_2$ independent copies of $(U^{n^{\beta}},U^{n})$. Using the above decomposition, we estimate $\theta^{*}$ by the quantity
$$
\Theta^{sr}_{n} = \theta^{n^{\beta}}_{M_1} + \theta^{n}_{M_2} - \theta^{n^{\beta}}_{M_2}.
$$

It is important to point out here that the couple $( \theta^{n}_{M_2}, \theta^{n^{\beta}}_{M_2})$ is computed using i.i.d. copies of $(U^{n^{\beta}},U^{n})$, the random variables $U^{n^{\beta}}$ and $U^{n}$ being perfectly correlated. Moreover, the random variables used to obtain $\theta^{n^{\beta}}_{M_1}$ are independent to those used for the computation of $( \theta^{n}_{M_2}, \theta^{n^{\beta}}_{M_2})$. 

We also establish a central limit theorem for the statistical Romberg based empirical sequence according to the Ruppert \& Polyak averaging principle. It consists in estimating $\theta^{*}$ by
$$
\bar{\Theta}^{sr}_n = \bar{\theta}^{n^{\beta}}_{M_3} + \bar{\theta}^{n}_{M_4} - \bar{\theta}^{n^{\beta}}_{M_4},
$$

\noindent where $(\bar{\theta}^{n^{\beta}}_p)_{p \in \leftB 0, M_3 \rightB}$ and $(\bar{\theta}^{n}_{p},\bar{\theta}^{n^{\beta}}_{p})_{p \in \leftB 0, M_4 \rightB}$ are respectively the empirical means of the sequences $(\theta^{n^{\beta}}_p)_{p \in \leftB 0,M_3 \rightB}$ and $(\theta^{n}_{p},\theta^{n^{\beta}}_{p})_{p \in \leftB 0,M_4 \rightB}$ devised with the same slow decreasing step, that is a step sequence $(\gamma(p))_{p\geq1}$ where $\gamma$ varies regularly with exponent $(-a)$, $a \in (1/2,1)$.


\begin{THM}\label{TWO:LVL:SA} Suppose that $h$ and $h^n$ satisfy the assumptions of Theorem \ref{thm:conv:disc} with $\alpha \in (\rho \vee 2\rho\beta,1]$ and that $h$ satisfies the assumptions of Theorem \ref{CLT_SA}. Assume that \A{HWR1}, \A{HSR}, \A{HD}, \A{HMR}, \A{HDH} and \A{HLH} hold and that $h^{n}$ are twice continuously differentiable in a neighborhood of $\theta^*$, with $Dh^{n}$ Lipschitz-continuous uniformly in $n$ satisfying: 
$$
\forall \theta \in \R^d, \ n^{\rho} \|Dh^{n}(\theta) -Dh(\theta)\| \rightarrow 0, \ \ \mbox{ as } n\rightarrow + \infty.
$$

  Suppose that $\tilde{\E} \left[(D_xH(\theta^*, U) V) (D_xH(\theta^*, U) V)^T\right]$ is a positive definite matrix. Assume that the step sequence is given by $\gamma_p = \gamma(p)$, $p\geq1$, where $\gamma$ is a positive function defined on $[0, + \infty[$ decreasing to zero, satisfying one of the following assumptions:
\begin{itemize}
\item $\gamma$ varies regularly with exponent $(-a)$, $a \in (1/2,1)$, that is, for any $x>0$, $\lim_{t\rightarrow +\infty}\gamma(tx)/\gamma(t)=x^{-a}$.

\item for $t\geq1$, $\gamma(t)=\gamma_0/t$ and $\gamma_0$ satisfies $
 \underline{\lambda} \gamma_0 > \alpha/(2\alpha - 2\rho \beta)$. 
 
\end{itemize}

Then, for $M_1=\gamma^{-1}(1/n^{2\alpha})$ and $M_2 = \gamma^{-1}(1/(n^{2\alpha-2\rho\beta}))$, one has
$$
n^{\alpha}(\Theta^{sr}_n - \theta^{*})  \Longrightarrow  Dh^{-1}(\theta^*) \mathcal{E}(h,\alpha, \theta^*) + \mathcal{N}(0,\Sigma^{*}), \ \ n\rightarrow +\infty
$$

\noindent with 
\begin{align*}
\Sigma^*& := \int_0^{\infty} \left(e^{-s(Dh(\theta^*)- \zeta I_d)}\right)^{T} ( \E\left[H(\theta^*,U)H(\theta^*,U)^{T}\right]  \\
 &   +   \tilde{\E}\left[\left(D_xH(\theta^{*}, U) V-\tilde{\E}[D_xH(\theta^{*}, U) V]\right) \left(D_xH(\theta^{*}, U) V-\tilde{\E}[D_xH(\theta^{*}, U) V]\right)^{T}\right] ) e^{-s(Dh(\theta^*)- \zeta I_d)} ds \label{asympt:cov:romberg}
\end{align*}

\end{THM}

\begin{LEMME}\label{conv:lem:2step:sr} Let $(\theta_p)_{p\geq0}$ be the procedure defined for $p\geq0$ by 
\begin{equation} \label{alg:sto:diff}
\theta_{p+1} = \theta_p - \gamma_{p+1} H(\theta_p, (U)^{p+1})
\end{equation}

\noindent where $((U^{n})^p, (U)^p)_{p\geq1}$ is an i.i.d sequence of random variables with the same law as $(U^{n},U)$, $(\gamma_p)_{p\geq1}$ is the step sequence of the procedure $(\theta^{n^{\beta}}_p)_{p\geq0}$ and $(\theta^{n}_p)_{p\geq0}$ and $\theta_0$ is independent of the innovation satisfying $\E|\theta_0|^2 < + \infty$. Under the assumptions of Theorem \ref{TWO:LVL:SA}, one has
$$
n^{\alpha} \left(\theta^{n^{\beta}}_{\gamma^{-1}(1/(n^{2\alpha-\beta}))} - \theta_{\gamma^{-1}(1/(n^{2\alpha-\beta}))} - (\theta^{*,n^{\beta}}-\theta^{*}) \right) \Longrightarrow \mathcal{N}(0, \Theta^{*}), \ \ n\rightarrow +\infty,
$$

\noindent with 
$$
\Theta^* := \int_0^{\infty} \left(e^{-s(Dh(\theta^*)- \zeta I_d)}\right)^{T}  \tilde{\E}\left[\left(D_xH(\theta^{*}, U) V-\tilde{\E}[D_xH(\theta^{*}, U) V]\right) \left(D_xH(\theta^{*}, U) V-\tilde{\E}[D_xH(\theta^{*}, U) V]\right)^{T}\right] e^{-s(Dh(\theta^*)- \zeta I_d)} ds
$$

\noindent  and
$$
n^{\alpha} \left(\theta^{n}_{\gamma^{-1}(1/(n^{2\alpha-\beta}))} - \theta_{\gamma^{-1}(1/(n^{2\alpha-\beta}))} - (\theta^{*,n}-\theta^{*}) \right) \overset{\P}{\longrightarrow} 0, \ \ n\rightarrow +\infty.
$$
  
\end{LEMME}
\begin{proof}[Proof of Theorem \ref{TWO:LVL:SA}]
We first write the following decomposition
$$
\Theta^{sr}_n - \theta^{*} =  \theta^{n^{\beta}}_{\gamma^{-1}(1/n^{2\alpha})} - \theta^{*,n^{\beta}} + \theta^{n}_{\gamma^{-1}(1/n^{2\alpha-2\rho\beta})} - \theta^{n^{\beta}}_{\gamma^{-1}(1/n^{2\alpha-2\rho\beta})} - (\theta^{*,n} - \theta^{*,n^{\beta}}) + \theta^{*,n}-\theta^{*}
$$

For the last term of the above sum, we use Theorem \ref{thm:conv:disc} to directly deduce
$$
n^{\alpha} (\theta^{*,n} - \theta^{*}) \rightarrow -Dh^{-1}(\theta^*) \mathcal{E}(h,\alpha, \theta^*), \ as \ n\rightarrow +\infty.
$$

For the first term, from Lemma \ref{lem:conv:opt:tradeoff} it follows
$$
n^{\alpha} (\theta^{n^{\beta}}_{\gamma^{-1}(1/n^{2\alpha})} - \theta^{*,n^{\beta}})  \Longrightarrow \mathcal{N}(0, \Gamma^{*}), 
$$ 

\noindent with $\Gamma^*:= \int_0^{\infty} \exp\left(-s(Dh(\theta^*)- \zeta I_d)\right)^{T} \E[H(\theta^*,U)H(\theta^*,U)^{T}] \exp\left(-s(Dh(\theta^*)-\zeta I_d)\right) ds$. We decompose the last remaining term, namely $ \theta^{n}_{\gamma^{-1}(1/n^{2\alpha-2\rho\beta})} - \theta^{n^{\beta}}_{\gamma^{-1}(1/n^{2\alpha-2\rho\beta})} - (\theta^{*,n} - \theta^{*,n^{\beta}})$ as follows
\begin{align*}
\theta^{n}_{\gamma^{-1}(1/n^{2\alpha-2\rho\beta})} - \theta^{n^{\beta}}_{\gamma^{-1}(1/n^{2\alpha-2\rho\beta})} - (\theta^{*,n} - \theta^{*,n^{\beta}}) & = \theta^{n}_{\gamma^{-1}(1/n^{2\alpha-2\rho\beta})} - \theta_{\gamma^{-1}(1/n^{2\alpha-2\rho\beta})} - (\theta^{*,n} - \theta^*)  \\
& - (\theta^{n^{\beta}}_{\gamma^{-1}(1/n^{2\alpha-2\rho\beta})} - \theta_{\gamma^{-1}(1/n^{2\alpha-2\rho\beta})} - (\theta^{*,n^{\beta}} - \theta^*) )
\end{align*}

\noindent and use Lemma \ref{conv:lem:2step:sr} to conclude the proof.
\end{proof}

\begin{THM}\label{TWO:LVL:SA:AV} Suppose that $h$ and $h^n$ satisfy the assumptions of Theorem \ref{thm:conv:disc} (with $\alpha \in (\rho\vee 2\rho\beta,1]$) and that $h$ satisfies the assumptions of Theorem \ref{CLT_SA}. Assume that the step sequence $\gamma=(\gamma_p)_{p\geq1}$ satisfies \A{HS1} with $a \in (1/2,1)$ and $a> \frac{\alpha}{2\alpha-2\rho\beta} \vee \frac{\alpha(1-\beta)}{(\alpha-\rho\beta)}$. Suppose that \A{HWR1}, \A{HSR}, \A{HD}, \A{HMR}, \A{HDH} and \A{HLH} hold and that $h^{n}$ is twice continuously differentiable in a neighborhood of $\theta^*$, with $Dh^{n}$ Lipschitz-continuous uniformly in $n$ satisfying: 
\begin{equation}
\label{weak:conv:jacob}
\forall \theta \in \R^d, \ n^{\alpha - (\alpha-\rho\beta)a} \|Dh(\theta)-Dh^{n^{\beta}}(\theta)\| \rightarrow 0, \ \ \mbox{ as } n\rightarrow + \infty.
\end{equation}

Suppose that $\tilde{\E}\left[(D_xH(\theta^*, U) V-\tilde{\E}[D_xH(\theta^{*}, U) V]) (D_xH(\theta^*, U) V-\tilde{\E}[D_xH(\theta^{*}, U) V])^T\right]$ is a positive definite matrix. Then, for $M_3=n^{2\alpha}$ and $M_4=n^{2\alpha-2\rho\beta}$, one has
$$
n^{\alpha} (\bar{\Theta}^{sr}_n - \theta^*) \Longrightarrow Dh^{-1}(\theta^*) \mathcal{E}(h,\alpha, \theta^*) + \mathcal{N}(0, \bar{\Sigma}^{*}), \ \ n\rightarrow + \infty,
$$

\noindent where 
\begin{align*}
\bar{\Sigma}^* & := Dh(\theta^*)^{-1} (\E\left[H(\theta^*,U)H(\theta^*,U)^{T}\right]  \\
& +  \tilde{\E}\left[ \left(D_xH(\theta^{*}, U) V-\tilde{\E}[D_xH(\theta^{*}, U) V]\right) \left(D_xH(\theta^{*}, U) V-\tilde{\E}[D_xH(\theta^{*}, U) V]\right)^{T}\right] ) (Dh(\theta^*)^{-1})^{T}.
\end{align*}

  
\begin{LEMME}\label{lem:conv:aver} Let $(\bar{\theta}_p)_{p\geq1}$ be the empirical mean sequence associated to $(\theta_p)_{p\geq1}$ defined by \eqref{alg:sto:diff}.Under the assumptions of Theorem \ref{TWO:LVL:SA:AV}, one has
$$
n^{\alpha} \left( \bar{\theta}^{n^{\beta}}_{n^{2\alpha-2\rho\beta}} - \bar{\theta}_{n^{2\alpha-2\rho\beta}} - (\theta^{*,n^{\beta}}-\theta^{*}) \right) \Longrightarrow  \mathcal{N}(0, \bar{\Theta}^*)
$$

\noindent with $\bar{\Theta}^*=Dh(\theta^*)^{-1}   \tilde{\E}\left[ \left(D_xH(\theta^{*}, U) V-\tilde{\E}[D_xH(\theta^{*}, U) V]\right) \left(D_xH(\theta^{*}, U) V-\tilde{\E}[D_xH(\theta^{*}, U) V]\right)^{T}\right] (Dh(\theta^*)^{-1})^{T}$ and
$$
 n^{\alpha} \left( \bar{\theta}^{n}_{n^{2\alpha-2\rho\beta}} - \bar{\theta}_{n^{2\alpha-2\rho\beta}} - (\theta^{*,n}-\theta^{*}) \right) \overset{\P}{\longrightarrow} 0.
$$

\end{LEMME}

\begin{proof}[Proof of Theorem \ref{TWO:LVL:SA:AV}] We decompose the error as follows
$$
\bar{\Theta}^{sr}_n - \theta^* = \bar{\theta}^{n^{\beta}}_{n^{2\alpha}} - \theta^{*,n^{\beta}} + \bar{\theta}^{n}_{n^{2\alpha-2\rho\beta}} - \bar{\theta}^{n^{\beta}}_{n^{2\alpha-2\rho\beta}} - (\theta^{*,n}-\theta^{*,n^{\beta}}) + \theta^{*,n} - \theta^*.
$$

For the first term, from Lemma \ref{lem:conv:mean} it follows that
$$
n^{\alpha} (\bar{\theta}^{n^{\beta}}_{n^{2\alpha}} - \theta^{*,n^{\beta}}) \Longrightarrow \mathcal{N}(0, Dh(\theta^*)^{-1} \E\left[H(\theta^*,U)H(\theta^*,U)^{T}\right] (Dh(\theta^*)^{-1})^{T}).
$$

For the last term using Theorem \ref{thm:conv:disc}, we have $n^{\alpha}(\theta^{*,n} - \theta^*) \rightarrow  -Dh^{-1}(\theta^*) \mathcal{E}(h,\alpha, \theta^*)$. We now focus on the last remaining term, namely $\bar{\theta}^{n}_{n^{2\alpha-2\rho\beta}} - \bar{\theta}^{n^{\beta}}_{n^{2\alpha-2\rho\beta}} - (\theta^{*,n}-\theta^{*,n^{\beta}})$. We decompose it as follows
\begin{align*}
\bar{\theta}^{n}_{n^{2\alpha-2\rho\beta}} - \bar{\theta}^{n^{\beta}}_{n^{2\alpha-2\rho\beta}} - (\theta^{*,n}-\theta^{*,n^{\beta}}) & =  \bar{\theta}^{n}_{n^{2\alpha-2\rho\beta}} - \bar{\theta}_{n^{2\alpha-2\rho\beta}} - (\theta^{*,n}-\theta^{*})  -  (\bar{\theta}^{n^{\beta}}_{n^{2\alpha-2\rho\beta}} - \bar{\theta}_{n^{2\alpha-2\rho\beta}} - (\theta^{*,n^{\beta}}-\theta^{*}))
\end{align*}

\noindent where $(\bar{\theta}_p)_{p\geq 1}$ is the empirical mean sequence associated to $(\theta_p)_{p\geq1}$ and use Lemma \ref{lem:conv:aver} to conclude the proof.
\end{proof}

\end{THM}
\subsection{The multi-level stochastic approximation method}\label{ml:sec}

As mentioned in the introduction the multi-level SA method uses $L+1$ stochastic schemes with a sequence of bias parameter $(m^{\ell})_{\ell\in \leftB0,L\rightB}$, for a fixed integer $m\geq2$, that satisfies $m^{L}=n$ and estimates $\theta^{*}$ by computing the quantity
$$
\Theta^{ml}_{n} = \theta^{1}_{M_0} + \sum_{\ell=1}^{L} \left(\theta^{m^{\ell}}_{M_\ell} - \theta^{m^{\ell -1}}_{M_\ell}\right).
$$

It is important to point out here that for each level $\ell$ the couple $(\theta^{m^{\ell}}_{M_\ell}, \theta^{m^{\ell -1}}_{M_\ell})$ is computed using i.i.d. copies of $(U^{m^{\ell-1}},U^{m^{\ell}})$. Moreover the random variables $U^{m^{\ell-1}}$ and $U^{m^{\ell}}$ use two different bias parameter but are perfectly correlated. Moreover, for two different levels, the SA schemes are based on independent samples.

\begin{THM}\label{MLVL:SA} Suppose that $h$ and $h^{m^{\ell}}$, $\ell=0, \cdots, L$, satisfy the assumptions of Theorem \ref{thm:conv:disc}. Assume that \A{HWR2}, \A{HSR}, \A{HD}, \A{HMR}, \A{HDH} and \A{HLH} hold and that $h^{n}$ is twice continuously differentiable in a neighborhood of $\theta^*$, with $Dh^{n}$ Lipschitz-continuous uniformly in $n$. Suppose that $\tilde{\E}[ (D_xH(\theta^*, U) V-\tilde{\E}[D_xH(\theta^*, U) V]) (D_xH(\theta^*,U) V-\tilde{\E}[D_xH(\theta^*, U) V])^T]$ is a positive definite matrix. Assume that the step sequence is given by $\gamma_p = \gamma(p)$, $p\geq1$, where $\gamma$ is a positive function defined on $[0, + \infty[$ decreasing to zero, satisfying one of the following assumptions:
\begin{itemize}
\item $\gamma$ varies regularly with exponent $(-a)$, $a \in (1/2,1)$, that is, for any $x>0$, $\lim_{t\rightarrow +\infty}\gamma(tx)/\gamma(t)=x^{-a}$.

\item for $t\geq1$, $\gamma(t)=\gamma_0/t$ and $\gamma_0$ satisfies $
 \underline{\lambda} \gamma_0 >  1$.
 
\end{itemize}

Suppose that $\rho$ satisfies one of the following assumptions: 
\begin{itemize}
\item if $\rho\in (0,1/2)$, then assume that $\alpha>2\rho$, $\underline{\lambda} \gamma_0 > \alpha/(\alpha-2\rho)$ (if $\gamma(t)=\gamma_0/t$) and
$$
\exists \beta>\rho, \ \ \forall \theta \in \R^d, \  \sup_{n\geq1} n^{\beta} \|Dh^{n}(\theta) -Dh(\theta)\| < + \infty.
$$

\noindent In this case we set $M_0=\gamma^{-1}(1/n^{2\alpha})$ and $M_l=\gamma^{-1}(m^{\ell\frac{(1+2\rho)}{2}}(m^{\frac{1-2\rho}{2}}-1)/(n^{2\alpha} (n^{\frac{(1-2\rho)}{2}}-1)) )$, $\ell=1, \cdots, L$.

\item if $\rho =1/2$, then assume that $\alpha=1$, $\theta^{m^{\ell}}_0=\theta_0$, $\ell=1,\cdots,L$, with $\E[|\theta_0|^2]< +\infty$ and 
$$
\exists \beta>1/2, \ \ \forall \theta \in \R^d, \ \sup_{n\geq1}n^{\beta} \|Dh^{n}(\theta) -Dh(\theta)\| < + \infty.
$$

\noindent In this case we set $M_0=\gamma^{-1}(1/n^{2})$ and $M_l=\gamma^{-1}(m^{\ell}\log(m)/(n^{2}\log(n)(m-1) ) )$, $\ell=1, \cdots, L$.
\end{itemize}

\noindent Then one has

$$
n^{\alpha} (\Theta^{ml}_n - \theta^{*})  \Longrightarrow  -Dh^{-1}(\theta^*) \mathcal{E}(h,1, \theta^*) + \mathcal{N}(0,\Sigma^{*}), \ \ n\rightarrow +\infty
$$

\noindent with 
\begin{align*}
\Sigma^* & := \int_0^{\infty} \left(e^{-s(Dh(\theta^*)- \zeta I_d)}\right)^{T} ( \E\left[H(\theta^*,U^{1})H(\theta^*,U^{1})^{T}\right]  \\
& +  \tilde{\E}\left[ \left(D_xH(\theta^{*}, U) V-\tilde{\E}[D_xH(\theta^*, U) V]\right) \left(D_xH(\theta^{*}, U) V-\tilde{\E}[D_xH(\theta^*, U) V]\right)^{T}\right] ) e^{-s(Dh(\theta^*)- \zeta I_d)} ds 
\end{align*}

\end{THM}

\begin{proof}
We first write the following decomposition
$$
\Theta^{ml}_n - \theta^{*} =  \theta^{1}_{\gamma^{-1}(1/n^{2})} - \theta^{*,1} + \sum_{\ell=1}^{L} \left( \theta^{m^{\ell}}_{M_\ell} - \theta^{m^{\ell-1}}_{M_{\ell}} - (\theta^{*,m^{\ell}} - \theta^{*,m^{\ell-1}}) \right) + \theta^{*,n}-\theta^{*}
$$

For the last term of the above sum, we use Theorem \ref{thm:conv:disc} to directly deduce
$$
n^{\alpha} (\theta^{*,n} - \theta^{*}) \rightarrow -Dh^{-1}(\theta^*) \mathcal{E}(h,1, \theta^*), \ as \ n\rightarrow +\infty.
$$

For the first term, the standard CLT (theorem \ref{CLT_SA}) for stochastic approximation leads to
$$
n^{\alpha} (\theta^{1}_{\gamma^{-1}(1/n^{2\alpha})} - \theta^{*,1})  \Longrightarrow \mathcal{N}(0, \Gamma^{*}), 
$$ 

\noindent with $\Gamma^*:= \int_0^{\infty} \exp\left(-s(Dh(\theta^*)- \zeta I_d)\right)^{T}  \E\left[H(\theta^*, U^{1})H(\theta^*,U^{1})^{T}\right]\exp\left(-s(Dh(\theta^*)-\zeta I_d)\right) ds$. To deal with the last remaining term, namely $n^{\alpha} \sum_{\ell=1}^{L} \left(\theta^{m^{\ell}}_{M_{\ell}} - \theta^{m^{\ell-1}}_{M_{\ell}} - (\theta^{*,m^{\ell}} - \theta^{*,m^{\ell-1}})\right)$
we will need the following lemma.
\end{proof}

\begin{LEMME}\label{conv:lem:2step} Under the assumptions of Theorem \ref{TWO:LVL:SA}, one has
$$
n^{\alpha} \sum_{\ell=1}^{L} \left(\theta^{m^{\ell}}_{M_{\ell}} - \theta^{m^{\ell-1}}_{M_{\ell}} - (\theta^{*,m^{\ell}} - \theta^{*,m^{\ell-1}}) \right) \Longrightarrow \mathcal{N}(0, \Theta^{*}), \ \ n\rightarrow +\infty,
$$

\noindent with 
\begin{align}
\Theta^*  := \int_0^{\infty} & (e^{-s(Dh(\theta^*)- \zeta I_d)})^{T}  \tilde{\E}\left[ \left(D_xH(\theta^{*}, U) V^m - \tilde{\E}[D_xH(\theta^{*}, U) V^{m}] \right) \left(D_xH(\theta^{*}, U) V^m - \tilde{\E}[D_xH(\theta^{*}, U) V^{m}]\right)^{T}\right] \nonumber  \\
& \times e^{-s(Dh(\theta^*)-\zeta I_d)} ds. \label{asympt:cov:mlvl}
\end{align}
  
\end{LEMME}

\begin{REM} The previous result shows that a CLT for the multi-level stochastic approximation estimator of $\theta^*$ holds if the standard weak error (and thus the implicit weak error), is of order $1/n^{\alpha}$ and the strong rate error is of order $1/n^{\rho}$ with $\alpha>\rho$ or $\alpha=1$ and $\rho=1/2$. Due to the non-linearity of the procedures, which leads to annoying remainder terms in the Taylor's expansions, those results do not seem to easily extend to a weak discretization error of order $1/n^{\alpha}$ with $\alpha < 1$ and $\rho=1/2$ or a faster strong convergence rate $\rho>1/2$. Moreover, for the same reason this result does not seem to extend to the empirical sequence associated to the multi-level estimator according to the Ruppert \& Polyak averaging principle.
\end{REM}

\subsection{Complexity Analysis}\label{complex:anal:sec}
The result of Theorem \ref{TWO:LVL:SA} can be interpreted as follows. For a total error of order $1/n^{\alpha}$, it is necessary to set $M_1= \gamma^{-1}(1/n^{2\alpha})$ steps of a stochastic algorithm with time step $n^{\beta}$ and $M_2 = \gamma^{-1}(1/(n^{2\alpha-2\rho\beta}))$ steps of two stochastic algorithms with time step $n$ and $n^{\beta}$ using the same Brownian motion, the samples used for the first $M_1$ steps being independent of those used for the second scheme. Hence, the complexity of the statistical Romberg stochastic approximation method is given by
\begin{equation}
\label{Compl:SR-SA}
C_{\text{SR-SA}}(\gamma) = C \times (n^{\beta} \gamma^{-1}(1/n^{2\alpha}) + (n + n^{\beta})\gamma^{-1}(1/(n^{2\alpha-2\rho\beta}))) 
\end{equation}

\noindent under the constraint: $\alpha > 2\rho\beta \vee \rho$. Consequently, concerning the impact of the step sequence $(\gamma_n)_{n \geq1}$ on the complexity of the procedure we have the two following cases:
\begin{itemize}
\item If we choose $\gamma(p) = \gamma_0/p$ then simple computations show that $\beta^{*}=1/(1+2\rho)$ is the optimal choice leading to a complexity 
$$
C_{\text{SR-SA}}(\gamma) = C' n^{2\alpha+1/(1+2\rho)}, 
$$

\noindent under the constraint $ \underline{\lambda} \gamma_0 > \alpha(1+2\rho)/(2\alpha(1+2\rho)-2\rho)$ and $\alpha > 2\rho/(1+2\rho)$. Let us note that this computational cost is similar to the one achieved by the statistical Romberg Monte Carlo method for the computation of $\E_x[f(X_T)]$.

\item If we choose $\gamma(p) = \gamma_0/p^{a}$, $\frac12 < a < 1$ then the computational cost is given by
$$
C_{\text{SR-SA}}(\gamma) = C' (n^{\frac{2\alpha}{a}+\beta} + n^{\frac{2 \alpha}{a}-\frac{\beta}{a} +1} )
$$ 

\noindent which is minimized for $\beta^*= a/(2\rho+a)$ leading to an optimal complexity
$$
C_{\text{SR-SA}}(\gamma) = C' n^{\frac{2\alpha}{a}+\frac{a}{2\rho+a}}.
$$

\noindent under the constraint $\alpha>2\rho a/(a+2\rho) \vee \rho$.
Observe that this complexity decreases with respect to $a$ and that it is minimal for $a \rightarrow 1$ leading to the optimal computational cost obtained in the previous case. Let us also point out that contrary to the case $\gamma(p)=\gamma_0/p$, $p\geq1$ there is no constraint on the choice of $\gamma_0$. Moreover, such condition is difficult to handle in practical implementation so that a blind choice has often to be made.

\end{itemize}

The CLT proved in Theorem \ref{TWO:LVL:SA:AV} shows that for a total error of order $1/n^{\alpha}$, it is necessary to set $M_1 = n^{2\alpha}$, $M_2= n^{2\alpha-2\rho\beta}$ and to simultaneously compute its empirical mean, which represents a negligible part of the total cost. Both stochastic approximation algorithm are devised with a step $\gamma$ satisfying \A{HS1} with $a \in (1/2,1)$ and $a> \frac{\alpha}{2\alpha-2\rho\beta} \vee \frac{\alpha(1-\beta)}{\alpha-\rho\beta}$. It is plain to see that $\beta^{*}=1/(1+2\rho)$ is the optimal choice leading to a complexity given by
\begin{align*}
C_{\text{SR-RP}}(\gamma) = C \times n^{2\alpha+1/(1+2\rho)}, 
\end{align*}

\noindent provided that $a> \frac{\alpha(1+2\rho)}{2\alpha +2 \rho (2\alpha-1)}$ and $\forall \theta \in \R^d, \ \ n^{\alpha - (\alpha- \frac{\rho}{1+2\rho})a} \|Dh(\theta) - Dh^{n^{1/(1+2\rho)}}(\theta)\| \rightarrow 0$ as $n\rightarrow +\infty$ (note that when $a \rightarrow 1$ this condition is the same as in Theorem \ref{TWO:LVL:SA}). For instance, if $\alpha=1$ and $\rho=1/2$, then this condition writes $a>2/3$ and $n^{1 - \frac{3}{4}a} \|Dh(\theta) - Dh^{n^{1/2}}(\theta)\| \rightarrow 0$ and $a$ should be selected sufficiently close to $1$ according to the weak discretization error of the Jacobian matrix of $h$. Therefore, the optimal complexity is reached for free without any condition on $\gamma_0$ thanks to the Ruppert \& Polyak averaging principle. Let us also note that ought we do not intend to develop this point, it is possible to prove that averaging allows to achieve the optimal asymptotic covariance matrix as for standard SA algorithms.

Finally, concerning the CLT provided in Theorem \ref{MLVL:SA} shows that in order to obtain an error of order $1/n^{\alpha}$, one has to set $M_0=\gamma^{-1}(1/n^{2\alpha})$ and $M_l=\gamma^{-1}(m^{\ell\frac{(1+2\rho)}{2}}(m^{\frac{1-2\rho}{2}}-1)/(n^{2\alpha} (n^{\frac{(1-2\rho)}{2}}-1)) )$, if $\rho\in(0,1/2)$ or $M_0=\gamma^{-1}(1/n^{2})$ and $M_l=\gamma^{-1}(m^{\ell}\log(m)/(n^{2}\log(n)(m-1) ) )$ if $\rho=1/2$, $\ell=1, \cdots, L$. In both cases the complexity of the multi-level SA method is given by
\begin{equation}
\label{Compl:ML-SA}
C_{\text{ML-SA}}(\gamma) = C \times \left(\gamma^{-1}(1/n^{2\alpha}) + \sum_{\ell=1}^{L} M_\ell (m^{\ell}+m^{\ell-1}) \right).
\end{equation}

As for the Statistical Romberg SA method, we distinguish the two following cases:
\begin{itemize}
\item If $\gamma(p) = \gamma_0/p$ then the optimal complexity is given by
$$
C_{\text{ML-SA}}(\gamma) = C \left(n^{2\alpha} + \frac{n^2 (n^{\frac{(1-2\rho)}{2}}-1)}{m^{\frac{1-2\rho}{2}}-1}\sum_{\ell=1}^L m^{-\frac{(1+2\rho)}{2}\ell} (m^{\ell}+m^{\ell-1}) \right)=\O(n^{2\alpha} n^{1-2\rho}), 
$$

\noindent if $\rho \in (0,1/2)$ under the constraint $\underline{\lambda}\gamma_0>\alpha(\alpha-2\rho)$ and
$$
C_{\text{ML-SA}}(\gamma) = C \left(n^{2} + n^2 (\log n )^2 \frac{m^2-1}{m (\log m)^2} \right)=\O(n^{2} (\log(n))^2), 
$$

\noindent if $\rho=1/2$ under the constraint $ \underline{\lambda} \gamma_0 > 1$. These computational costs are similar to those achieved by the multi-level Monte Carlo method for the computation of $\E_x[f(X_T)]$, see \cite{Giles:08} and \cite{ben:keb:12}. As discussed in \cite{Giles:08}, this complexity attains a minimum near $m=7$.

\item If we choose $\gamma(p) = \gamma_0/p^{a}$, $\frac12 < a < 1$ then simple computations show that the computational cost is given by
$$
C_{\text{ML-SA}}(\gamma) = C \left(n^{\frac{2\alpha}{a}} + n^{\frac{2}{a}} (n^{1-2\rho}-1)^{\frac{1}{a}} \sum_{\ell=1}^L m^{-\frac{(1+2\rho)}{a}\ell} (m^{\ell}+m^{\ell-1}) \right)=\O(n^{\frac{2\alpha}{a}} n^{\frac{1-2\rho}{a}}), 
$$

\noindent if $\rho \in (0,1/2)$ and
$$
C_{\text{ML-SA}}(\gamma) = C \times \left(n^{\frac{2}{a}} + n^{\frac{2}{a}} (\log n)^{\frac{1}{a}} \frac{(m-1)^{\frac{1}{a}}(m+1)}{m (\log m)^{\frac{1}{a}}} \sum_{\ell=1}^{L} m^{-\ell(\frac{1}{a}-1)} \right)= \O(n^{\frac{2}{a}} (\log n)^{\frac{1}{a}})
$$ 

\noindent if $\rho=1/2$.
Observe that once again these computational costs decrease with respect to $a$ and that they are minimal for $a \rightarrow 1$ leading to the optimal computational cost obtained in the previous case. In this last case, the optimal choice for the parameter $m$ depends on the value of $a$.
\end{itemize}

\begin{REM}The value of $M_{0}$ in Theorem \ref{MLVL:SA} seems arbitrary and is asymptotically suboptimal. Indeed choosing $M_0=\gamma^{1}(1/(n^{2\alpha}n^{1-2\rho}))$ for $\rho\in (0,1/2)$ and $M_0=\gamma^{-1}(1/(n^2\log(n)))$ for $\rho=1/2$ does not change the asymptotic computational complexity and simplifies the asymptotic covariance matrix $\Sigma^{*}$. One easily proves that $n^{\alpha}(\theta^{1}_{M_0}-\theta^{*,1})$ converges to $0$ in probability so that $\Sigma^{*}$ now writes
\begin{align*}
\Sigma^*  := \int_0^{\infty}&  \left(e^{-s(Dh(\theta^*)- \zeta I_d)}\right)^{T}  \tilde{\E}\left[ \left(D_xH(\theta^{*}, U) V-\tilde{\E}[D_xH(\theta^*, U) V]\right) \left(D_xH(\theta^{*}, U) V-\tilde{\E}[D_xH(\theta^*, U) V]\right)^{T}\right] \\
& \times e^{-s(Dh(\theta^*)- \zeta I_d)} ds.
\end{align*}
\end{REM}
\section{Proofs of main results}\label{proof:sec}

\subsection{Proof of Theorem \ref{thm:conv:disc}}
We first prove that $\theta^{*,n}\rightarrow \theta^*$, $n\rightarrow + \infty$. Let $\epsilon>0$. The mean-reverting assumption \eqref{MeanRevert_Ass} and the continuity of $u \mapsto \langle u, h(\theta^* + \epsilon u) \rangle$ on the (compact) set $\mathcal{S}_d:=\left\{ u \in \R^d, |u| =1 \right\}$ yields
$$
\eta := \inf_{u \in S_d} \langle u, h(\theta^* + \epsilon u) \rangle >0 .
$$

The local uniform convergence of $(h^{n})_{n\geq1}$ implies
$$
\exists n_{\eta} \in \N^{*},  \ \forall n\geq n_{\eta},  \ \ \theta \in \bar{B}(\theta^*,\epsilon) \ \Rightarrow \ |h^{n}(\theta)-h(\theta)| \leq \eta/2.
$$

Then, using the following decomposition
$$
\langle \theta - \theta^*, h^{n}(\theta) \rangle = \langle \theta - \theta^*, h(\theta) \rangle + \langle \theta - \theta^*, h^{n}(\theta) - h(\theta) \rangle 
$$ 

\noindent one has for $\theta = \theta^* \pm \epsilon u$, $u\in \mathcal{S}_d$,
\begin{align*}
\epsilon \langle u, h^{n}(\theta^* + \epsilon u) \rangle  & \geq \langle \epsilon u, h(\theta^* + \epsilon u) \rangle - \epsilon \eta/2 \geq \epsilon \eta - \epsilon \eta /2 = \epsilon \eta/2 \\
-\epsilon \langle u, h^{n}(\theta^* - \epsilon u) \rangle  & \geq \langle -\epsilon u, h(\theta^* - \epsilon u) \rangle - \epsilon \eta/2 \geq \epsilon \eta - \epsilon \eta /2 = \epsilon \eta/2 
\end{align*}

\noindent so that, $\langle u , h^n(\theta^* +\epsilon u )\rangle >0$ and $\langle u , h^n(\theta^* - \epsilon u) \rangle < 0$ which combined with the intermediate value theorem applied to the continuous function $x \mapsto \langle u , h^n(\theta^* +x u) \rangle$ on the interval $[-\epsilon, \epsilon]$ yields: 
$$
\langle u , h^{n}(\theta^*+  \tilde{x} u) \rangle = 0
$$ 

\noindent for some $\tilde{x}=\tilde{x}(u) \in ]-\epsilon, \epsilon [$. Now we set $u=\theta^*-\theta^{*,n}/|\theta^*-\theta^{*,n}|$ as soon as it is possible (otherwise the proof is complete). Hence, there exists $x^* \in ]-\epsilon, \epsilon[$ such that
$$
\left< \frac{\theta^*-\theta^{*,n}}{|\theta^*-\theta^{*,n}|}, h^{n}\left(\theta^* + x^* \frac{\theta^*-\theta^{*,n}}{|\theta^*-\theta^{*,n}|} \right)\right> = 0
$$

\noindent so that multiplying the previous equality by $x^{*}+|\theta^*-\theta^{*,n}|$ we get
$$
\left< \theta^{*,n} + \left(\frac{x^*}{|\theta^*-\theta^{*,n}|}+1\right) (\theta^*-\theta^{*,n}) - \theta^{*,n}, h^{n}\left(\theta^{*,n} + \left(\frac{x^*}{|\theta^*-\theta^{*,n}|}+1\right) (\theta^*-\theta^{*,n}) \right)\right> = 0.
$$

\noindent Consequently, by the very definition of $\theta^{*,n}$, we deduce that $x^* = - |\theta^*-\theta^{*,n}| $ and finally $|\theta^* - \theta^{*,n}| < \epsilon$ for $n\geq n_{\eta}$. Hence, we conclude that $\theta^{*,n}\rightarrow \theta^*$. We now derive a convergence rate. A Taylor expansion yields for all $n\geq1$
$$
h^{n}(\theta^*) = h^{n}(\theta^{*,n}) + \left( \int_0^1 Dh^{n}(\lambda \theta^{*,n} + (1-\lambda) \theta^*) d\lambda \right) (\theta^* - \theta^{*,n}).
$$

Combining the local uniform convergence of $(Dh^{n})_{n\geq1}$ to $Dh$, the convergence of $(\theta^{*,n})_{n\geq1}$ to $\theta^*$ and the non-singularity of $Dh(\theta^*)$, one clearly gets that for $n$ large enough $\int_0^1 Dh^{n}(\lambda \theta^{*,n} + (1-\lambda) \theta^*) d\lambda$ is non singular and that
$$
 \left(\int_0^1 Dh^{n}(\lambda \theta^{*,n} + (1-\lambda) \theta^*) d\lambda \right)^{-1} \rightarrow Dh^{-1}(\theta^*), \ n\rightarrow + \infty.
$$

Consequently, recalling that $h(\theta^*)=0$ and $h^{n}(\theta^{*,n})=0$, it is plain to see 
$$
n^{\alpha}(\theta^{*,n}-\theta^{*}) = -\left(\int_0^1 Dh^{n}(\lambda \theta^{*,n} + (1-\lambda) \theta^*) d\lambda \right)^{-1}  n^{\alpha}(h^n(\theta^*) -h(\theta^*)) \rightarrow -Dh^{-1}(\theta^*) \mathcal{E}(h,\alpha, \theta^*).
$$

\subsection{Proof of Lemma \ref{lem:conv:opt:tradeoff}}
We define for all $p\ge 1, \ \Delta M^{n^{\delta}}_p:= h^{n^{\delta}}(\theta^{n^{\delta}}_{p-1})- H(\theta^{n^{\delta}}_{p-1}, (U^{n^{\delta}})^{p}) = \E[ H(\theta^{n^{\delta}}_{p-1} , (U^{n^{\delta}})^{p}) | \mathcal{F}_{p-1}] - H(\theta^{n^{\delta}}_{p-1}, (U^{n^{\delta}})^{p})$. Recalling that $((U^{n^{\delta}})^p)_{p \geq 1} $ is a sequence of i.i.d. random variables we have that $(\Delta M^{n^{\delta}}_p)_{p\geq1}$ is a sequence of martingale increments w.r.t. the natural filtration $\mathcal{F}:=(\F_p:=\sigma(\theta^{n^{\delta}}_0, (U^{n^{\delta}})^1, \cdots, (U^{n^{\delta}})^p); p\geq1)$. From the dynamic \eqref{RM}, one clearly gets for $p\geq0$
 \begin{align*}
\theta^{n^{\delta}}_{p+1}-\theta^{*,n^{\delta}} & = \theta^{n^{\delta}}_{p} - \theta^{*,n^{\delta}} - \gamma_{p+1}  Dh^{n^{\delta}}(\theta^{*,n^{\delta}}) (\theta^{n^{\delta}}_p-\theta^{*,n^{\delta}})  + \gamma_{p+1} \Delta M^{n^{\delta}}_{p+1} + \gamma_{p+1} \zeta^{n^{\delta}}_p 
\end{align*}

\noindent with $\zeta^{n^{\delta}}_p := Dh^{n^{\delta}}(\theta^{*,n^{\delta}}) (\theta^{n^{\delta}}_p - \theta^{*,n^{\delta}}) - h^{n^{\delta}}(\theta^{n^{\delta}}_p)$. Moreover, since $Dh^{n^{\delta}}$ is Lipschitz-continuous (uniformly in $n$) by Taylor's formula one gets $\zeta^{n^{\delta}}_p = \mathcal{O}(|\theta^{n^{\delta}}_p-\theta^{*,n^{\delta}}|^2)$. Hence, by a simple induction, we obtain
\begin{align}
\theta^{n^{\delta}}_{n}-\theta^{*,n^{\delta}} & = \Pi_{1, n} (\theta^{n^{\delta}}_{0}-\theta^{*,n^{\delta}})  +  \sum_{k=1}^{n} \gamma_k  \Pi_{k+1, n}  \Delta M^{n^{\delta}}_{k}  +  \sum_{k=1}^{n} \gamma_k  \Pi_{k+1, n} \left( \zeta^{n^{\delta}}_{k-1} + (Dh(\theta^{*})-Dh^{n^{\delta}}(\theta^{*,n^{\delta}}))(\theta^{n^{\delta}}_{k-1}-\theta^{*,n^{\delta}})\right) \label{rec:mainterm}
\end{align}

\noindent where $\Pi_{k, n}:= \prod_{j=k}^{n} \left(I_d - \gamma_{j} Dh(\theta^{*})\right)$, with the convention that $\Pi_{n+1, n} = I_d$. We now investigate the asymptotic behavior of each term in the above decomposition. Actually in step 1 and step 2 we will prove that the first and third terms in the right-hand side of above equality converges in probability to zero at a faster rate than $n^{-\alpha}$. We will then prove in step 3 that the second term satisfies a CLT at rate $n^{\alpha}$.

\noindent \textbf{Step 1: study of the sequence $\left\{n^{\alpha} \Pi_{1, \gamma^{-1}(1/n^{2\alpha})} (\theta^{n^{\delta}}_{0}-\theta^{*,n^{\delta}}) , n\geq0 \right\}$}

First, since $-Dh(\theta^*)$ is a Hurwitz matrix, $\forall \lambda \in [0, \lambda_m)$, there exists $C>0$ such that for any $k\leq n$, $\| \Pi_{k,n} \| \leq C \prod_{j=k}^{n} (1-\lambda \gamma_j) \leq C \exp(-\lambda \sum_{j=k}^n \gamma_j)$. We refer to \cite{Duflo1996} and \cite{ben:met:pri} for more details. Hence, one has for all $\eta \in (0,\lambda_m)$
$$
n^{\alpha} \E| \Pi_{1, \gamma^{-1}(1/n^{2\alpha})} (\theta^{n^{\delta}}_0 - \theta^{*,n^{\delta}}) | \leq C(\sup_{n\geq1} \E|\theta^{n}_0| +1) n^{\alpha} \exp\left(-(\lambda_m-\eta)\sum_{k=1}^{\gamma^{-1}(1/n^{2\alpha})} \gamma_k\right).
$$

Selecting $\eta$ such that $2(\lambda_m-\eta) \gamma_0 > 2 (\underline{\lambda}-\eta) \gamma_0 > 1$ under \A{HS2} and any $\eta \in (0,\lambda_m)$ under \A{HS1}, we derive the convergence to zero of the right hand side of the last but one inequality.
 
\noindent \textbf{Step 2: study of the sequence $\left\{n^{\alpha} \sum_{k=1}^{\gamma^{-1}(1/n^{2\alpha})} \gamma_k  \Pi_{k+1, \gamma^{-1}(1/n^{2\alpha})} \left( \zeta^{n^{\delta}}_{k-1} + (Dh(\theta^{*})-Dh^{n^{\delta}}(\theta^{*,n^{\delta}}))(\theta^{n^{\delta}}_{k-1}-\theta^{*,n^{\delta}})\right), n\geq0 \right\}$}

We focus on the last term of \eqref{rec:mainterm}. Using Lemma \ref{sstrongerror:tech:lemme} we get
\begin{align*}
\E \left| \sum_{k=1}^{n} \gamma_k  \Pi_{k+1, n}   (\zeta^{n^{\delta}}_{k-1} + (Dh(\theta^{*})-Dh^{n^{\delta}}(\theta^{*,n^{\delta}}))(\theta^{n^{\delta}}_{k-1}-\theta^{*,n^{\delta}}) )  \right| & \leq C  \sum_{k=1}^{n}  \|\Pi_{k+1, n}\|  (\gamma^2_k +\gamma^{3/2}_k \|Dh(\theta^{*})-Dh^{n^{\delta}}(\theta^{*,n^{\delta}})\| ),
\end{align*}

\noindent so that by Lemma \ref{stepseq:tech:lemme} (see also remark \ref{eigenvalue:rem}), the local uniform convergence of $(Dh^{n})_{n\geq1}$ and the continuity of $Dh$ at $\theta^*$, we derive
$$
\limsup_{n} n^{\alpha} \E \left| \sum_{k=1}^{\gamma^{-1}(1/n^{2\alpha})} \gamma_k  \Pi_{k+1, \gamma^{-1}(1/n^{2\alpha})}  (\zeta^{n^{\delta}}_{k-1} + (Dh(\theta^{*})-Dh^{n^{\delta}}(\theta^{*,n^{\delta}}))(\theta^{n^{\delta}}_{k-1}-\theta^{*,n^{\delta}}) )   \right| =0.
$$

\noindent \textbf{Step 3: study of the sequence $\left\{n^{\alpha} \sum_{k=1}^{\gamma^{-1}(1/n^{2\alpha})} \gamma_k  \Pi_{k+1, \gamma^{-1}(1/n^{2\alpha})} \Delta M^{n^{\delta}}_{k}, n\geq0 \right\}$}

 We use the following decomposition
\begin{align*}
\sum_{k=1}^{n} \gamma_k  \Pi_{k+1, n}  \Delta M^{n^{\delta}}_{k} & =  \sum_{k=1}^{n} \gamma_k  \Pi_{k+1, n} ( h^{n^{\delta}}(\theta^{n^{\delta}}_k) - h^{n^{\delta}}(\theta^{*,n^{\delta}}) - (H(\theta^{n^{\delta}}_k,(U^{n^{\delta}})^{k+1}) - H(\theta^{*,n^{\delta}},(U^{n^{\delta}})^{k+1}))   ) \\
& + \sum_{k=1}^{n} \gamma_k  \Pi_{k+1, n} (h^{n^{\delta}}(\theta^{*,n^{\delta}}) - H(\theta^{*,n^{\delta}},(U^{n^{\delta}})^{k+1}) ) \\
& := R_n + M_n
\end{align*}

Now, using that $\E\left[ H(\theta^{n^{\delta}}_k,(U^{n^{\delta}})^{k+1}) | \mathcal{F}_k\right] =  h^{n^{\delta}}(\theta^{n^{\delta}}_k)$, $\E\left[ H(\theta^{*,n^{\delta}},(U^{n^{\delta}})^{k+1}) | \mathcal{F}_k\right] = h^{n^{\delta}}(\theta^{*,n^{\delta}})$ and \A{HR}, we have
\begin{align*}
\E|R_n|^2 \leq \sum_{k=1}^{n} \gamma^2_k  \|\Pi_{k+1, n}\|^2 \E[|\theta^{n^{\delta}}_k - \theta^{*,n^{\delta}}|^{2a}]  \leq   \sum_{k=1}^{n} \gamma^{2+a}_k \|\Pi_{k+1, n}\|^2
\end{align*}

\noindent where we used Lemma \ref{sstrongerror:tech:lemme} and Jensen's inequality for the last inequality. Moreover, according to Lemma \ref{stepseq:tech:lemme}, we have 
$$
\limsup_{n} n^{2 \alpha} \sum_{k=1}^{\gamma^{-1}(1/n^{2\alpha})} \gamma^{2+a}_k  \|\Pi_{k+1, \gamma^{-1}(1/n^{2\alpha})}\|^2 = 0 
$$

\noindent so that, $n^{\alpha}  \sum_{k=1}^{n} \gamma_k  \Pi_{k+1, n} ( h^{n^{\delta}}(\theta^{n^{\delta}}_k)-h^{n^{\delta}}(\theta^{*,n^{\delta}}) - (H(\theta^{n^{\delta}}_k,(U^{n^{\delta}})^{k+1})-H(\theta^{*,n^{\delta}},(U^{n^{\delta}})^{k+1}))   ) \overset{L^{2}(\P)}{\longrightarrow}0$.

To conclude we prove that the sequence $\left\{\gamma^{-1/2}(n) M_n, \ n\geq0\right\}$, satisfies a CLT. In order to do this we apply standard results on CLT for martingale arrays. More precisely, we will apply Theorem 3.2 and Corollary 3.1, p.58 in \cite{Hall:Heyde:80} so that we need to prove that the conditional Lindeberg assumption is satisfied, that is $\lim_{n}\sum_{k=1}^{n}\E[|\gamma^{-1/2}(n)\gamma_k  \Pi_{k+1, n} (h^{n^{\delta}}(\theta^{*,n^{\delta}}) - H(\theta^{*,n^{\delta}},(U^{n^{\delta}})^{k+1}) )|^{p}]=0$, for some $p>2$ and that the conditional variance $(S_n)_{n\geq1}$ defined by
\begin{align*}
S_n & := \frac{1}{\gamma(n)} \sum_{k=1}^{n} \gamma^{2}_k \Pi_{k+1, n} \E_k[(h^{n^{\delta}}(\theta^{*,n^{\delta}}) - H(\theta^{*,n^{\delta}},(U^{n^{\delta}})^{k+1})) (h^{n^{\delta}}(\theta^{*,n^{\delta}}) - H(\theta^{*,n^{\delta}},(U^{n^{\delta}})^{k+1}))^{T}] \Pi^{T}_{k+1, n}, \\
& =  \frac{1}{\gamma(n)} \sum_{k=1}^{n} \gamma^{2}_k \Pi_{k+1, n} \Gamma_n \Pi^{T}_{k+1, n}
\end{align*}

\noindent with $
\Gamma_n:= \E[H(\theta^{*,n^{\delta}},U^{n^{\delta}}) (H(\theta^{*,n^{\delta}}, U^{n^{\delta}}))^{T}] 
$, since $h^{n^{\delta}}(\theta^{*,n^{\delta}})=0$, satisfies $S_n \overset{a.s.}{\longrightarrow} \Sigma^*$ as $n\rightarrow +\infty$. We also set $\Gamma^{*} := \E[H(\theta^{*},U)) (H(\theta^{*},U))^{T}]$.

By \A{HI}, it holds for some $R>0$ such that $\forall n \geq1$, $\theta^{*,n} \in B(0,R)$
\begin{align*}
\sum_{k=1}^{n} \E  \left|\gamma^{-\frac12}(n) \gamma_k  \Pi_{k+1, n}  (h^{n^{\delta}}(\theta^{*,n^{\delta}}) - H(\theta^{*,n^{\delta}},(U^{n^{\delta}})^{k+1})) \right|^{2+\delta}
\leq C \sup_{\left\{\theta: |\theta| \leq R, n\in \N^{*}\right\}}\E[|H(\theta, U^{n})|^{2+\delta}] \gamma^{-1+\frac{\delta}{2}}(n)\sum_{k=1}^{n} \gamma^{2+\delta}_k  \|\Pi_{k+1, n}\|^{2+\delta}  
\end{align*}

By Lemma \ref{stepseq:tech:lemme}, we have $\lim\sup_n  \gamma^{-1+ \delta/2}(n)\sum_{k=1}^{n} \gamma^{2+\delta}_k  \|\Pi_{k+1, n}\|^{2+\delta} \leq \lim\sup_n \gamma^{\delta/2}(n) = 0$, so that the conditional Lindeberg condition, see \cite{Hall:Heyde:80}, Corollary 3.1, is satisfied. Now we focus on the conditional variance. By the local uniform convergence of $(\theta\mapsto \E[H(\theta,U^{n^{\delta}}) (H(\theta,U^{n^{\delta}}))^{T}])_{n\geq0}$, the continuity of $\theta \mapsto \E[H(\theta,U) (H(\theta,U))^{T}]$ at $\theta^*$ and since $\theta^{*,n^{\delta}}\rightarrow \theta^*$, we have $
\Gamma_n \rightarrow \Gamma^{*}$, so that from Lemma \ref{stepseq:tech:lemme}, it follows that
$$
\lim\sup_{n} \left\|\frac{1}{\gamma(n)} \sum_{k=1}^{n} \gamma^{2}_k \Pi_{k+1, n} (\Gamma_n - \Gamma^{*})\Pi^{T}_{k+1, n} \right\| \leq \lim\sup_{n} \|\Gamma_n - \Gamma^{*}\| =0.
$$

Hence we see that $\lim_n S_n = \lim_n \frac{1}{\gamma(n)} \sum_{k=1}^{n} \gamma^2_k  \Pi_{k+1, n} \Gamma^{*} \Pi^{T}_{k+1, n}$ if this latter limit exists. Let us note that $\Sigma^{*}$ given by \eqref{cov:matrix:opt:alloc} is the (unique) matrix $A$ solution to the Lyapunov equation:
$$
\Gamma^{*} - (Dh(\theta^*)-\zeta I_d)A - A(Dh(\theta^*)-\zeta I_d)^{T} = 0.
$$
 
 We aim at proving that $S_n \overset{a.s.}{\longrightarrow} \Sigma^*$. In order to do this, we define
 $$
 A_{n+1} := \frac{1}{\gamma(n+1)}\sum_{k=1}^{n+1} \gamma^{2}_k \Pi_{k+1, n} \Gamma^{*} \Pi^{T}_{k+1, n}
 $$
 
 \noindent which can be written in the following recursive form
\begin{align*}
A_{n+1} & = \gamma_{n+1} \Gamma^{*} + \frac{\gamma_n}{\gamma_{n+1}} (I_d-\gamma_{n+1}Dh(\theta^*))A_n (I_d-\gamma_{n+1}Dh(\theta^*))^{T} \\
& = A_{n} + \gamma_n (\Gamma^{*} - Dh(\theta^*) A_n - A_n Dh(\theta^*)^{T}) + (\gamma_{n+1}-\gamma_n) \Gamma^{*} + \gamma_n \gamma_{n+1} Dh(\theta^*)A_n Dh(\theta^*)^{T} \\
& + \frac{\gamma_n - \gamma_{n+1}}{\gamma_{n+1}} A_n
\end{align*}

Under the assumptions made on the step sequence $(\gamma_n)_{n\geq1}$, we have $\frac{\gamma_n - \gamma_{n+1}}{\gamma_{n+1}} = 2 \zeta \gamma_n + o(\gamma_n)$ and $\gamma_{n+1} - \gamma_n = \O(\gamma^{2}_n)$. Consequently, introducing $Z_n = A_{n} - \Sigma^{*}$, simple computations from the previous equality yield
\begin{align*}
Z_{n+1} & = Z_{n} - \gamma_{n} \left( (Dh(\theta^*)-\zeta I_d) Z_n + Z_n (Dh(\theta^*)-\zeta I_d)^{T} \right) + \gamma_{n} \gamma_{n+1} Dh(\theta^*) Z_n Dh(\theta^*)^{T} \\
& + \left(\frac{\gamma_n - \gamma_{n+1}}{\gamma_{n+1}}-2\zeta \gamma_n I_d \right) Z_n +  \gamma_n \gamma_{n+1} Dh(\theta^*) \Sigma^* Dh(\theta^*)^{T} + (\gamma_{n+1} - \gamma_{n}) \Gamma^{*}  +\left(\frac{\gamma_n - \gamma_{n+1}}{\gamma_{n+1}}-2\zeta \gamma_n I_d \right) \Sigma^{*}
\end{align*}

\noindent Let us note that by the very definition of $\zeta$ and assumptions \A{HS1}, \A{HS2}, the matrix $Dh(\theta^*)-\zeta I_d$ is stable, so that taking the norm in the previous equality, there exists $\lambda >0$ such that
$$
\|Z_{n+1}\| \leq (1-\lambda \gamma_n + o(\gamma_n)) \|Z_n\| + o(\gamma_n)
$$ 

\noindent for $n\geq n_0$, $n_0$ large enough. By a simple induction, it holds for $n\geq N \geq n_0$
$$
\|Z_n\| \leq C \|Z_N\| \exp(-\lambda s_{N,n}) + C \exp(-\lambda s_{N,n}) \sum_{k=N}^{n} \exp(\lambda s_{N,k}) \gamma_k \|e_k\|
$$ 

\noindent where $e_n = o(1)$ and we set $s_{N,n}:=\sum_{k=N}^{n} \gamma_k$. From the assumption \eqref{STEP}, it follows that for $N\geq n_0$
$$
\lim\sup_{n} \|Z_n\| \leq C \sup_{k\geq N} \|e_k\|
$$

\noindent and passing to the limit as $N$ goes to infinity it clearly yields $\lim\sup_{n} \| Z_n \| =0$. Hence, $S_n \overset{a.s.}{\longrightarrow}\Theta^{*}$ and the proof is complete.


%

\subsection{Proof of Lemma \ref{lem:conv:mean}}
 We freely use the notations and the intermediate results of the proof of Lemma \ref{lem:conv:opt:tradeoff}. Using \eqref{rec:mainterm} in its recursive form, for any $p\geq0$ and for $n$ large enough, it holds 
$$
\theta^{n^{\delta}}_p - \theta^{*,n^{\delta}}  = - \frac{1}{\gamma_{p+1}} (Dh^{n^{\delta}}(\theta^{*,n^{\delta}}))^{-1} (\theta^{n^{\delta}}_{p+1} - \theta^{n^{\delta}}_{p}) + (Dh^{n^{\delta}}(\theta^{*,n^{\delta}}))^{-1} \Delta M^{n^{\delta}}_{p+1} + (Dh^{n^{\delta}}(\theta^{*,n^{\delta}}))^{-1} \zeta^{n^{\delta}}_{p} .
$$

Hence, using an Abel's transform we derive
\begin{align*}
\bar{\theta}^{n^{\delta}}_{n^{2\alpha}} - \theta^{*,n^{\delta}} & = \frac{1}{n^{2\alpha}+1} \sum_{k=0}^{n^{2\alpha}} \theta^{n^{\delta}}_k - \theta^{*,n^{\delta}} = - \frac{(Dh^{n^{\delta}}(\theta^{*,n^{\delta}}))^{-1} }{n^{2\alpha}+1} \sum_{k=0}^{n^{2\alpha}} \frac{1}{\gamma_{k+1}} (\theta^{n^{\delta}}_{k+1} - \theta^{n^{\delta}}_{k})  \\
& +  \frac{(Dh^{n^{\delta}}(\theta^{*,n^{\delta}}))^{-1}}{n^{2\alpha}+1} \sum_{k=0}^{n^{2\alpha}} \Delta M^{n^{\delta}}_{k+1} +  \frac{(Dh^{n^{\delta}}(\theta^{*,n^{\delta}}))^{-1}}{n^{2\alpha}+1} \sum_{k=0}^{n^{2\alpha}}  \zeta^{n^{\delta}}_{k} \\
& = - \frac{(Dh^{n^{\delta}}(\theta^{*,n^{\delta}}))^{-1}}{n^{2\alpha}+1} \left(\frac{\theta^{n^{\delta}}_{n^{2\alpha}+1}-\theta^{*,n^{\delta}}}{\gamma_{n^{2\alpha}+1}} - \frac{\theta^{n^{\delta}}_0-\theta^{*,n^{\delta}}}{\gamma_1}\right) -  \frac{(Dh^{n^{\delta}}(\theta^{*,n^{\delta}}))^{-1}}{n^{2\alpha}+1} \sum_{k=1}^{n^{2\alpha}} \left(\frac{1}{\gamma_k} - \frac{1}{\gamma_{k+1}}  \right) (\theta^{n^{\delta}}_k - \theta^{*,n^{\delta}}) \\
& + \frac{(Dh^{n^{\delta}}(\theta^{*,n^{\delta}}))^{-1} }{n^{2\alpha}+1} \sum_{k=0}^{n^{2\alpha}} \Delta M^{n^{\delta}}_{k+1}  + \frac{(Dh^{n^{\delta}}(\theta^{*,n^{\delta}}))^{-1}}{n^{2\alpha}+1} \sum_{k=0}^{n^{2\alpha}}  \zeta^{n^{\delta}}_{k} 
\end{align*}

 We now study each term of the above decomposition. 

\noindent \textbf{Step 1: study of the sequence $\left\{\frac{n^{\alpha}}{n^{2\alpha}+1} \left(\frac{\theta^{n^{\delta}}_{n^{2\alpha}+1}-\theta^{*,n^{\delta}}}{\gamma_{n^{2\alpha}+1}} - \frac{\theta^{n^{\delta}}_0-\theta^{*,n^{\delta}}}{\gamma_1}\right), n\geq0 \right\}$}

For the first term, by Lemma \ref{sstrongerror:tech:lemme} it follows 
\begin{align*}
\E\left|\frac{n^{\alpha}}{n^{2\alpha}+1}  \left(\frac{\theta^{n^{\delta}}_{n^{2\alpha}+1}-\theta^{*,n^{\delta}}}{\gamma_{n^{2\alpha}+1}} - \frac{\theta^{n^{\delta}}_0-\theta^{*,n^{\delta}}}{\gamma_1}\right) \right| & \leq C \left( \frac{1}{ \sqrt{n^{2\alpha} \gamma_{n^{2\alpha}+1}}}  + \frac{1}{n^{\alpha}} (\sup_{n\geq1}\E|\theta^{n}_0| + 1)\right) \\ 
& \leq C \left( \frac{1}{\sqrt{n^{2\alpha} \gamma_{n^{2\alpha}+1}}} + \frac{1}{n^{\alpha}} \right)\longrightarrow 0,
\end{align*}

\noindent since by \A{HS1} one has $n\gamma_{n}\rightarrow 0$, $n\rightarrow +\infty$.

\noindent \textbf{Step 2: study of the sequence $\left\{ \frac{n^{\alpha}}{n^{2\alpha}+1}  \sum_{k=1}^{n^{2\alpha}} \left( \frac{1}{\gamma_k} - \frac{1}{\gamma_{k+1}}  \right) (\theta^{n^{\delta}}_k-\theta^{*,n^{\delta}}), n\geq0 \right\}$}

Similarly for the second term, we have
\begin{align*}
\E \left| \frac{n^{\alpha}}{n^{2\alpha}+1}  \sum_{k=1}^{n^{2\alpha}} \left( \frac{1}{\gamma_k} - \frac{1}{\gamma_{k+1}}  \right) (\theta^{n^{\delta}}_k - \theta^{*,n^{\delta}}) \right| & \leq C \frac{1}{n^{\alpha}} \sum_{k=1}^{n^{2\alpha}} \gamma^{1/2}_{k} \left( \frac{1}{\gamma_{k+1}} - \frac{1}{\gamma_{k}} \right) \gamma^{-1/2}_k \E|\theta^{n^{\delta}}_k-\theta^{*,n^{\delta}}| \\
& \leq C \frac{1}{n^{\alpha}} \sum_{k=1}^{n^{2\alpha}} \gamma^{1/2}_{k} \left( \frac{1}{\gamma_{k+1}} - \frac{1}{\gamma_{k}} \right) \rightarrow 0, \ \ n\rightarrow + \infty.
\end{align*}

\noindent where we used Lemma \ref{sstrongerror:tech:lemme} for the last inequality and assumption \A{HS1} with $a<1$.

%

\noindent \textbf{Step 3: study of the sequence $\left\{ \frac{n^{\alpha}}{n^{2\alpha}+1} \sum_{k=0}^{n^{2\alpha}} \Delta M^{n^{\delta}}_{k+1}  , n\geq0 \right\}$}

 As in the proof of Lemma \ref{lem:conv:opt:tradeoff}, we decompose this sequence as follows
\begin{align*}
 \frac{n^{\alpha}}{n^{2\alpha}+1} \sum_{k=0}^{n^{2\alpha}}  \Delta M^{n^{\delta}}_{k+1} & =   \frac{n^{\alpha}}{n^{2\alpha}+1} \sum_{k=1}^{n^{2\alpha}}  ( h^{n^{\delta}}(\theta^{n^{\delta}}_k)-h^{n^{\delta}}(\theta^{*,n^{\delta}}) - (H(\theta^{n^{\delta}}_k,(U^{n^{\delta}})^{k+1})-H(\theta^{*,n^{\delta}},(U^{n^{\delta}})^{k+1}))   ) \\
& + \frac{n^{\alpha}}{n^{2\alpha}+1} \sum_{k=1}^{n^{2\alpha}}  (h^{n^{\delta}}(\theta^{*,n^{\delta}}) - H(\theta^{*,n^{\delta}},(U^{n^{\delta}})^{k+1}) ) \\
& := R_n + M_n
\end{align*}

For the sequence $(R_n)_{n\geq1}$ we use \A{HR} to write
$$
\E|R_n|^2 \leq \frac{C}{n^{2\alpha}} \sum_{k=0}^{n^{2\alpha}}  \E|H(\theta^{n^{\delta}}_k, (U^{n^{\delta}})^{k+1}) - H(\theta^{*,n^{\delta}}, U^{n^{\delta}})|^2 =  \frac{C}{n^{2\alpha}} \sum_{k=1}^{n^{2\alpha}} \gamma^{2a}_k  \rightarrow 0,
$$

\noindent owing to Ces\`{a}ro's Lemma. We now prove a CLT for the sequence $(M_n)_{n\geq1}$ by applying Theorem 3.2 and Corollary 3.1, p.58 in \cite{Hall:Heyde:80}. Since $\theta^{*,n^{\delta}} \rightarrow \theta^{*}$ and by \A{HI} it holds for some $R>0$
\begin{align*}
 \sum_{k=0}^{n^{2\alpha}} \E\left| \frac{n^{\alpha}}{n^{2\alpha}+1} (h^{n^{\delta}}(\theta^{*,n^{\delta}}) - H(\theta^{*,n^{\delta}},(U^{n^{\delta}})^{k+1}) )  \right|^{2+\delta} & \leq \frac{C}{n^{\alpha b}} (\sup_{\theta: |\theta| \leq R, \ n \in \N^{*}} \E|H(\theta, U^{n})|^{2+\delta})  \rightarrow 0, \ \ n\rightarrow +\infty,
\end{align*}

\noindent so that the conditional Lindeberg condition is satisfisfied, see \cite{Hall:Heyde:80} Corollary 3.1. Now, we focus on the conditional variance. For convenience, we set
\begin{align*}
S_n & := \frac{n^{2\alpha}}{(n^{2\alpha}+1)^2} \sum_{k=1}^{n^{2\alpha}} \E_k[ (h^{n^{\delta}}(\theta^{*,n^{\delta}}) - H(\theta^{*,n^{\delta}},(U^{n^{\delta}})^{k+1}) )  (h^{n^{\delta}}(\theta^{*,n^{\delta}}) - H(\theta^{*,n^{\delta}}, (U^{n^{\delta}})^{k+1}) )^{T} ] \\
& =  \frac{n^{2\alpha}}{(n^{2\alpha}+1)^2} \sum_{k=1}^{n^{2\alpha}} \E[H(\theta^{*,n^{\delta}}, U^{n^{\delta}})  (H(\theta^{*,n^{\delta}}, U^{n^{\delta}}) )^{T} ] \\
& = \frac{n^{4\alpha}}{(n^{2\alpha}+1)^2} \E[H(\theta^{*,n^{\delta}}, U^{n^{\delta}})  (H(\theta^{*,n^{\delta}}, U^{n^{\delta}}) )^{T} ],
\end{align*}

\noindent so that we clearly have $S_n \rightarrow \E[H(\theta^{*},U)  (H(\theta^{*},U) )^{T} ]$ by the local uniform convergence of $(\theta \mapsto \E[H(\theta,U^{n})  (H(\theta, U^{n}) )^{T} ] )_{n\geq1}$, the continuity of $\theta \mapsto \E[H(\theta, U)  (H(\theta, U) )^{T} ]$ at $\theta^*$ and the convergence of $(\theta^{*,n^{\delta}})_{n\geq1}$ towards $\theta^*$. Therefore, since $(Dh^{n^{\delta}}(\theta^{*,n^{\delta}}))^{-1} \rightarrow (Dh(\theta^*))^{-1}$, we conclude that
$$
 (Dh^{n^{\delta}}(\theta^{*,n^{\delta}}))^{-1} \frac{n^{\alpha}}{n^{2\alpha}+1} \sum_{k=0}^{n^{2\alpha}}  \Delta M^{n^{\delta}}_{k+1} \Longrightarrow \mathcal{N}(0, Dh(\theta^*)^{-1} \E_x[H(\theta^*, U)H(\theta^*, U)^{T}] (Dh(\theta^*)^{-1})^{T}). 
$$

%
%
\noindent \textbf{Step 4: study of the sequence $\left\{ \frac{n^{\alpha}}{n^{2\alpha}+1} \sum_{k=0}^{n^{2\alpha}} \zeta^{n^{\delta}}_{k}  , n\geq0 \right\}$}

Now, observe that by Lemma \ref{sstrongerror:tech:lemme} the last term is bounded in $L^{1}$-norm by
\begin{align*}
 \frac{n^{\alpha}}{n^{2\alpha}+1} \sum_{k=0}^{n^{2\alpha}}  \E|\zeta^{n^{\delta}}_{k}|  & \leq \frac{C}{n^{\alpha}} \sum_{k=0}^{n^{2\alpha}} \gamma_k  \rightarrow 0, \ n\rightarrow+\infty 
\end{align*}

\noindent since $\gamma$ satisfies \A{HS1} with $a <1$.
\subsection{Proof of Lemma \ref{conv:lem:2step:sr}}
 We will just prove the first assertion of the Lemma. The second one will readily follow. When the exact value of a constant is not important we may repeat the same symbol for constants that may change from one line to next. We come back to the decomposition used in the proof of Lemma \ref{lem:conv:opt:tradeoff}. We consequently use the same notations. Let us note that the procedure $(\theta_p)_{p\geq0}$ $a.s.$ converges to $\theta^*$ and satisfies a CLT according to Theorem \ref{CLT_SA}.

From the dynamics of $(\theta^{n^{\delta}}_p)_{p\geq0}$ and $(\theta_p)_{p\geq0}$ we write for $p\geq0$
\begin{align*}
\theta^{n^{\beta}}_{p+1}-\theta^{*,n^{\beta}} & = \theta^{n^{\beta}}_{p} - \theta^{*,n^{\beta}} - \gamma_{p+1}  Dh^{n^{\beta}}(\theta^{*,n^{\beta}}) (\theta^{n^{\beta}}_p-\theta^{*,n^{\beta}})  + \gamma_{p+1} \Delta M^{n}_{p+1} + \gamma_{p+1} \zeta^{n^{\beta}}_p \\
\theta_{p+1}-\theta^{*} & = \theta_{p} - \theta^{*} - \gamma_{p+1}  Dh(\theta^{*}) (\theta_p-\theta^{*})  + \gamma_{p+1} \Delta M_{p+1} + \gamma_{p+1} \zeta_p,
\end{align*}

\noindent with $\Delta M_{p+1}=h(\theta_p)-H(\theta_p,(U)^{p+1})$, $p\geq0$, and $\zeta^{n^{\beta}}_p:=  Dh^{n^{\beta}}(\theta^{*,n^{\beta}}) (\theta^{n^{\beta}}_p-\theta^{*,n^{\beta}}) - h^{n^{\beta}}(\theta^{n^{\beta}}_p)$, $ \zeta_p = Dh(\theta^{*}) (\theta_p-\theta^{*}) - h(\theta_p)$. Since $Dh^{n}$ and $Dh$ are Lipschitz-continuous, by Taylor's formula one gets $\zeta^{n^{\beta}}_p = \O(|\theta^{n^{\beta}}_p-\theta^{*,n^{\beta}}|^2)$ and $\zeta_p = \O(|\theta_p-\theta^{*}|^2)$. Therefore, defining $z^{n^{\beta}}_{p} = \theta^{n^{\beta}}_{p}-\theta_{p}  - (\theta^{*,n^{\beta}} - \theta^{*} )$, $p\geq0$, with $z^{n^{\beta}}_{0} = \theta^{*}- \theta^{*,n^{\beta}}$, by a simple induction argument one has
\begin{align}
z^{n^{\beta}}_{n} & = \Pi_{1, n} z^{n^{\beta}}_{0}  +  \sum_{k=1}^{n} \gamma_k  \Pi_{k+1, n} \Delta N^{n^{\beta}}_{k}  +   \sum_{k=1}^{n} \gamma_k  \Pi_{k+1, n}  \Delta R^{n^{\beta}}_{k} \nonumber \\
& +  \sum_{k=1}^{n} \gamma_k  \Pi_{k+1, n} \left( \zeta^{n}_{k-1} - \zeta_{k-1} + (Dh(\theta^{*})-Dh^{n^{\beta}}(\theta^{*,n^{\beta}}))(\theta^{n^{\beta}}_{k-1}-\theta^{*,n^{\beta}})\right) \label{rec:two:level:sr}
\end{align}

\noindent where $\Pi_{k, n}:= \prod_{j=k}^{n} \left(I_d - \gamma_{j} Dh(\theta^{*})\right)$, with the convention that $\Pi_{n+1, n} = I_d$, and $\Delta N^{n^{\beta}}_{k} := h^{n^{\beta}}(\theta^{*}) - h(\theta^{*}) - (H(\theta^{*},(U^{n^{\beta}})^{k+1}) - H(\theta^{*}, U^{k+1}))$, $\Delta R^{n^{\beta}}_{k} = h^{n^{\beta}}(\theta^{n^{\beta}}_k)-h^{n^{\beta}}(\theta^{*}) - (H(\theta^{n^{\beta}}_k,(U^{n^{\beta}})^{k+1})-H(\theta^{*},(U^{n^{\beta}})^{k+1}))  + H(\theta_k, U^{k+1}) - H(\theta^*, U^{k+1}) - (h(\theta_k) - h(\theta^*) ) $ for $k\geq1$. We will now investigate the asymptotic behavior of each term in the above decomposition. We will see that the second term which represents the non-linearity in the innovation variables $(U^{n^{\beta}},U)$ provides the announced weak rate of convergence.

\noindent \textbf{Step 1: study of the sequence $\left\{n^{\alpha} \Pi_{1, \gamma^{-1}(1/(n^{2\alpha-2\rho\beta}))} z^{n^{\beta}}_{0}, n\geq0 \right\}$}

Under the assumptions on the step sequence $\gamma$, one has for all $\eta \in (0, \lambda_m)$
\begin{align*}
n^{\alpha} \E\left[|\Pi_{1, \gamma^{-1}(1/(n^{2\alpha-2\rho \beta}))} z^{n^{\beta}}_{0} |\right] & \leq n^{\alpha} \|\Pi_{1, \gamma^{-1}(1/(n^{2\alpha - 2\rho \beta}))}\| (\E|\theta_0|+\sup_{n\geq1}\E|\theta^{n}_0|+|\theta^{*,n^{\beta}}-\theta^{*}|)\\
& \leq C n^{\alpha}  \exp\left(-(\lambda_m-\eta)\sum_{k=1}^{\gamma^{-1}(1/(n^{2\alpha-2\rho \beta}))} \gamma_k\right)  \rightarrow 0,
\end{align*}

\noindent by selecting $\eta$ s.t. $(\lambda_m-\eta) \gamma_0 > (\underline{\lambda}-\eta) \gamma_0 > \alpha/(2\alpha-2\rho\beta)$ if $\gamma(p)=\gamma_0/p$, $p\geq1$.

\noindent \textbf{Step 2: study of the sequence \\
$\left\{n^{\alpha} \sum_{k=1}^{\gamma^{-1}(1/(n^{2\alpha-2\rho \beta}))} \gamma_k  \Pi_{k+1, \gamma^{-1}(1/(n^{2\alpha-2\rho \beta}))} \left( \zeta^{n}_{k-1} - \zeta_{k-1} + (Dh(\theta^{*})-Dh^{n^{\beta}}(\theta^{*,n^{\beta}}))(\theta^{n^{\beta}}_{k-1}-\theta^{*,n^{\beta}})\right), n\geq0 \right\}$}

By Lemma \ref{sstrongerror:tech:lemme}, one has
\begin{align*}
\E \left| \sum_{k=1}^{n} \gamma_k  \Pi_{k+1, n}   (\zeta^{n^{\beta}}_{k-1} -\zeta_{k-1} + (Dh(\theta^{*})-Dh^{n^{\beta}}(\theta^{*,n^{\beta}}))(\theta^{n^{\beta}}_{k-1}-\theta^{*,n^{\beta}}) )  \right| & \leq C  \sum_{k=1}^{n}  \|\Pi_{k+1, n}\|  (\gamma^2_k +\gamma^{3/2}_k \|Dh(\theta^{*})-Dh^{n^{\beta}}(\theta^{*,n^{\beta}})\| ),
\end{align*}

\noindent so that by Lemma \ref{stepseq:tech:lemme}, we easily derive that (if $\gamma(p)=\gamma_0/p$ recall that $\underline{\lambda} \gamma_0> \alpha/(2\alpha-2\rho\beta)$) $\sum_{k=1}^{n}  \gamma^{2}_k  \|\Pi_{k+1, n}\| = o(\gamma^{\alpha/(2\alpha-2\rho \beta)}(n))$ and (recall that  $\underline{\lambda} \gamma_0> \alpha/(2\alpha-2\rho\beta)>1/2$) $\sum_{k=1}^{n}  \gamma^{3/2}_k  \|\Pi_{k+1, n}\| = \O(\gamma^{1/2}(n))$ so that
$$
\limsup_n n^{\alpha} \sum_{k=1}^{\gamma^{-1}(1/(n^{2\alpha - 2\rho \beta}))}  \gamma^{2}_k  \|\Pi_{k+1, \gamma^{-1}(1/(n^{2\alpha-2\rho\beta}))}\| = 0.
$$

Moreover, since $Dh^{n^{\beta}}$ is Lipschitz-continuous (uniformly in $n$) we clearly have
\begin{align*}
 \sum_{k=1}^{n} \gamma^{3/2}_k \|\Pi_{k+1, n}\|  \|Dh(\theta^{*})-Dh^{n^{\beta}}(\theta^{*,n^{\beta}})\| \leq \sum_{k=1}^{n} \gamma^{3/2}_k \|\Pi_{k+1, n}\|  (\|Dh(\theta^{*})-Dh^{n^{\beta}}(\theta^{*})\| +  |\theta^{*,n^{\beta}} -\theta^{*}|)
\end{align*}

\noindent which combined with $n^{\rho \beta} \|Dh(\theta^{*})-Dh^{n^{\beta}}(\theta^{*})\| \rightarrow 0$ and $n^{\rho \beta} |\theta^{*,n^{\beta}} -\theta^{*}| \rightarrow 0$ (recall that $\alpha>\rho$) imply that $\limsup_{n} n^{\alpha} \sum_{k=1}^{\gamma^{-1}(1/(n^{2\alpha-2\rho\beta}))} \gamma^{3/2}_k \|\Pi_{k+1, \gamma^{-1}(1/(n^{2\alpha-2\rho\beta}))}\|  \|Dh(\theta^{*})-Dh^{n^{\beta}}(\theta^{*,n^{\beta}})\| =0$. Hence, we conclude that
$$
n^{\alpha} \sum_{k=1}^{\gamma^{-1}(1/(n^{2\alpha-2\rho\beta}))} \gamma_k  \Pi_{k+1, \gamma^{-1}(1/(n^{2\alpha-2\rho\beta}))} \left( \zeta^{n^{\beta}}_{k-1} - \zeta_{k-1} + (Dh(\theta^{*})-Dh^{n^{\beta}}(\theta^{*,n^{\beta}}))(\theta^{n^{\beta}}_{k-1}-\theta^{*,n^{\beta}})\right) \overset{L^{1}(\P)}{\longrightarrow} 0. 
$$

\noindent \textbf{Step 3: study of the sequence $\left\{n^{\alpha} \sum_{k=1}^{\gamma^{-1}(1/(n^{2\alpha-2\rho\beta}))} \gamma_k  \Pi_{k+1, \gamma^{-1}(1/(n^{2\alpha-2\rho\beta}))} \Delta R^{n^{\beta}}_k, n\geq0 \right\}$}

Regarding the third term of \eqref{rec:two:level:sr}, namely $\sum_{k=1}^{n} \gamma_k  \Pi_{k+1, n}  \Delta R^{n^{\beta}}_{k}$, we decompose it as follows
\begin{align*}
\sum_{k=1}^{n} \gamma_k  \Pi_{k+1, n}  \Delta R^{n^{\beta}}_{k} & =  \sum_{k=1}^{n} \gamma_k  \Pi_{k+1, n} ( h^{n^{\beta}}(\theta^{n^{\beta}}_k)-h^{n^{\beta}}(\theta^{*}) - (H(\theta^{n^{\beta}}_k, (U^{n^{\beta}})^{k+1})-H(\theta^{*},(U^{n^{\beta}})^{k+1}))   ) \\
& + \sum_{k=1}^{n} \gamma_k  \Pi_{k+1, n} ( H(\theta_k, U^{k+1}) - H(\theta^*, U^{k+1}) - (h(\theta_k) - h(\theta^*) )  ) \\
& = A_n + B_n
\end{align*}

Now, using that $\E\left[ \left. H(\theta^{n^{\beta}}_k,(U^{n^{\beta}})^{k+1}) - H(\theta^{*},(U^{n^{\beta}})^{k+1}) \right| \mathcal{F}_{k} \right] = h^{n^{\beta}}(\theta^{n^{\beta}}_k)-h^{n^{\beta}}(\theta^{*})$ and \A{HLH} it follows that
\begin{align*}
\E |A_n|^2 & \leq C \sum_{k=1}^{n} \gamma^2_k  \|\Pi_{k+1, n}\|^2  ( \E|\theta^{n^{\beta}}_k - \theta^{*,n^{\beta}}|^2 + |\theta^{*,n^{\beta}} - \theta^*|^2 ) \\
 & \leq C ( \sum_{k=1}^{n} \gamma^3_k  \|\Pi_{k+1, n}\|^2  +  \sum_{k=1}^{n} \gamma^2_k  \|\Pi_{k+1, n}\|^2  |\theta^{*,n^{\beta}} - \theta^*|^2 ) \\
 & := A^1_n + A^2_n
\end{align*}

From Lemma \ref{stepseq:tech:lemme} we get $\sum_{k=1}^{n} \gamma^3_k  \|\Pi_{k+1, n}\|^2 = o(\gamma^{2\alpha/(2\alpha-2\rho\beta)}_n)$ and $\sum_{k=1}^{n} \gamma^2_k  \|\Pi_{k+1, n}\|^2  = \O(\gamma_n)$. Consequently, we derive $\lim\sup_n n^{2\alpha} A^1_{\gamma^{-1}(1/(n^{2\alpha - 2\rho \beta}))} = 0$ and $\lim\sup_n n^{2\alpha} A^2_{\gamma^{-1}(1/(n^{2\alpha-2\rho \beta}))} = 0$. Similarly using \A{HLH} and Lemma \ref{sstrongerror:tech:lemme} we derive $n^{\alpha} B_{\gamma^{-1}(1/(n^{2\alpha - 2\rho \beta}))} \overset{L^{2}(\P)}{\longrightarrow} 0$ as $ n\rightarrow +\infty$ so that
$$
n^{\alpha} \sum_{k=1}^{\gamma^{-1}(1/(n^{2\alpha - 2\rho\beta}))} \gamma_k  \Pi_{k+1, \gamma^{-1}(1/(n^{2\alpha - 2\rho\beta}))}  \Delta R^{n^{\beta}}_{k} \overset{\P}{\longrightarrow} 0, \ \ n\rightarrow +\infty.
$$

\noindent \textbf{Step 4: study of the sequence $\left\{n^{\alpha} \sum_{k=1}^{\gamma^{-1}(1/(n^{2\alpha-2\rho\beta}))} \gamma_k  \Pi_{k+1, \gamma^{-1}(1/(n^{2\alpha-2\rho\beta}))} \Delta N^{n^{\beta}}_k, n\geq0 \right\}$}

We now prove a CLT for the sequence $\left\{ n^{\alpha} \sum_{k=1}^{\gamma^{-1}(1/(n^{2\alpha-2\rho\beta}))} \gamma_k  \Pi_{k+1, \gamma^{-1}(1/(n^{2\alpha-2\rho\beta}))} \Delta N^{n^{\beta}}_{k}, \ n\geq0  \right\}$. It holds
\begin{align*}
 \sum_{k=1}^{\gamma^{-1}(1/(n^{2\alpha-2\rho\beta}))}  \E  \left| n^{\alpha} \gamma_k \Pi_{k+1, \gamma^{-1}(1/(n^{2\alpha-2\rho\beta})) } \Delta N^{n^{\beta}}_{k} \right|^{2+\delta} & \leq \sup_{n\geq1}\sup_{k\in \leftB 1, n \rightB} \E \left| n^{\rho \beta}  \Delta N^{n^{\beta}}_{k} \right|^{2+\delta}  \\
 & \times n^{(2+\delta)(\alpha - \rho \beta)}\sum_{k=1}^{\gamma^{-1}(1/(n^{2\alpha-2\rho\beta}))} \gamma^{2+\delta}_k  \|\Pi_{k+1, \gamma^{-1}(1/(n^{2\alpha-2\rho\beta}))}\|^{2+\delta} .
\end{align*}

By Lemma \ref{stepseq:tech:lemme}, we have the following bound:  $ \sum_{k=1}^{n} \gamma^{2+\delta}_k  \|\Pi_{k+1, n}\|^{2+\delta} = o(\gamma^{(2+\delta)(\alpha-\rho\beta)/(2\alpha-2\rho\beta)}(n))$ which implies
$$
\limsup_n n^{(2+\delta)(\alpha - \rho \beta)}\sum_{k=1}^{\gamma^{-1}(1/(n^{2\alpha-2\rho\beta}))} \gamma^{2+\delta}_k  \|\Pi_{k+1, \gamma^{-1}(1/(n^{2\alpha-2\rho\beta}))}\|^{2+\delta} =0.
$$
 
 Moreover simple computations lead
$$
\E\left|n^{\rho \beta} \Delta N^{n^{\beta}}_{k} \right|^{2+\delta} \leq C (|n^{\rho \beta} (h^{n^{\beta}}(\theta^*) - h(\theta^*))|^{2+\delta} + \E (n^{\rho \beta}|H(\theta^*,U^{n^{\beta}}) - H(\theta^*,U)|)^{2+\delta} ).
$$

For the first term in the above inequality we have $\sup_{n\geq1} |n^{\rho \beta} (h^{n^{\beta}}(\theta^*) - h(\theta^*))|^{2+Ê\delta} < + \infty \Leftrightarrow \alpha \geq \rho$. For the second term, using assumptions \A{HLH} and \A{HSR} we get $\sup_{n\geq1} \E\left[ (n^{\rho \beta}|H(\theta^*,U^{n^{\beta}}) - H(\theta^*,U)|)^{2+\delta}\right] < +\infty$. Hence we conclude that
$$
\sup_{n\geq1}\sup_{k\in \leftB 1, n \rightB} \E \left| n^{\rho \beta}  \Delta N^{n^{\beta}}_{k} \right|^{2+\delta} < +\infty,
$$

\noindent so that the conditional Lindeberg condition holds. Now, we focus on the conditional variance. We set 
\begin{equation}\label{cond:var:eq}
S_{n} :=  n^{2\alpha} \sum_{k=1}^{\gamma^{-1}(1/(n^{2\alpha-2\rho\beta}))} \gamma^{2}_k  \Pi_{k+1, \gamma^{-1}(1/(n^{2\alpha-2\rho\beta}))} \E_{k}[ \Delta N^{n^{\beta}}_{k} (\Delta N^{n^{\beta}}_{k})^{T}] \Pi^{T}_{k+1, \gamma^{-1}(1/(n^{2\alpha-2\rho\beta}))}, \ \ \mbox{and} \ \ V^{n^{\beta}}= U^{n^{\beta}} - U.
\end{equation}

%

A Taylor's expansion yields
\begin{align*}
n^{\rho \beta} \left(H(\theta^{*},U^{n^{\beta}}) - H(\theta^{*}, U)\right)  & =  D_xH(\theta^{*}, U) n^{\rho \beta}  V^{n^{\beta}} + \psi(\theta^{*}, U, V^{n^{\beta}}) n^{\rho\beta}  V^{n^{\beta}}
\end{align*}

\noindent with $\psi(\theta^{*}, U, V^{n^{\beta}}) \overset{\P}{\longrightarrow} 0$. From the tightness of $(n^{\rho\beta} V^{n^{\beta}})_{n\geq1}$, we get $\psi(\theta^{*}, U, V^{n^{\beta}}) n^{\rho\beta} V^{n^{\beta}} \overset{\P}{\longrightarrow} 0$ so that using Theorem \ref{weak:conv:euler} and Lemma \ref{lemme:conv:stab} yield 
$$
n^{\rho\beta}  \left(H(\theta^{*}, U^{n^{\beta}}) - H(\theta^{*}, U)\right)  \Longrightarrow  D_xH(\theta^*, U) V .
$$

Moreover, from assumptions \A{HLH} and \A{HSR} it follows that
$$
 \sup_{n\geq1} \E\left[|n^{\rho\beta}  (H(\theta^{*}, U^{n^{\beta}}) - H(\theta^{*}, U)) |^{2+\delta}\right] < + \infty,
$$

\noindent which combined with \A{HDH} imply
\begin{align*}
\E \left[  n^{\rho \beta} \left(H(\theta^{*},U^{n^{\beta}}) - H(\theta^{*}, U)\right) \right] & \rightarrow \tilde{\E}[D_xH(\theta^{*}, U) V] \\
 \E \left[n^{\rho \beta} \left(H(\theta^{*},U^{n^{\beta}}) - H(\theta^{*}, U)\right) \left(n^{\rho \beta} \left(H(\theta^{*},U^{n^{\beta}}) - H(\theta^{*},U)\right)\right)^{T} \right]& \rightarrow \tilde{\E}\left[ \left(D_xH(\theta^{*}, U) V\right) \left(D_xH(\theta^{*}, U) V\right)^{T}\right].
\end{align*}

Hence, we have
$$
\Gamma_n \rightarrow \Gamma^{*}:= \tilde{\E}\left[ \left(D_xH(\theta^{*}, U) V -\tilde{\E}[D_xH(\theta^{*}, U) V] \right) \left(D_xH(\theta^{*}, U) V -\tilde{\E}[D_xH(\theta^{*}, U) V]\right)^{T}\right]
$$

\noindent where for $n\geq1$ we set
$$
\Gamma_{n}:= n^{2\rho\beta}\E_{k}[ \Delta N^{n^{\beta}}_{k} (\Delta N^{n^{\beta}}_{k})^{T}].
$$

 Consequently, using the following decomposition
$$
\frac{1}{\gamma(n)} \sum_{k=1}^{n} \gamma^{2}_k  \Pi_{k+1, n} \Gamma_n \Pi^{T}_{k+1, n} = \frac{1}{\gamma(n)} \sum_{k=1}^{n} \gamma^{2}_k  \Pi_{k+1, n} \Gamma^{*} \Pi^{T}_{k+1, n} + \frac{1}{\gamma(n)} \sum_{k=1}^{n} \gamma^{2}_k  \Pi_{k+1, n} (\Gamma_n-\Gamma^{*}) \Pi^{T}_{k+1, n}
$$

\noindent with 
$$
\lim\sup_{n} \frac{1}{\gamma(n)} \left\| \sum_{k=1}^{n} \gamma^{2}_k  \Pi_{k+1, n} (\Gamma_n-\Gamma^{*}) \Pi^{T}_{k+1, n} \right\| \leq C \lim\sup_{n} \|\Gamma_n-\Gamma^{*}\| = 0,
$$

\noindent which is a consequence of Lemma \ref{stepseq:tech:lemme}, we clearly see that $\lim_{n} S_n= \lim_{n} \frac{1}{\gamma(n)} \sum_{k=1}^{n} \gamma^{2}_k  \Pi_{k+1, n} \Gamma^{*} \Pi^{T}_{k+1, n}$ if this latter limit exists. Let us note that $\Theta^{*}$ is the (unique) matrix $A$ solution to the Lyapunov equation: 
$$
\Gamma^{*} - (Dh(\theta^*)-\zeta I_d) A - A (Dh(\theta^*)-\zeta I_d)^{T} = 0.
$$

Following the lines of the proof of Lemma \ref{lem:conv:opt:tradeoff}, step 3, we have $S_n \overset{a.s.}{\longrightarrow} \Theta^{*}$. We leave the computational details to the reader. 
\subsection{Proof of Lemma \ref{lem:conv:aver} }
We will just prove the first assertion. The second one will readily follow. We use $C$ to denote a constant that may change from one line to the next. Using the notations of Lemma \ref{conv:lem:2step:sr}, the sequence $(\bar{z}^{n^{\beta}}_{p})_{p\in \leftB 0,n^{2\alpha-2\rho\beta} \rightB}$ can be decomposed as follows:
\begin{align*}
\bar{z}^{n^{\beta}}_{n^{2\alpha-2\rho\beta}} & = \frac{1}{n^{2\alpha-2\rho\beta}+1} \sum_{k=0}^{n^{2\alpha-2\rho\beta}} z^{n^{\beta}}_k \\
& = - (Dh(\theta^*))^{-1} \frac{1}{n^{2\alpha-2\rho\beta}+1} \left(\frac{z^{n^{\beta}}_{n^{2\alpha-2\rho\beta}+1}}{\gamma_{n^{2\alpha-2\rho\beta}+1}} - \frac{z^{n^{\beta}}_0}{\gamma_1}\right) -  (Dh(\theta^*))^{-1} \frac{1}{n^{2\alpha-2\rho\beta}+1} \sum_{k=1}^{n^{2\alpha-2\rho\beta}} \left(\frac{1}{\gamma_k} - \frac{1}{\gamma_{k+1}}  \right) z^{n^{\beta}}_k  \\
& + (Dh(\theta^*))^{-1} \frac{1}{n^{2\alpha-2\rho\beta}+1} \sum_{k=0}^{n^{2\alpha-2\rho\beta}}  (\Delta N^{n^{\beta}}_{k+1} + \Delta R^{n^{\beta}}_{k+1}) \\
& + (Dh(\theta^*))^{-1} \frac{1}{n^{2\alpha-2\rho\beta}+1} \sum_{k=0}^{n^{2\alpha-2\rho\beta}}  (\zeta^{n^{\beta}}_{k} - \zeta_{k} + (Dh(\theta^{*})-Dh^{n^{\beta}}(\theta^{*,n^{\beta}}))(\theta^{n^{\beta}}_{k}-\theta^{*,n^{\beta}})).
\end{align*}

Our aim is to study the contribution of each term in this decomposition. 

\noindent \textbf{Step 1: study of the sequence $\left\{ \frac{n^{\alpha}}{n^{2\alpha-2\rho\beta}+1} \left( \frac{z^{n^{\beta}}_{n^{2\alpha-2\rho\beta}+1}}{\gamma_{n^{2\alpha-2\rho\beta}+1}} -  \frac{z^{n^{\beta}}_0}{\gamma_1} \right), n\geq0 \right\}$}:

Using Proposition \ref{dec:aver:step} clearly yields
$$
\frac{n^{\alpha}}{n^{2\alpha-2\rho\beta}+1} \E\left|\frac{z^{n^{\beta}}_{n^{2\alpha-2\rho\beta}+1}}{\gamma_{n^{2\alpha-2\rho\beta}+1}} -  \frac{z^{n^{\beta}}_0}{\gamma_1} \right| \leq \frac{C}{(n^{\alpha-2\rho\beta} \gamma_{n^{2\alpha-2\rho\beta}+1})}(\E|\tilde{\mu}^{n^{\beta}}_{n^{2\alpha-2\rho\beta}}|  +  \E|\tilde{r}^{n^{\beta}}_{n^{2\alpha-2\rho\beta}}|) + \frac{C}{n^{\alpha-2\rho\beta}} (1+ |\theta^*-\theta^{*,n^{\beta}}|).
$$

We evaluate each term appearing in the right hand side of the last but one inequality. First we clearly have
$$
\frac{1}{(n^{\alpha-2\rho\beta}) \gamma_{n^{2\alpha-2\rho\beta}+1}} \E|\tilde{\mu}^{n^{\beta}}_{n^{2\alpha-2\rho\beta}}| \leq \frac{C}{\sqrt{(n^{2\alpha-2\rho\beta}) \gamma_{n^{2\alpha-2\rho\beta}+1}}} \rightarrow 0, \ \ \mbox{as} \ n\rightarrow +\infty,
$$

\noindent and 
$$
\frac{1}{(n^{\alpha-2\rho\beta}) \gamma_{n^{2\alpha-2\rho\beta}+1}} \E|\tilde{r}^{n^{\beta}}_{n^{2\alpha-2\rho\beta}}|  \leq \frac{C}{n^{\alpha-2\rho\beta}} \rightarrow 0,  \ \mbox{as} \ n\rightarrow +\infty.
$$

From these computations we get
$$
\frac{n^{\alpha}}{n^{2\alpha-2\rho\beta}+1} \left( \frac{z^{n^{\beta}}_{n^{2\alpha-2\rho\beta}+1}}{\gamma_{n^{2\alpha-2\rho\beta}+1}} -  \frac{z^{n^{\beta}}_0}{\gamma_1} \right) \overset{\L^{1}(\P)}{\longrightarrow} 0, \ \ \mbox{as } \ n\rightarrow + \infty. 
$$

\noindent \textbf{Step 2: study of the sequence $\left\{ \frac{n^{\alpha}}{n^{2\alpha-2\rho\beta}+1} \sum_{k=1}^{n^{2\alpha-2\rho\beta}} \left(\frac{1}{\gamma_k} - \frac{1}{\gamma_{k+1}}  \right) z^{n^{\beta}}_k , n\geq0 \right\}$}:

We use the decomposition of Proposition \ref{dec:aver:step} to derive
$$
 \frac{n^{\alpha}}{n^{2\alpha-2\rho\beta}+1} \left|\sum_{k=1}^{n^{2\alpha-2\rho\beta}} \left(\frac{1}{\gamma_k} - \frac{1}{\gamma_{k+1}}  \right) z^{n^{\beta}}_k \right| \leq \frac{n^{\alpha}}{n^{2\alpha-2\rho\beta}+1} \sum_{k=1}^{n^{2\alpha-2\rho\beta}} \left(\frac{1}{\gamma_{k+1}} - \frac{1}{\gamma_{k}}  \right) (|\tilde{\mu}^{n^{\beta}}_k| + |\tilde{r}^{n^{\beta}}_{k}| ).
$$

Then taking the expectation in the previous inequality and using that $\rho<1$ we deduce
$$
 \frac{n^{\alpha}}{n^{2\alpha-2\rho\beta}+1}  \sum_{k=1}^{n^{2\alpha - 2\rho\beta}} \left(\frac{1}{\gamma_{k+1}} - \frac{1}{\gamma_{k}}\right) \E|\tilde{\mu}^{n^{\beta}}_k| \leq \frac{C}{n^{\alpha-\rho\beta}} \sum_{k=1}^{n^{2\alpha-2\rho\beta}} \left(\frac{1}{\gamma_{k+1}} - \frac{1}{\gamma_{k}}\right) \gamma^{\frac12}_{k} \rightarrow 0.
$$

For the second term, we have
$$
 \frac{n^{\alpha}}{n^{2\alpha-2\rho\beta}+1} \sum_{k=1}^{n^{2\alpha-2\rho\beta}} \left(\frac{1}{\gamma_{k+1}} - \frac{1}{\gamma_{k}}\right) \E|\tilde{r}^{n^{\beta}}_{k}| \leq  \frac{ C }{n^{\alpha-2\rho\beta}} \sum_{k=1}^{n^{2\alpha-2\rho\beta}} \left(\frac{1}{\gamma_{k+1}} - \frac{1}{\gamma_{k}}\right) \gamma_k \rightarrow 0 ,
$$


\noindent since $\alpha>2\rho\beta$ which in turn implies
$$
 \frac{n^{\alpha}}{n^{2\alpha-\beta}+1} \sum_{k=1}^{n^{2\alpha-2\rho\beta}} \left(\frac{1}{\gamma_k} - \frac{1}{\gamma_{k+1}}  \right) z^{n^{\beta}}_k \overset{L^{1}(\P)}{\longrightarrow} 0.
$$

\noindent \textbf{Step 3: study of the sequence $\left\{  \frac{n^{\alpha}}{n^{2\alpha-2\rho\beta}+1} \sum_{k=1}^{n^{2\alpha-2\rho\beta}}  (\zeta^{n^{\beta}}_{k} - \zeta_{k} + (Dh(\theta^{*})-Dh^{n^{\beta}}(\theta^{*,n^{\beta}}))(\theta^{n^{\beta}}_{k}-\theta^{*,n^{\beta}})), n\geq0 \right\}$}:

Now we focus on the last term. We firstly note that thanks to Lemma \ref{sstrongerror:tech:lemme} we clearly have
$$
 \frac{n^{\alpha}}{n^{2\alpha-2\rho\beta}+1}  \E\left| \sum_{k=0}^{n^{2\alpha-2\rho\beta}}  (\zeta^{n^{\beta}}_{k} - \zeta_k) \right| \leq \frac{C}{n^{\alpha-2\rho\beta}}\sum_{k=0}^{n^{2\alpha-2\rho\beta}} \gamma_k \rightarrow 0 
$$ 

\noindent since $a > \alpha/(2\alpha-2\rho\beta)$. Now since $Dh^{n^{\beta}}$ is Lipschitz-continuous uniformly in $n$ we easily get
$$
 \frac{n^{\alpha}}{n^{2\alpha-2\rho\beta}+1} \E\left| \sum_{k=0}^{n^{2\alpha-2\rho\beta}} (Dh(\theta^{*})-Dh^{n^{\beta}}(\theta^{*,n^{\beta}}))(\theta^{n^{\beta}}_{k}-\theta^{*,n^{\beta}}) \right| \leq  \frac{C}{n^{\alpha-2\rho\beta}}(\|Dh(\theta^*)-Dh^{n^{\beta}}(\theta^*)\| + |\theta^*-\theta^{*,n^{\beta}}|) \sum_{k=0}^{n^{2\alpha-2\rho\beta}} \gamma^{\frac12}_k,
$$

\noindent and recalling that $n^{\alpha - (\alpha-\rho\beta)a} \|Dh(\theta^*)-Dh^{n^{\beta}}(\theta^*)\| \rightarrow 0$ and $a > \alpha (1 - \beta)/(\alpha-\rho\beta)$ which implies $n^{\alpha - (\alpha-\rho\beta)a} |\theta^*-\theta^{*,n^{\beta}}| \rightarrow 0$ we deduce
$$
 \frac{n^{\alpha}}{n^{2\alpha-2\rho\beta}+1} \sum_{k=0}^{n^{2\alpha-2\rho\beta}} (Dh(\theta^{*})-Dh^{n^{\beta}}(\theta^{*,n^{\beta}}))(\theta^{n^{\beta}}_{k}-\theta^{*,n^{\beta}})  \overset{L^{1}(\P)}{\longrightarrow} 0.
$$

%

 \noindent \textbf{Step 4: study of the sequence $\left\{  \frac{n^{\alpha}}{n^{2\alpha-2\rho\beta}+1} \sum_{k=0}^{n^{2\alpha-2\rho\beta}}  (\Delta N^{n^{\beta}}_{k+1} + \Delta R^{n^{\beta}}_{k+1}), n\geq0 \right\}$}:
 
Similarly to the proof of Lemma \ref{conv:lem:2step}, we decompose the sequence $\left\{ \frac{n^{\alpha}}{n^{2\alpha-2\rho\beta}+1} \sum_{k=1}^{n^{2\alpha-2\rho\beta}} \Delta R^{n^{\beta}}_{k}, n\geq1 \right\}$ as follows
\begin{align*}
 \frac{n^{\alpha}}{n^{2\alpha-2\rho\beta}+1} \sum_{k=0}^{n^{2\alpha-2\rho\beta}} \Delta R^{n^{\beta}}_{k} & =   \frac{n^{\alpha}}{n^{2\alpha-2\rho\beta}+1} \sum_{k=0}^{n^{2\alpha-2\rho\beta}} ( h^{n^{\beta}}(\theta^{n^{\beta}}_k)-h^{n^{\beta}}(\theta^{*}) - (H(\theta^{n^{\beta}}_k,(U^{n^{\beta}})^{k+1})-H(\theta^{*},(U^{n^{\beta}})^{k+1}))   ) \\
& +  \frac{n^{\alpha}}{n^{2\alpha-2\rho\beta}+1} \sum_{k=0}^{n^{2\alpha-\beta}T}  ( H(\theta_k, U^{k+1}) - H(\theta^*, U^{k+1}) - (h(\theta_k) - h(\theta^*) )  ) \\
& = A_n + B_n.
\end{align*}

From the Cauchy-Schwarz inequality and Lemma \ref{sstrongerror:tech:lemme} it easily follows
$$
\E|A_n| +\E|B_n| \leq   \frac{1}{n^{\alpha-2\rho\beta}} \left(\sum_{k=0}^{n^{2\alpha-2\rho\beta}}\gamma_k\right)^{\frac12} \rightarrow 0
$$

\noindent since $a > \alpha/(2\alpha-2\rho\beta)> \rho\beta/(\alpha-\rho\beta) $ so that
%
%
$$
\frac{n^{\alpha}}{n^{2\alpha-2\rho\beta}+1} \sum_{k=0}^{n^{2\alpha-2\rho\beta}} \Delta R^{n^{\beta}}_{k} \overset{\P}{\longrightarrow} 0, \ \ n\rightarrow +\infty.
$$

We now prove a CLT for the sequence $\left\{ \frac{n^{\alpha}}{n^{2\alpha-2\rho\beta}+1} \sum_{k=0}^{n^{2\alpha-2\rho\beta}} \Delta N^{n^{\beta}}_{k}, \ n\geq0  \right\}$. We first note 
\begin{align*}
 \sum_{k=0}^{n^{2\alpha-2\rho\beta}} \E  \left| \frac{n^{\alpha}}{n^{2\alpha-2\rho\beta}+1} \Delta N^{n^{\beta}}_{k} \right|^{2+\delta} & \leq \sup_{n\geq1} \sup_{k\in \leftB 0, n \rightB} \E \left| n^{\rho\beta}  \Delta N^{n^{\beta}}_{k} \right|^{2+\delta} \frac{1}{n^{\alpha \delta- \rho\beta\delta}} \rightarrow 0 \ \ n\rightarrow +\infty
\end{align*}

\noindent where we used assumptions \A{HLH} and \A{HSR} to derive that $\sup_{n\geq1} \sup_{k\in \leftB 1, n \rightB} \E \left| n^{\rho\beta}  \Delta N^{n^{\beta}}_{k} \right|^{2+\delta} < +\infty$. Therefore the conditional Lindeberg condition is satisfied. Then we examine the conditional variance. 
Recall that (see the the proof of Lemma \ref{conv:lem:2step}) we have 
\begin{align*}
n^{2\rho\beta} \E_{k}[\Delta N^{n^{\beta}}_{k} (\Delta N^{n^{\beta}}_{k})^{T}] & = n^{2\rho\beta} \E\left[(H(\theta^*,U^{n^{\beta}}) - H(\theta^*, U) - (h^{n^{\beta}}(\theta^*)-h(\theta^*))) \right. \\
& \left.  \ \ \ \ \ \ \times (H(\theta^*, U^{n^{\beta}}) - H(\theta^*, U) - (h^{n^{\beta}}(\theta^*)-h(\theta^*)))^{T} \right] \\
& \rightarrow \tilde{\E}\left[ \left(D_xH(\theta^{*}, U) V - \tilde{\E}[D_xH(\theta^{*}, U) V]\right) \left(D_xH(\theta^{*}, U) V- \tilde{\E}[D_xH(\theta^{*}, U) V]\right)^{T}\right],
\end{align*}

\noindent so that if we set
\begin{align*}
S_{n} & :=  \frac{n^{2\alpha} }{(n^{2\alpha-2\rho\beta}+1)^2} \sum_{k=0}^{n^{2\alpha-2\rho\beta}} \E_{k}[ \Delta N^{n^{\beta}}_{k} (\Delta N^{n^{\beta}}_{k})^{T}]  \\
& = \frac{n^{2\alpha-2\rho\beta} }{n^{2\alpha-2\rho\beta}+1} n^{2\rho\beta} \E\left[(H(\theta^*,U^{n^{\beta}}) - H(\theta^*, U) - (h^{n^{\beta}}(\theta^*)-h(\theta^*)))(H(\theta^*,U^{n^{\beta}}) - H(\theta^*, U) - (h^{n^{\beta}}(\theta^*)-h(\theta^*)))^{T} \right], \\ 
\end{align*}

\noindent we clearly get 
$$
S_n \rightarrow \tilde{\E} \left(D_xH(\theta^{*}, U) V\right) \left(D_xH(\theta^{*}, U) V\right)^{T}.
$$

\noindent This completes the proof.
\subsection{Proof of Lemma \ref{conv:lem:2step}}
We come back to the decomposition used in the proof of Lemma \ref{lem:conv:opt:tradeoff}. We consequently use the same notations. We will not go into all computational details. We deal with the case $\rho \in (0,1/2)$. The case $\rho=1/2$ can be handled in a similar fashion.

We first write for $p\geq0$
\begin{align*}
\theta^{m^{\ell}}_{p+1}-\theta^{*,m^{\ell}} & = \theta^{m^{\ell}}_{p} - \theta^{*,m^{\ell}} - \gamma_{p+1}  Dh^{m^{\ell}}(\theta^{*,m^{\ell}}) (\theta^{m^{\ell}}_p-\theta^{*,m^{\ell}})  + \gamma_{p+1} \Delta M^{m^{\ell}}_{p+1} + \gamma_{p+1} \zeta^{m^{\ell}}_p 
\end{align*}

\noindent with $\Delta M^{m^{\ell}}_{p+1} = h^{m^{\ell}}(\theta^{m^{\ell}}_p) - H(\theta^{m^{\ell}}_p,(X^{m^{\ell}}_T)^{p+1})$ and $\zeta^{m^{\ell}}_p = \O(|\theta^{m^{\ell}}_{p+1}-\theta^{*,m^{\ell}}|^2)$, $p\geq0$. Therefore, defining $z^{\ell}_{p} = \theta^{m^{\ell}}_{p} - \theta^{m^{\ell-1}}_{p}  - (\theta^{*,m^{\ell}} - \theta^{*,m^{\ell-1}})$, $p\geq0$, with $z^{\ell}_{0} = \theta^{m^{\ell}}_0 - \theta^{m^{\ell}}_0 - (\theta^{*,m^{\ell}}- \theta^{*,m^{\ell-1}})$, by a simple induction argument one has
\begin{align}
z^{\ell}_{M_{\ell}} & = \Pi_{1, M_{\ell}} z^{\ell}_{0}  +  \sum_{k=1}^{M_{\ell}} \gamma_k  \Pi_{k+1, M_{\ell}} \Delta N^{\ell}_{k}  +   \sum_{k=1}^{M_{\ell}} \gamma_k  \Pi_{k+1, M_{\ell}}  \Delta R^{\ell}_{k} \nonumber \\
& +  \sum_{k=1}^{M_{\ell}} \gamma_k  \Pi_{k+1, M_{\ell}} \left( \zeta^{\ell}_{k-1} - \zeta^{\ell-1}_{k-1} + (Dh(\theta^{*})-Dh^{m^{\ell}}(\theta^{*,m^{\ell}}))(\theta^{m^{\ell}}_{k-1}-\theta^{*,m^{\ell}}) \right. \nonumber \\
& \left. - (Dh(\theta^{*})-Dh^{m^{\ell-1}}(\theta^{*,m^{\ell-1}}))(\theta^{m^{\ell-1}}_{k-1}-\theta^{*,m^{\ell-1}}) \right) \label{rec:two:level}
\end{align}

\noindent where $\Pi_{k, n}:= \prod_{j=k}^{n} \left(I_d - \gamma_{j} Dh(\theta^{*})\right)$, with the convention that $\Pi_{n+1, n} = I_d$, and $\Delta N^{\ell}_{k} := h^{m^{\ell}}(\theta^{*}) - h^{m^{\ell-1}}(\theta^{*}) - (H(\theta^{*},(U^{m^{\ell}})^{k+1}) - H(\theta^{*}, (U^{m^{\ell-1}})^{k+1}))$, $\Delta R^{\ell}_{k} = h^{m^{\ell}}(\theta^{m^{\ell}}_k)-h^{m^{\ell}}(\theta^{*}) - (H(\theta^{m^{\ell}}_k,(U^{m^{\ell}})^{k+1}) - H(\theta^{*}, (U^{m^{\ell}})^{k+1}))  + H(\theta^{m^{\ell-1}}_k,(U^{m^{\ell-1}})^{k+1}) - H(\theta^*,(U^{m^{\ell-1}})^{k+1}) - (h^{m^{\ell-1}}(\theta^{m^{\ell-1}}_k) - h^{m^{\ell-1}}(\theta^{*}) ) $ for $k\geq0$. We follow the same methodology developed so far and quantify the contribution of each term. Once again the weak rate of convergence will be ruled by the second term which involves the non-linearity in the innovation variable $(U^{m^{\ell-1}},U^{m^{\ell}})$, for which we prove a CLT.

\noindent \textbf{Step 1: study of $\left\{n^{\alpha}\sum_{\ell=1}^{L} \Pi_{1, M_{\ell}} z^{\ell}_{0}, n\geq0 \right\}$}

Under the assumptions on the step sequence $\gamma$, for all $\eta \in (0, \lambda_m)$ we have $\|\Pi_{1, M_{\ell}}\| \leq \exp(-(\lambda_m-\eta)\sum_{k=1}^{M_{\ell}} \gamma_k) = C m^{\ell \frac{(1+2\rho)}{2} (\lambda_m-\eta)\gamma_0}/((n^{2\alpha+ \frac{1-2\rho}{2}})^{(\lambda_m-\eta)\gamma_0} )$ if $\gamma(p)=\gamma_0/p$ or $\|\Pi_{1, M_{\ell}}\| = \O(\gamma(M_\ell))$ otherwise. Therefore,  if $\gamma(p)=\gamma_0/p$ we select $\eta>0$ such that $\gamma_0(\lambda_m-\eta)>\alpha/(\alpha-2\rho)$ then one has
\begin{align*}
\E\left| n^{\alpha} \sum_{\ell=1}^{L} \Pi_{1, M_{\ell}} z^{\ell}_{0} \right| & \leq C n^{\alpha} \sum_{\ell=1}^{L}  \|\Pi_{1, M_{\ell}}\| \leq  \frac{C}{n^{(\lambda_m-\eta)\gamma_0 (2\alpha+ \frac{1-2 \rho}{2})-\alpha} } \sum_{\ell=1}^{L} m^{\ell(\lambda_m-\eta)\gamma_0 \frac{1+2\rho}{2}} \leq \frac{C}{n^{2(\lambda_m-\eta)\gamma_0(\alpha- 2\rho)-\alpha}} \rightarrow 0
\end{align*}

\noindent as $n\rightarrow + \infty$. Otherwise one has
$$
\E\left| n^{\alpha} \sum_{\ell=1}^{L} \Pi_{1, M_{\ell}} z^{\ell}_{0} \right| \leq C n^{\alpha} \sum_{\ell =1}^{L} \gamma(M_\ell) \leq C \frac{n^{\frac{1+2\rho}{2}}}{n^{\alpha + \frac{1-2\rho}{2}}}=\frac{C}{n^{\alpha-2 \rho}}\rightarrow 0.
$$

%

\noindent \textbf{Step 2: study of  $\left\{n^{\alpha}  \sum_{\ell=1}^{L} \sum_{k=1}^{M_{\ell}} \gamma_k  \Pi_{k+1, M_{\ell}} \left( \zeta^{\ell}_{k-1} - \zeta^{\ell-1}_{k-1} \right), n\geq0 \right\}$}

By Lemma \ref{sstrongerror:tech:lemme}, one has
\begin{align*}
\E \left| n^{\alpha} \sum_{\ell=1}^{L} \sum_{k=1}^{M_{\ell}} \gamma_k  \Pi_{k+1, M_{\ell}} \left( \zeta^{\ell}_{k-1} - \zeta^{\ell-1}_{k-1} \right)  \right| & \leq C n^{\alpha} \sum_{\ell=1}^{L} \sum_{k=1}^{M_{\ell}} \gamma^{2}_k  \|\Pi_{k+1, M_{\ell}}\|.
\end{align*}

\noindent However, by Lemma \ref{stepseq:tech:lemme} (if $\gamma(p)=\gamma_0/p$ recall that $\lambda_m \gamma_0 > 1$) we easily derive $
\limsup_{n} \frac{1}{\gamma(n)} \sum_{k=1}^{n}  \gamma^{2}_k  \|\Pi_{k+1, n}\| \leq  1$, so that
\begin{align*}
n^{\alpha} \sum_{\ell=1}^{L} \sum_{k=1}^{M_{\ell}} \gamma^{2}_k  \|\Pi_{k+1, M_{\ell}}\| & \leq C n^{\alpha} \sum_{\ell=1}^{L} \gamma(M_\ell) \rightarrow 0, \  \ n\rightarrow + \infty.
\end{align*}

\noindent \textbf{Step 3: study of  $\left\{n^{\alpha} \sum_{\ell=1}^{L} \sum_{k=1}^{M_{\ell}} \gamma_k  \Pi_{k+1, M_{\ell}} \left( (Dh(\theta^{*})-Dh^{m^{\ell}}(\theta^{*,m^{\ell}}))(\theta^{m^{\ell}}_{k-1}-\theta^{*,m^{\ell}}) \right), n\geq0 \right\}$ \\ 
and $\left\{ n^{\alpha} \left( \sum_{\ell=1}^{L} \sum_{k=1}^{M_{\ell}} \gamma_k  \Pi_{k+1, M_{\ell}} (Dh(\theta^{*})-Dh^{m^{\ell-1}}(\theta^{*,m^{\ell-1}}))(\theta^{m^{\ell-1}}_{k-1}-\theta^{*,m^{\ell-1}}) \right), n\geq0 \right\}$}

By Lemma \ref{sstrongerror:tech:lemme} and since $Dh^{m^{\ell}}$ is a Lipschitz function uniformly in $m$ we clearly have
\begin{align*}
\E\left| n^{\alpha} \sum_{\ell=1}^{L} \sum_{k=1}^{M_{\ell}} \gamma^{3/2}_k \Pi_{k+1, M_{\ell}}  (Dh(\theta^{*})-Dh^{m^{\ell}}(\theta^{*,m^{\ell}}))(\theta^{m^{\ell}}_{k-1}-\theta^{*,m^{\ell}}) \right| & \leq  n^{\alpha} \sum_{\ell=1}^{L} \sum_{k=1}^{M_{\ell}} \gamma^{3/2}_k \|\Pi_{k+1, n}\|  \\ 
& \ \ \ \ \ \ \  \times (\|Dh(\theta^{*})-Dh^{m^{\ell}}(\theta^{*})\| +  |\theta^{*,m^{\ell}} -\theta^{*})|) \\
& \leq C n^{\alpha} \sum_{\ell=1}^{L} \gamma^{1/2}(M_{\ell}) (\|Dh(\theta^{*})-Dh^{m^{\ell}}(\theta^{*})\| +  |\theta^{*,m^{\ell}} -\theta^{*})|)
\end{align*}

\noindent which combined with $\sup_{n\geq1}n^{\beta} \|Dh(\theta^{*})-Dh^{n}(\theta^{*})\|<+\infty$ with $\beta>\rho$ and $\sup_{n\geq1} n^{\alpha} |\theta^{*,n} -\theta^{*}| < +\infty$ imply that 
\begin{align*}
\E\left| n^{\alpha} \sum_{\ell=1}^{L} \sum_{k=1}^{M_{\ell}} \gamma_k \Pi_{k+1, M_{\ell}}  (Dh(\theta^{*})-Dh^{m^{\ell}}(\theta^{*,m^{\ell}}))(\theta^{m^{\ell}}_{k-1}-\theta^{*,m^{\ell}}) \right| & \leq \frac{C}{n^{\frac{(1-2\rho)}{4}}} \sum_{\ell=1}^{L} m^{\ell\frac{(1+2\rho)}{4}}(m^{-\ell\alpha} + m^{-\ell \beta})\\
& \leq C(n^{\rho-\alpha}+n^{\rho-\beta})
\end{align*}

\noindent so that $n^{\alpha} \sum_{\ell=1}^{L} \sum_{k=1}^{M_{\ell}} \gamma_k \Pi_{k+1, M_{\ell}}  (Dh(\theta^{*})-Dh^{m^{\ell}}(\theta^{*,m^{\ell}}))(\theta^{m^{\ell}}_{k-1}-\theta^{*,m^{\ell}}) \overset{L^{1}(\P)}{\longrightarrow} 0$. By similar arguments, we easily deduce $n^{\alpha} \sum_{\ell=1}^{L} \sum_{k=1}^{M_{\ell}} \gamma_k \Pi_{k+1, M_{\ell}}  (Dh(\theta^{*})-Dh^{m^{\ell-1}}(\theta^{*,m^{\ell-1}}))(\theta^{m^{\ell-1}}_{k-1}-\theta^{*,m^{\ell-1}})\overset{L^{1}(\P)}{\longrightarrow} 0$.

\noindent \textbf{Step 4: study of $\left\{n^{\alpha} \sum_{\ell=1}^{L}\sum_{k=1}^{M_{\ell}} \gamma_k  \Pi_{k+1, M_{\ell}}  \Delta R^{\ell}_{k}, n\geq0 \right\}$}

Using the Cauchy-Schwarz inequality we deduce
\begin{align*}
\E\left| n^{\alpha} \sum_{\ell=1}^{L}\sum_{k=1}^{M_{\ell}} \gamma_k  \Pi_{k+1, M_{\ell}}  \Delta R^{\ell}_{k} \right|^2 & \leq  n^{\alpha} \sum_{\ell=1}^{L} \left(\sum_{k=1}^{M_{\ell}} \gamma^{2}_k  \|\Pi_{k+1, M_{\ell}}\|^2 \E|H(\theta^{m^{\ell}}_k, (U^{m^{\ell}})^{k+1}) - H(\theta^{*},(U^{m^{\ell}})^{k+1})|^2\right)^{1/2}\\
& +  n^{\alpha} \sum_{\ell=1}^{L} \left(\sum_{k=1}^{M_{\ell}} \gamma^{2}_k  \|\Pi_{k+1, M_{\ell}}\|^2 \E|H(\theta^{m^{\ell-1}}_k,(U^{m^{\ell-1}})^{k+1}) - H(\theta^{*},(U^{m^{\ell-1}})^{k+1})|^2\right)^{1/2} \\
& \leq C n^{\alpha} \sum_{\ell=1}^{L} \left(\sum_{k=1}^{M_{\ell}} \gamma^{3}_k  \|\Pi_{k+1, M_{\ell}}\|^2\right)^{1/2} 
\end{align*}

\noindent where we used \A{HLH} and Lemma \ref{sstrongerror:tech:lemme}. Now from Lemma \ref{stepseq:tech:lemme} and simple computations it follows
$$
n^{\alpha} \sum_{\ell=1}^{L} \left( \sum_{k=1}^{M_{\ell}} \gamma^{3}_k  \|\Pi_{k+1, M_{\ell}}\|^2\right)^{1/2} \leq C n^{\alpha} \sum_{\ell=1}^{L} \gamma(M_\ell) \rightarrow 0, \ \ n\rightarrow + \infty.
$$


Therefore, we conclude that
$$
n \sum_{\ell=1}^{L}\sum_{k=1}^{M_{\ell}} \gamma_k  \Pi_{k+1, M_{\ell}}  \Delta R^{\ell}_{k} \overset{L^{2}(\P)}{\longrightarrow}0, \ \ n\rightarrow + \infty.
$$

\noindent \textbf{Step 5: study of $\left\{n^{\alpha} \sum_{\ell=1}^{L}\sum_{k=1}^{M_{\ell}} \gamma_k  \Pi_{k+1, M_{\ell}}  \Delta N^{\ell}_{k}, n\geq0 \right\}$}

We now prove a CLT for the sequence $\left\{ n^{\alpha} \sum_{\ell=1}^{L}\sum_{k=1}^{M_{\ell}} \gamma_k  \Pi_{k+1, M_{\ell}}  \Delta N^{\ell}_{k}, \ n\geq0  \right\}$. By Burkholder's inequality and elementary computations, it holds
\begin{align*}
\sum_{\ell=1}^{L}\E   \left| \sum_{k=1}^{M_{\ell}}  n^{\alpha} \gamma_k \Pi_{k+1, M_{\ell}} \Delta N^{\ell}_{k} \right|^{2+\delta} & \leq C n^{(2+\delta)\alpha} \sum_{\ell=1}^{L} \E\left(\sum_{k=1}^{M_{\ell}} \gamma^{2}_k \|\Pi_{k+1,M_{\ell}}\|^{2} |\Delta N^{\ell}_{k}|^2   \right)^{1+\delta/2} \\
& \leq C n^{(2+\delta)\alpha} \sum_{\ell=1}^{L} (\sum_{k=1}^{M_{\ell}} \gamma^{2}_k \|\Pi_{k+1,M_{\ell}}\|^{2})^{\delta/2} \sum_{k=1}^{M_{\ell}} \gamma^{2+\delta}_k \|\Pi_{k+1,M_{\ell}}\|^{2+\delta} \E|\Delta N^{\ell}_{k}|^{2+\delta}.
\end{align*}

Using \A{HLH} and \A{HSR} we have $\sup_{\ell\geq1} \E (m^{\rho\ell}|H(\theta^*,U^{m^{\ell}}) - H(\theta^*, U)|)^{2+\delta} < +\infty$ so that
$$
\E|\Delta N^{\ell}_{k}|^{2+\delta} \leq \frac{K}{m^{\ell (2\rho+\rho \delta)}}.
$$

Moreover, by Lemma \ref{stepseq:tech:lemme}, we have 
$$
\lim\sup_{n} (1/\gamma^{(1+\delta)}(n)) \sum_{k=1}^{n} \gamma^{2+\delta}_k \|\Pi_{k+1, n}\|^{2+\delta} \leq 1\ \ \mbox{and} \ \ \lim\sup_{n} (1/\gamma(n)) \sum_{k=1}^{n} \gamma^{2}_k \|\Pi_{k+1, n}\|^{2} \leq 1.
$$

\noindent Consequently we deduce
\begin{align*}
\sum_{\ell=1}^{L}\E   \left| \sum_{k=1}^{M_{\ell}}  n^{\alpha} \gamma_k \Pi_{k+1, M_{\ell}} \Delta N^{\ell}_{k} \right|^{2+\delta} & \leq C n^{ (2+\delta)\alpha} \sum_{\ell=1}^{L} \gamma^{1+3\delta/2}(M_\ell) m^{-\ell(2\rho+\rho\delta)} \leq \frac{C}{n^{2 \alpha \delta}} n^{2\rho(1+3\delta/2)-2\rho-\rho\delta} = \frac{C}{n^{2\delta(\alpha-\rho)}}
\end{align*}

\noindent which in turn implies
$$
\sum_{\ell=1}^{L}\E   \left| \sum_{k=1}^{M_{\ell}}  n^{\alpha} \gamma_k \Pi_{k+1, M_{\ell}} \Delta N^{\ell}_{k} \right|^{2+\delta} \rightarrow 0, \ \ n\rightarrow +\infty
$$

\noindent so that the conditional Lindeberg condition is satisfied.
%
%
Now, we focus on the conditional variance. We set 
\begin{equation}\label{cond:var:eq}
S_{\ell} :=  n^{2\alpha} \sum_{k=1}^{M_{\ell}} \gamma^{2}_k  \Pi_{k+1, M_{\ell}} \E_{k}[ \Delta N^{\ell}_{k} (\Delta N^{\ell}_{k})^{T}] \Pi^{T}_{k+1, M_{\ell}}, \ \ \mbox{and} \ \ U^{\ell}= U^{m^{\ell}} - U^{m^{\ell-1}}.
\end{equation}

Observe that by the very definition of $M_{\ell}$ one has
$$
S_{\ell} = \frac{1}{\gamma(M_\ell)} (m^{\frac{1-2\rho}{2}}-1)\frac{m^{\ell\frac{(1+2\rho)}{2}}}{n^{\frac{1-2\rho}{2}}-1} \sum_{k=1}^{M_{\ell}} \gamma^{2}_k  \Pi_{k+1, M_{\ell}} \E_{k}[ \Delta N^{\ell}_{k} (\Delta N^{\ell}_{k})^{T}] \Pi^{T}_{k+1, M_{\ell}}
$$
%

A Taylor's expansion yields
\begin{align*}
H(\theta^{*}, U^{m^{\ell}}) - H(\theta^{*}, U^{m^{\ell-1}})  & = D_xH(\theta^{*}, U) U^{\ell} + \psi(\theta^{*}, U, U^{m^{\ell}} - U) (U^{m^{\ell}} -U) + \psi(\theta^{*}, U, U^{m^{\ell-1}} - U) (U^{m^{\ell-1}} - U)
\end{align*}

\noindent with $(\psi(\theta^{*}, U, U^{m^{\ell}} - U), \psi(\theta^{*}, U, U^{m^{\ell-1}} - U)) \overset{\P}{\longrightarrow} 0$ as $\ell\rightarrow +\infty$. From the tightness of the sequences $( m^{\rho\ell}(U^{m^{\ell}} - U))_{\ell\geq1}$ and $(m^{\rho \ell}(U^{m^{\ell-1}} - U))_{\ell\geq1}$, we get 
$$
m^{\rho \ell} \left(\psi(\theta^{*}, U, U^{m^{\ell}} - U) (U^{m^{\ell}} - U) + \psi(\theta^{*}, U, U^{m^{\ell-1}} - U) (U^{m^{\ell-1}} - U) \right) \overset{\P}{\longrightarrow} 0, \ \ \ell\rightarrow +\infty.
$$

Therefore using Theorem \ref{weak:conv:euler} and Lemma \ref{lemme:conv:stab} yield 
$$
m^{\rho\ell} \left(H(\theta^{*}, U^{m^{\ell}}) - H(\theta^{*}, U^{m^{\ell-1}})\right)  \Longrightarrow  D_xH(\theta^*, U) V^{m} .
$$

Moreover, from assumption \A{HLH} and \A{HRH} it follows that
$$
  \sup_{\ell \geq1} \E\left|m^{\rho \ell}(H(\theta^{*},U^{m^{\ell}}) - H(\theta^{*}, U^{m^{\ell-1}})) \right|^{2+\delta} < + \infty,
$$

\noindent which combined with \A{HDH} imply
\begin{align*}
m^{\rho\ell}\E [H(\theta^{*},U^{m^{\ell}}) - H(\theta^{*}, U^{m^{\ell-1}})] & \rightarrow \tilde{\E}[D_xH(\theta^{*}, U) V^{m}] \\
m^{2\rho\ell} \E[ (H(\theta^{*}, U^{m^{\ell}}) - H(\theta^{*}, U^{m^{\ell-1}})) (H(\theta^{*}, U^{m^{\ell}}) - H(\theta^{*},U^{m^{\ell-1}}))^{T}] & \rightarrow \tilde{\E}[ \left(D_xH(\theta^{*}, U) V^{m}\right) \left(D_xH(\theta^{*}, U) V^{m}\right)^{T}]
\end{align*}

\noindent as $\ell \rightarrow + \infty$. Hence, we have
$$
m^{2\rho\ell} \Gamma_\ell \rightarrow \Gamma^{*}:= \tilde{\E}\left[ \left(D_xH(\theta^{*}, U) V^{m} - \tilde{\E}[D_xH(\theta^{*}, U) V^{m}]\right) \left(D_xH(\theta^{*}, U) V^{m} -\tilde{\E}[D_xH(\theta^{*}, U) V^{m}]\right)^{T}\right]
$$

\noindent where for $\ell\geq1$
\begin{align*}
\Gamma_{\ell} & := \E_{k}[ \Delta N^{\ell}_{k} (\Delta N^{\ell}_{k})^{T}]  \\
& = \E[(H(\theta^{*},U^{m^{\ell}}) - H(\theta^{*}, U^{m^{\ell-1}})) (H(\theta^{*}, U^{m^{\ell}}) - H(\theta^{*}, U^{m^{\ell-1}}))^{T}] -  (h^{m^{\ell}}(\theta^{*}) - h^{m^{\ell-1}}(\theta^{*}))(h^{m^{\ell}}(\theta^{*}) - h^{m^{\ell-1}}(\theta^{*}))^{T} .
\end{align*}
 Consequently, using the following decomposition
\begin{align*}
\frac{1}{\gamma(M_{\ell})} m^{2\rho\ell} \sum_{k=1}^{M_{\ell}} \gamma^{2}_k  \Pi_{k+1, M_{\ell}} \Gamma_\ell \Pi^{T}_{k+1, M_{\ell}} & = \frac{1}{\gamma(M_{\ell})} \sum_{k=1}^{M_{\ell}} \gamma^{2}_k  \Pi_{k+1, M_{\ell}} \Gamma^{*} \Pi^{T}_{k+1, M_{\ell}} \\
& + \frac{1}{\gamma(M_{\ell})} \sum_{k=1}^{M_{\ell}} \gamma^{2}_k  \Pi_{k+1, M_\ell}  \left(m^{2\rho\ell}\Gamma_\ell-\Gamma^{*}\right) \Pi^{T}_{k+1, M_\ell}
\end{align*}

\noindent with 
$$
\lim\sup_{\ell} \frac{1}{\gamma(M_{\ell})} \left\| \sum_{k=1}^{M_{\ell}} \gamma^{2}_k  \Pi_{k+1, M_\ell}  \left( m^{2\rho \ell}\Gamma_\ell-\Gamma^{*}\right) \Pi^{T}_{k+1, M_{\ell}} \right\| \leq C \lim\sup_{\ell} \left\|m^{2\rho \ell}  \Gamma_\ell-\Gamma^{*} \right\| = 0,
$$

\noindent which is a consequence of Lemma \ref{stepseq:tech:lemme}, we clearly see that $\frac{n^{\frac{1-2\rho}{2}}-1}{m^{\ell\frac{(1-2\rho)}{2}}(m^{\frac{1-2\rho}{2}}-1)}\lim_{\ell} S_\ell= \lim_{p\rightarrow + \infty} \frac{1}{\gamma(p)} \sum_{k=1}^{p} \gamma^{2}_k  \Pi_{k+1, p} \Gamma^{*} \Pi^{T}_{k+1, p}$ if this latter limit exists. The matrix $\Theta^{*}$ defined by \eqref{asympt:cov:mlvl} is the (unique) matrix $A$ solution to the Lyapunov equation: 
$$
\Gamma^{*} - (Dh(\theta^*)-\zeta I_d) A - A (Dh(\theta^*)-\zeta I_d)^{T} = 0.
$$

Following the lines of the proof of step 3, Lemma \ref{lem:conv:opt:tradeoff}, we have $S_{\ell}\frac{(n^{\frac{1-2\rho}{2}}-1)}{m^{\ell\frac{(1-2\rho)}{2}}(m^{\frac{1-2\rho}{2}}-1)} \overset{a.s.}{\longrightarrow} \Theta^{*}$ as $\ell \rightarrow + \infty$. We leave the computational details to the reader. Finally, from Ces\`{a}ro's Lemma it follows that
$$
\sum_{\ell=1}^{L} S_{\ell} = \left(\frac{m^{\frac{1-2\rho}{2}}-1}{n^{\frac{1-2\rho}{2}}-1}\right) \sum_{\ell=1}^{L} \left(S_{\ell}\frac{(n^{\frac{1-2\rho}{2}}-1)}{m^{\ell\frac{(1-2\rho)}{2}}(m^{\frac{1-2\rho}{2}}-1)}\right) m^{\ell\frac{(1-2\rho)}{2}} \overset{a.s.}{\underset{n\rightarrow + \infty}{\longrightarrow}} \Theta^{*}.
$$

 \section{Numerical Results}\label{num:res:sec}
In this section we illustrate the results obtained in Section \ref{gen:frame:sec}.

\subsection{Computation of quantiles of a one dimensional diffusion process}\label{quant:est:sec}
We first consider the problem of the computation of a quantile at level $l \in (0,1)$ of a one dimensional diffusion process. This quantity, also referred as the Value-at-Risk at level $l$ in the practice of risk management, is the lowest amount not exceeded by $X_T$ with probability $l$, namely
$$
q_{l}(X_T):= \inf \left\{ \theta : \P(X_T \leq \theta) \geq l \right\}.
$$

To illustrate the results of sections \ref{implicit:discret:sec} and \ref{opt:tradeoff:sec}, we consider a simple geometric Brownian motion
\begin{equation}
\label{geom:brow:mot}
X_t = x  + \int_0^t r X_s ds + \int_0^t \sigma X_s dW_s, \ \ t\in [0,T]    
\end{equation}

\noindent for which the quantile is explicitly known at any level $l$. Hence we have $\rho=1/2$. The distribution function of $X_T$ being increasing, $q_{l}(X_T)$ is the unique solution of the equation $h(\theta)=\E_x[H(\theta,X_T)] = 0$ with $H(\theta,x) = \mbox{\textbf{1}}_{\left\{x \leq \theta \right\}} - l$. A simple computation shows that 
 $$
 q_{l}(X_T) = x_0 \exp((r-\sigma^2/2)T + \sigma\sqrt{T} \phi^{-1}(l))
 $$

\noindent where $\phi$ is the distribution function of the standard normal distribution $\mathcal{N}(0,1)$. We associate to the SDE \eqref{geom:brow:mot} its Euler like scheme $X^{n}=(X^{n}_{t})_{t \in [0,T]}$ with time step $\Delta=T/n$. We use the following values for the parameters: $x=100, \ r=0.05, \ \sigma=0.4, \ T=1, \ l = 0.7$. The reference Black-Scholes quantile is $q_{0.7}(X_T) = 119.69$.

\begin{REM}Let us note that when $l$ is close to $0$ or $1$ (usually less than $0.05$ or more than $0.95$) the convergence of the considered SA algorithm is slow and chaotic. This is mainly due to the fact that the procedure obtains few significant samples to update the estimate in this rare event situation. One solution is to combine it with a variance reduction algorithm such as an adaptive importance sampling procedure that will generate more samples in the area of interest, see e.g. \cite{bar:fri:pag:09} and \cite{bar:fri:pag:09:2}.
\end{REM}

 In order to illustrate the result of Theorem \ref{thm:conv:disc}, we plot in Figure \ref{Conv:WeakDisc:Theta:Quantile} the behaviors of $nh^{n}(\theta^*)$ and $n(\theta^{*,n}-\theta^*)$ for $n=100, \cdots,500$. Actually, $h^{n}(\theta^*)$ is approximated by its Monte Carlo estimator and $\theta^{*,n}$ is estimated by $\theta^{n}_{M}$, both estimators being computed with $M=10^{8}$ samples. The variance of the Monte Carlo estimator ranges from $2102.4$ for $n=100$ to $53012.5$ for $n=500$. We set $\gamma_p=\gamma_0/p$ with $\gamma_0=200$. We clearly see that $nh^{n}(\theta^*)$ and $n(\theta^{*,n}-\theta^*)$ are stable with respect to $n$. The histogram of Fig \ref{Hist:SA:quantile} illustrates Theorem \ref{globalTCL}. The distribution of $n(\theta^{n}_{\gamma^{-1}(1/n^2)}-\theta^*)$, obtained with $n=100$ and $N=1000$ samples, is close to a normal distribution.
 
 \begin{figure}[!ht]
\begin{center}
\includegraphics[width=8cm,height=7cm]{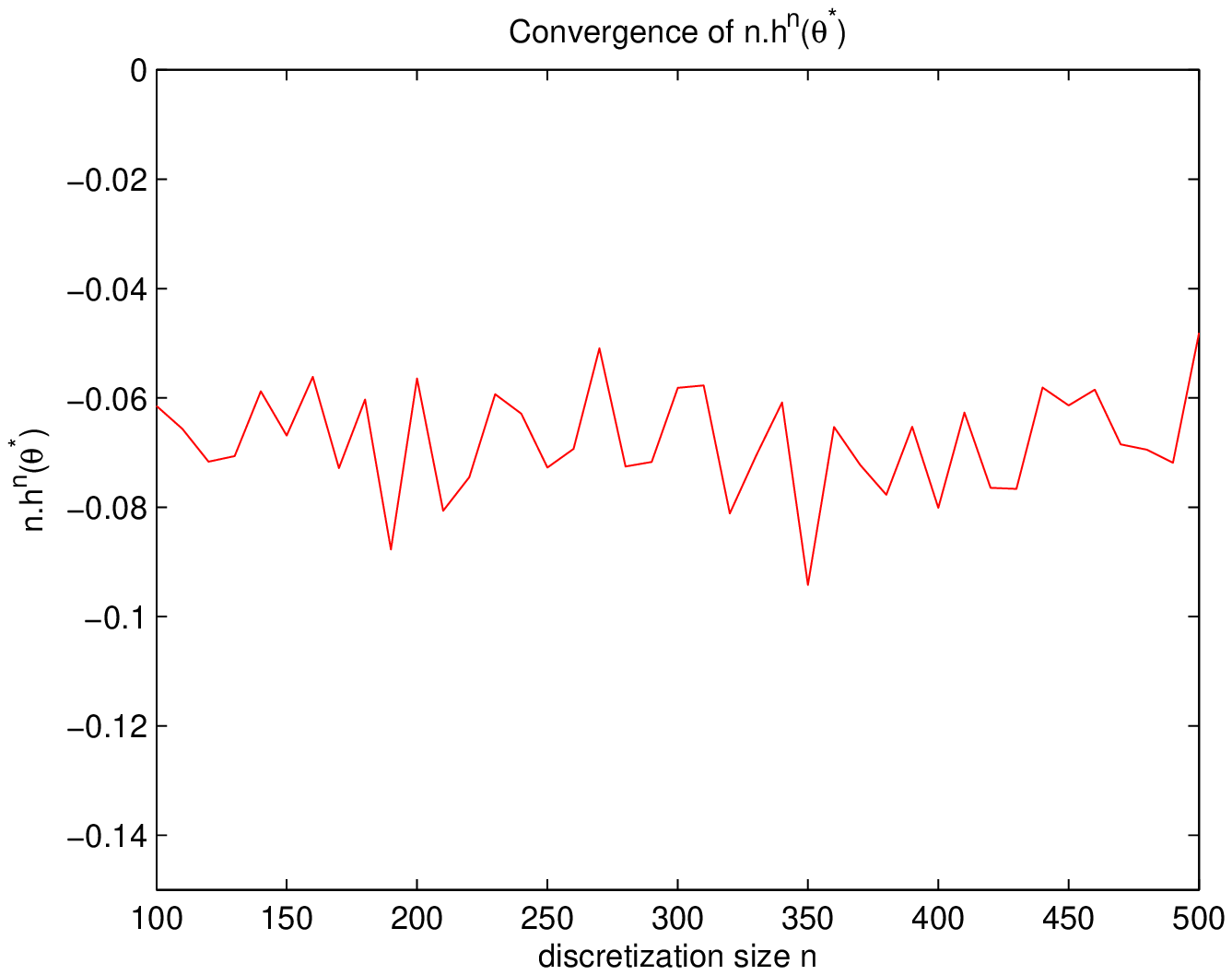}
\includegraphics[width=8cm,height=7cm]{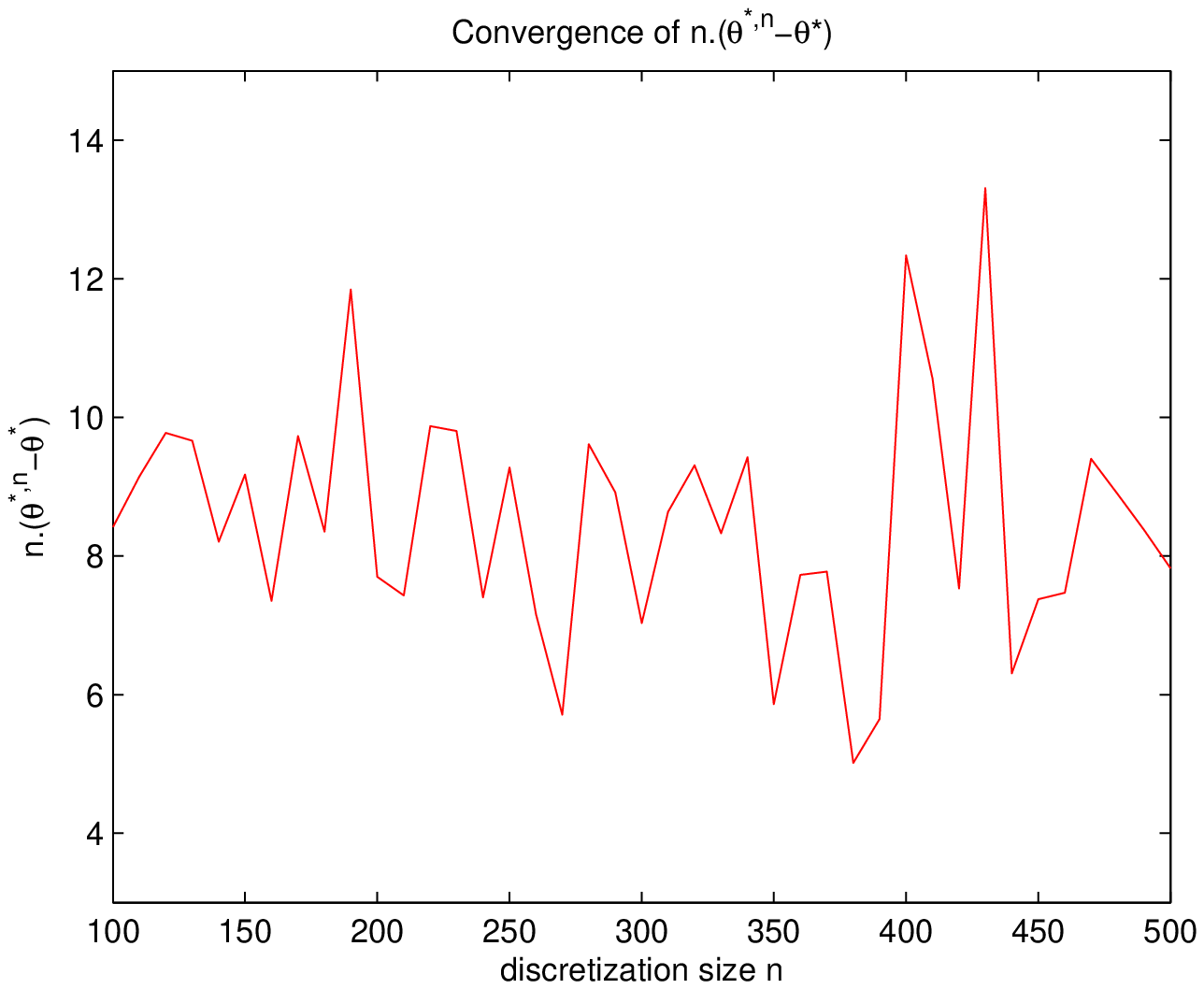}
\caption{On the left: Weak discretization error $n\mapsto n h^{n}(\theta^*)$. On the right: Implicit discretization error $n\mapsto n(\theta^{*,n}-\theta^*)$, $n=100, \cdots, 500$.}
\label{Conv:WeakDisc:Theta:Quantile}
\end{center}
\end{figure}

 \begin{figure}[!ht]
\begin{center}
\includegraphics[width=8cm,height=8cm, keepaspectratio=true]{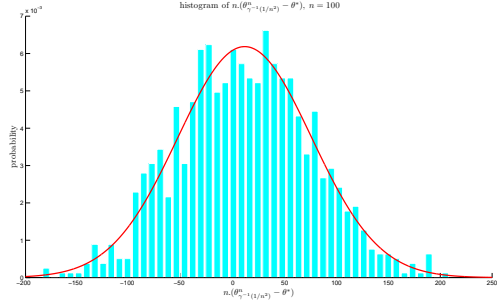}
\caption{Histogram of $n(\theta^{n}_{\gamma^{-1}(1/n^2)}-\theta^*)$, $n=100$, with $N=1000$ samples.}
\label{Hist:SA:quantile}
\end{center}
\end{figure}
 \subsection{Computation of the level of an unknown function}
 We turn our attention to the computation of the level of the function $\theta\mapsto e^{-rT}\E(X_T-\theta)_{+}$ (European call option) for which the closed-form formula under the dynamic \eqref{geom:brow:mot} is given by
 \begin{equation}
 \label{bs:formula}
e^{-rT}\E(X_T-\theta)_{+} = e^{-rT}x \phi(d_+ (x,\theta, \sigma)) -e^{-rT} \theta \phi(d_{-}(x,\theta,\sigma)),
 \end{equation}
 
 \noindent where $d_{\pm}(x,y,z)= \log(x/y)/(z\sqrt{T}) \pm z\sqrt{T}/2$. Therefore, we first fix a value $\theta^*$ (the target of our procedure) and compute the corresponding level $l=\E(X_T-\theta^*)_+$ by \eqref{bs:formula}. The values of the parameters $x, r, \sigma, T$ remain unchanged. We plot in Figure \ref{Conv:WeakDisc:Theta:Call} the behaviors of $nh^{n}(\theta^*)$ and $n(\theta^{*,n}-\theta^*)$ for $n=100, \cdots,500$. As in the previous example, $h^{n}(\theta^*)$ is approximated by its Monte Carlo estimator and $\theta^{*,n}$ is estimated by $\theta^{n}_{M}$, both estimators being computed with $M=10^{8}$ samples. The variance of the Monte Carlo estimator ranges from $9.73 \times 10^{6}$ for $n=100$ to $9.39 \times 10^7$ for $n=500$.

 To compare the three methods to approximate the solution to $h(\theta)= \E_{x}[H(\theta,X_T)]=0$ with $H(\theta,x) = l - (x-\theta)_{+}$ in terms of computational costs, we compute the different estimators, namely $\theta^{n}_{\gamma^{-1}(1/n^2)}$ where $(\theta^{n}_p)_{p\geq1}$ is given by \eqref{RM}, $\Theta^{sr}_n$ and $\Theta^{ml}_n$ for a set of $N=200$ values of the target $\theta^*$ equidistributed on the interval $[90,110]$ and for different values of $n$. For each value $n$ and for each method we compute the complexity given by \eqref{Compl:SA}, \eqref{Compl:SR-SA} and \eqref{Compl:ML-SA} respectively and the root-mean-squared error which is given by
 $$
 \text{RMSE} = \left(\frac{1}{N} \sum_{k=1}^{N} (\Theta^{n}_k - \theta^{*}_{k})^2 \right)^{1/2}
 $$  
 
 \noindent where $\Theta^{n}_k= \theta^{n}_{\gamma^{-1}(1/n^2)}, \ \Theta^{sr}_n$ or $\Theta^{ml}_n$ is the considered estimator. For each given $n$, we provide a couple $(\text{RMSE}, \text{Complexity})$ which is plotted on Figure \ref{Compl:RMSE}. Let us note that the multi-level SA estimator has been computed for different values of $m$ (ranging from $m=2$ to $m=7$) and different values of $L$. We set $\gamma(p)=\gamma_0/p$, with $\gamma_0=2$, $p\geq1$, so that $\beta^{*}=1/2$.

\begin{figure}[!ht]
\begin{center}
\includegraphics[width=8cm,height=7cm]{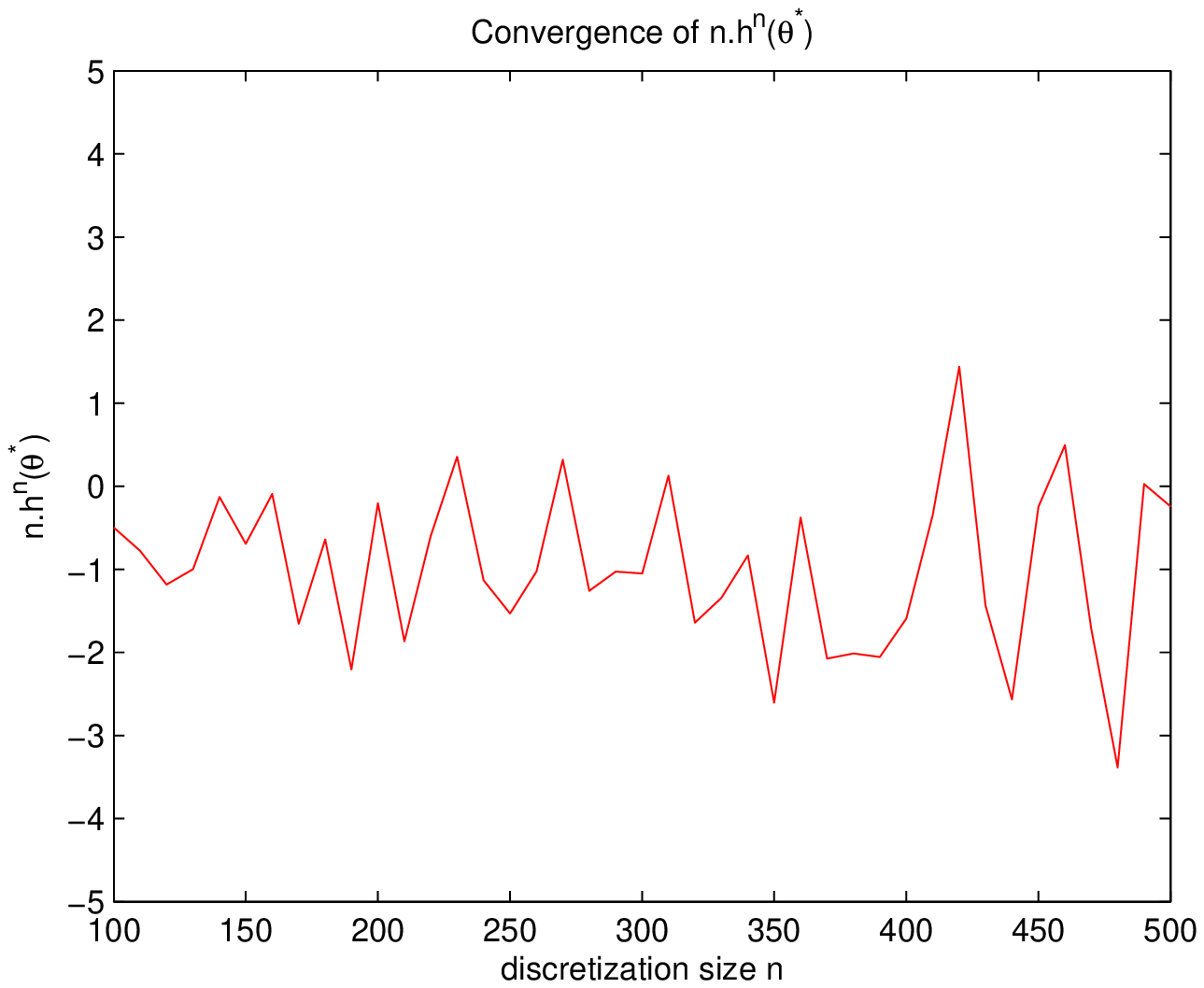}
\includegraphics[width=8cm,height=7cm]{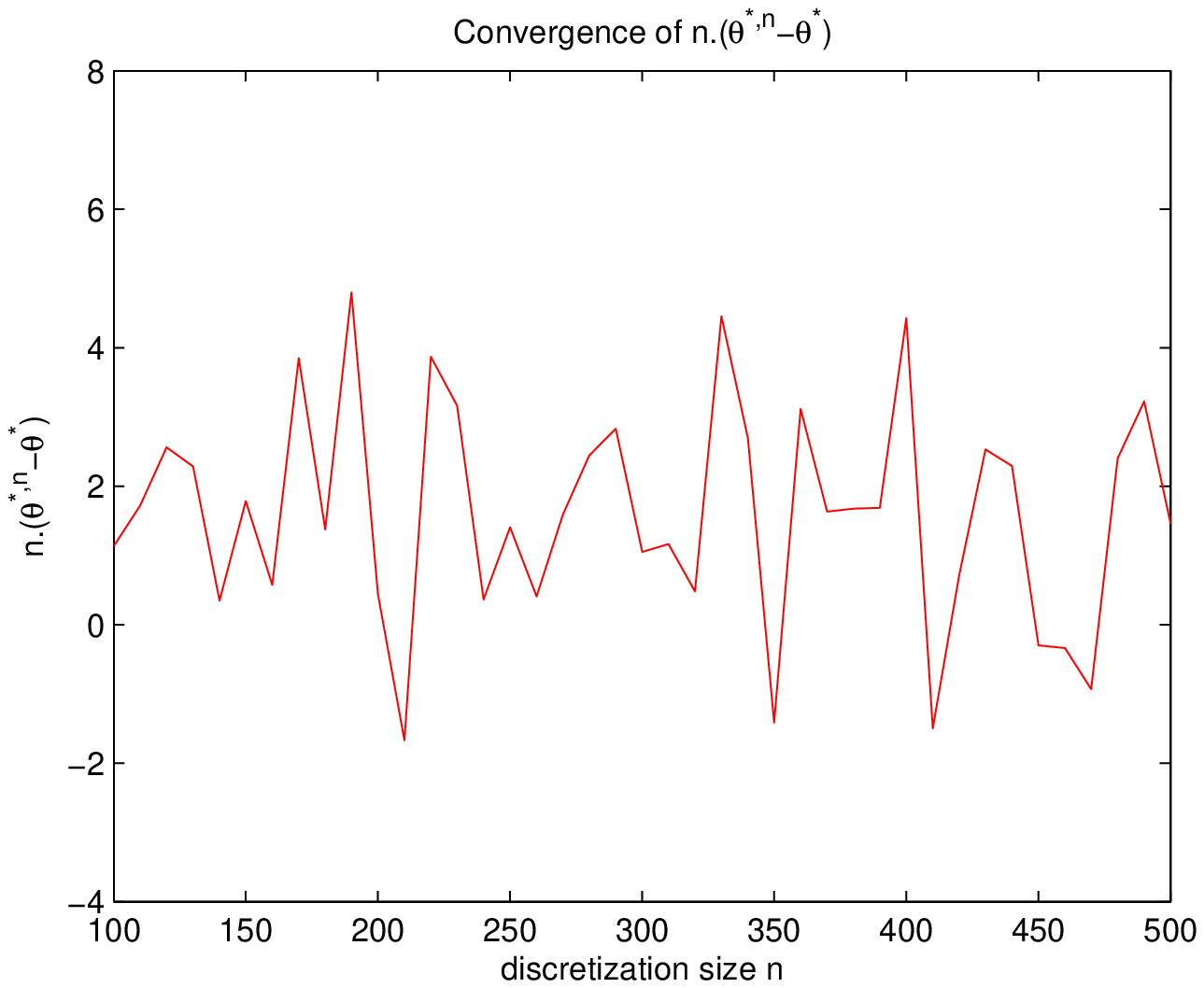}
\caption{On the left: Weak discretization error $n\mapsto n h^{n}(\theta^*)$. On the right: Implicit discretization error $n\mapsto n(\theta^{*,n}-\theta^*)$, $n=100, \cdots, 500$.}
\label{Conv:WeakDisc:Theta:Call}
\end{center}
\end{figure}

From a practical point of view, it is of interest to use the information provided at level 1 by the Statistical Romberg SA estimator and at each level by the multi-level SA estimator. More precisely, the initialization point of the SA procedures devised to compute the correction terms $\theta^{n}_{\gamma_0 n^{3/2}} - \theta^{\sqrt{n}}_{\gamma_0 n^{3/2}}$ (for the statistical Romberg SA) and $\theta^{m^{\ell}}_{M_{\ell}} - \theta^{m^{\ell-1}}_{M_{\ell}}$ (for the Multi-level SA) at level $\ell$ are fixed to $\theta^{\sqrt{n}}_{\gamma_0 n^{2}}$ and to $ \theta^{1}_{\gamma_0 n^{2}}  + \sum_{\ell=1}^{L-1} \theta^{m^{\ell}}_{M_\ell} - \theta^{m^{\ell-1}}_{M_{\ell}}$ respectively.  We set $\theta^{n^{1/2}}_0= \theta^{1}_{0} = x$ for all $k\in \left\{1, \cdots, M\right\}$ to initialize the procedures. Moreover, by Lemma \ref{sstrongerror:tech:lemme}, the $L^{1}(\P)$-norm of an increment of a SA algorithm is of order $\sqrt{\gamma_0/p}$ since $\E|\theta^{n}_{p+1}-\theta^{n}_p| \leq \E[|\theta^{n}_{p+1}-\theta^{*,n}|^2]^{1/2} + \E[|\theta^{n}_p-\theta^{*,n}|^2]^{1/2} \leq C(H,\gamma) \sqrt{\gamma(p)}$. Hence, during the first iterations (say $M/100$ if $M$ denotes the number of samples of the estimator), to ensure that the different procedures do not jump too far ahead in one step, we freeze the value of $\theta^{\sqrt{n}}_{p+1}$ (respectively $\theta^{m^{\ell}}_{p+1}$) and reset it to the value of the previous step as soon as $|\theta^{\sqrt{n}}_{p+1}-\theta^{\sqrt{n}}_{p}| \leq K/\sqrt{p}$ (respectively $|\theta^{m^{\ell}}_{p+1}-\theta^{m^{\ell}}_{p}| \leq K/\sqrt{p}$), for a pre-specified value of $K$. This is just an heuristic approach that notably prevents the algorithm from blowing up during the first steps of the procedure. We select $K=5$ in the different procedures. Note anyway that this projection-reinitialization step does not lead to additional bias but slightly increases the complexity of each procedures. In our numerical examples, we observe that it only represents around $1$-$2 \%$ of the total complexity.
 
Now let us interpret Figure \ref{Compl:RMSE}. The curves of the statical romberg SA and the multi-level SA methods are displaced below the curve of the SA method. Therefore, for a given error, the complexity of both methods are much lower than the one of the crude SA. The difference in terms of computational cost becomes more significant as the RMSE is small, which corresponds to large values of $n$. The difference between the statistical romberg and the multi-level SA method is not significant for small values of $n$, $i.e.$ for a RMSE between $1$ and $0.1$. For a RMSE lower than $5.10^{-2}$, which corresponds to a number of steps $n$ greater than about $600$-$700$, we observe that the multi-level SA procedure becomes much more effective than both methods. For a RMSE fixed around $1$ (which corresponds to $n=100$ for the SA algorithm and Statiscal Romberg SA), one divides the complexity by a factor of approximately $5$ by using the statistical romberg SA. For a RMSE fixed at $10^{-1}$, the computational cost gain is approximately equal to $10$ by using either the statistical romberg SA algorithm or the multi-level SA one. Finally, for a RMSE fixed at $5.5 . 10^{-2}$, the complexity gain achieved by using the multi-level SA procedure instead of the statistical romberg one is approximately equal to $5$. 
 
 The histograms of Fig \ref{Hist:SA-SR-ML:Strike} illustrates Theorems \ref{globalTCL}, \ref{TWO:LVL:SA} and \ref{MLVL:SA}. The distributions of $n(\theta^{n}_{\gamma^{-1}(1/n^2)}-\theta^*)$, $n(\Theta^{sr}_n-\theta^*)$ and $n(\Theta^{ml}_n-\theta^*)$, obtained with $n=4^{4}=256$ and $N=1000$ samples, are close to a normal distribution.

\begin{figure}[!ht]
\begin{center}
\includegraphics[width=8cm,height=8cm, keepaspectratio=true]{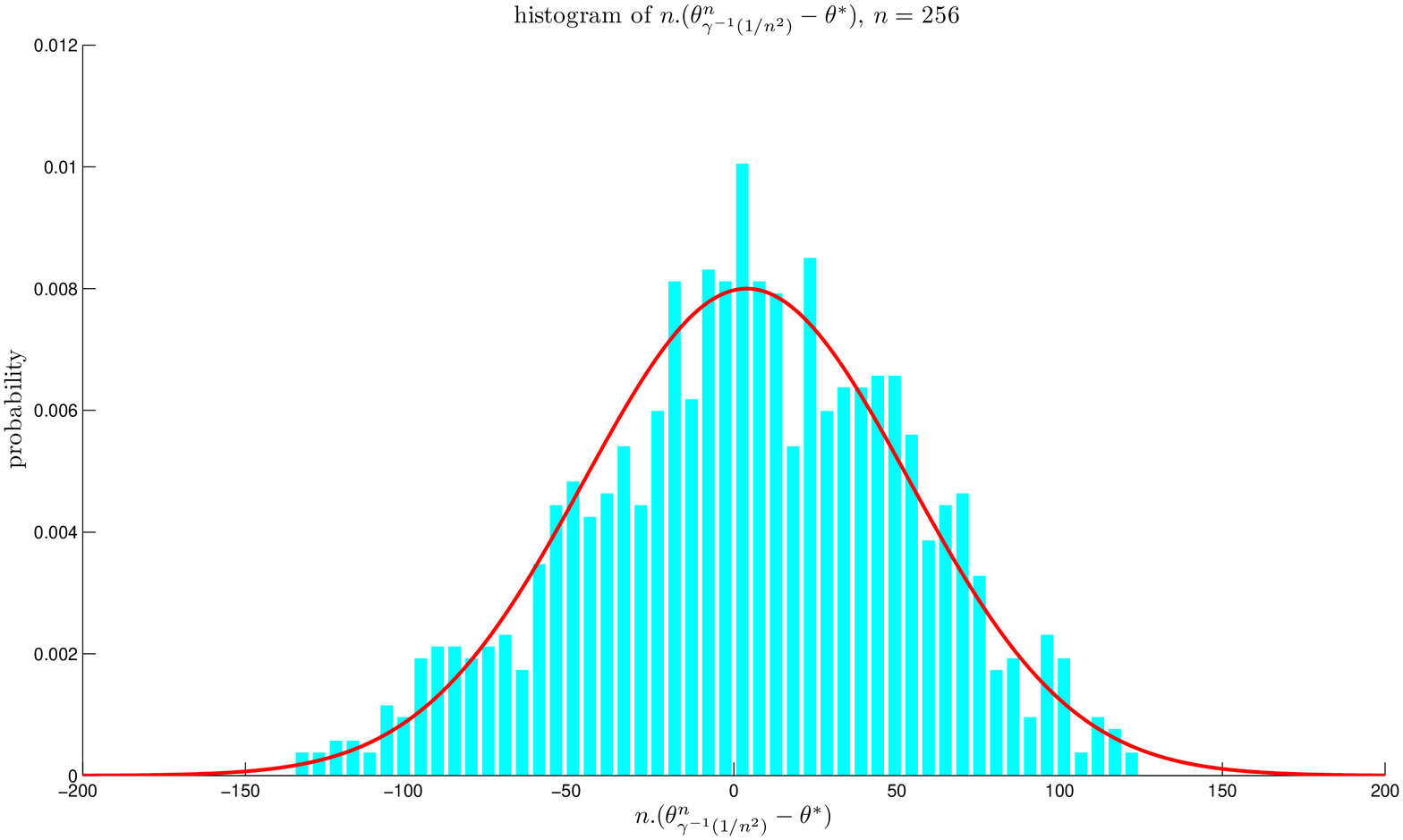}
\includegraphics[width=8cm,height=8cm, keepaspectratio=true]{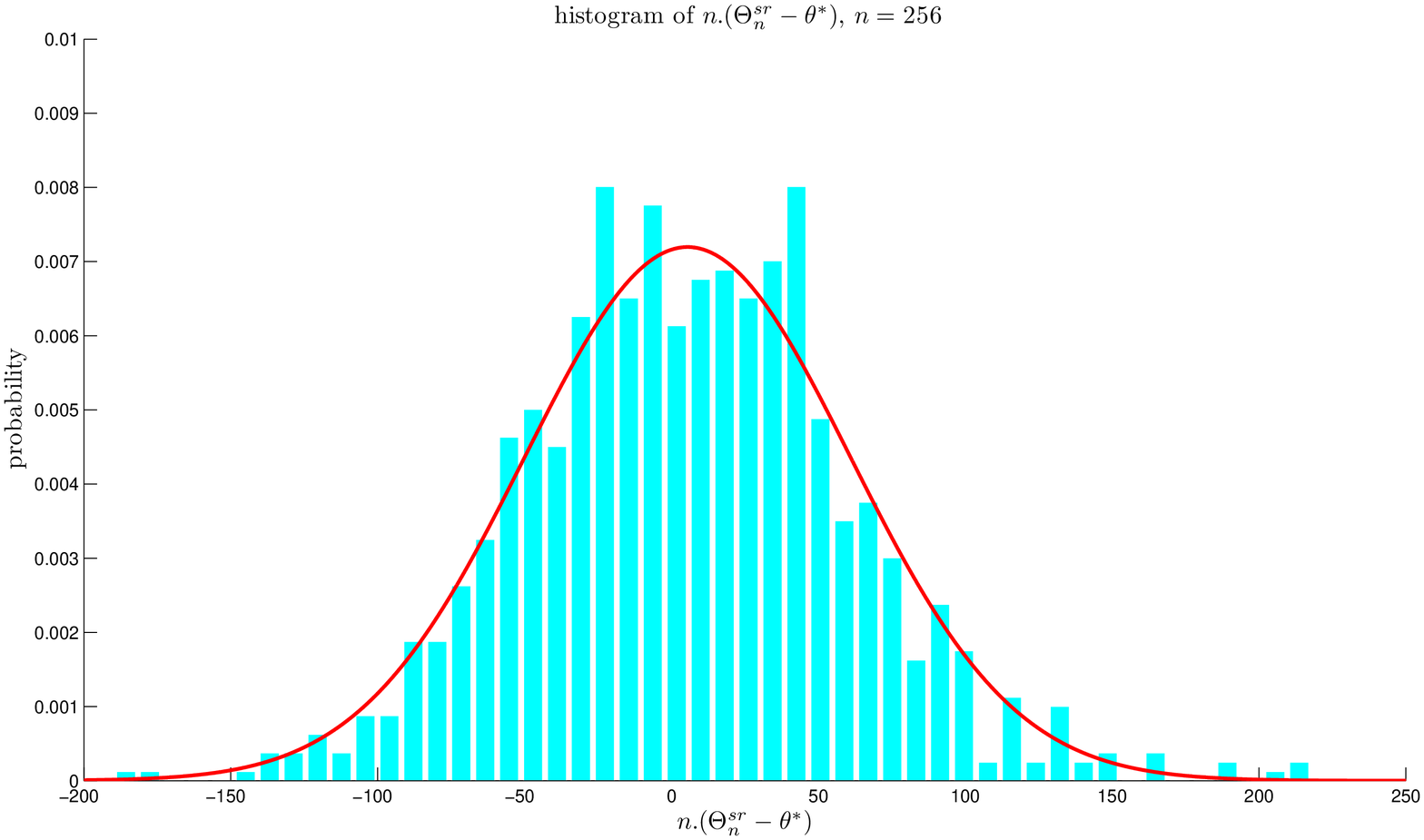}
\includegraphics[width=8cm,height=8cm, keepaspectratio=true]{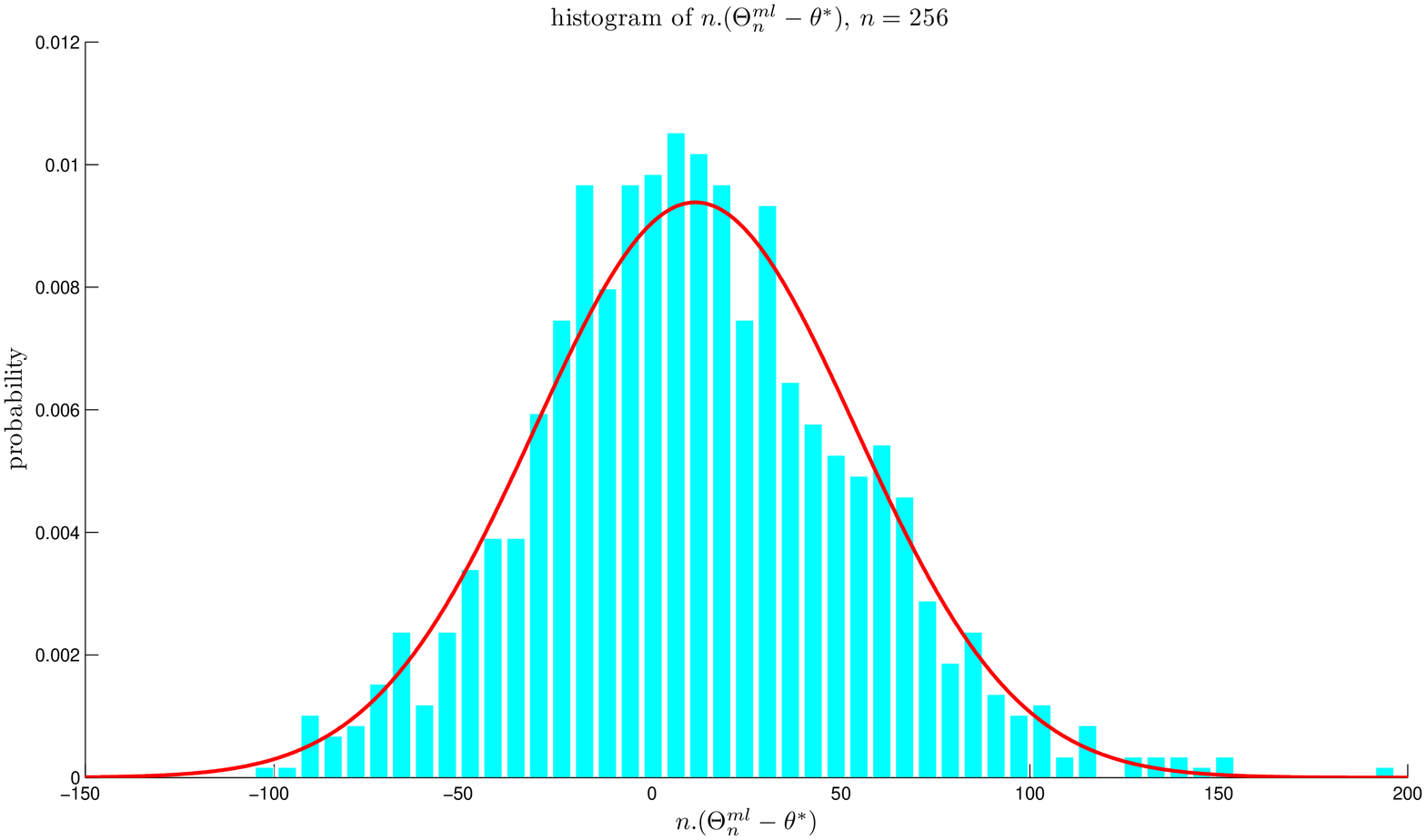}
\caption{Histograms of $n(\theta_{\gamma^{-1}(1/n^{2})}-\theta^*)$, $n(\Theta^{sr}_n-\theta^*)$ and $n(\Theta^{ml}_n-\theta^*)$ (from left to right), $n=256$, with $N=1000$ samples.}
\label{Hist:SA-SR-ML:Strike}
\end{center}
\end{figure}
\begin{figure}
\begin{center}
\hspace*{-6.5cm}
\includegraphics[width=30cm,height=15cm]{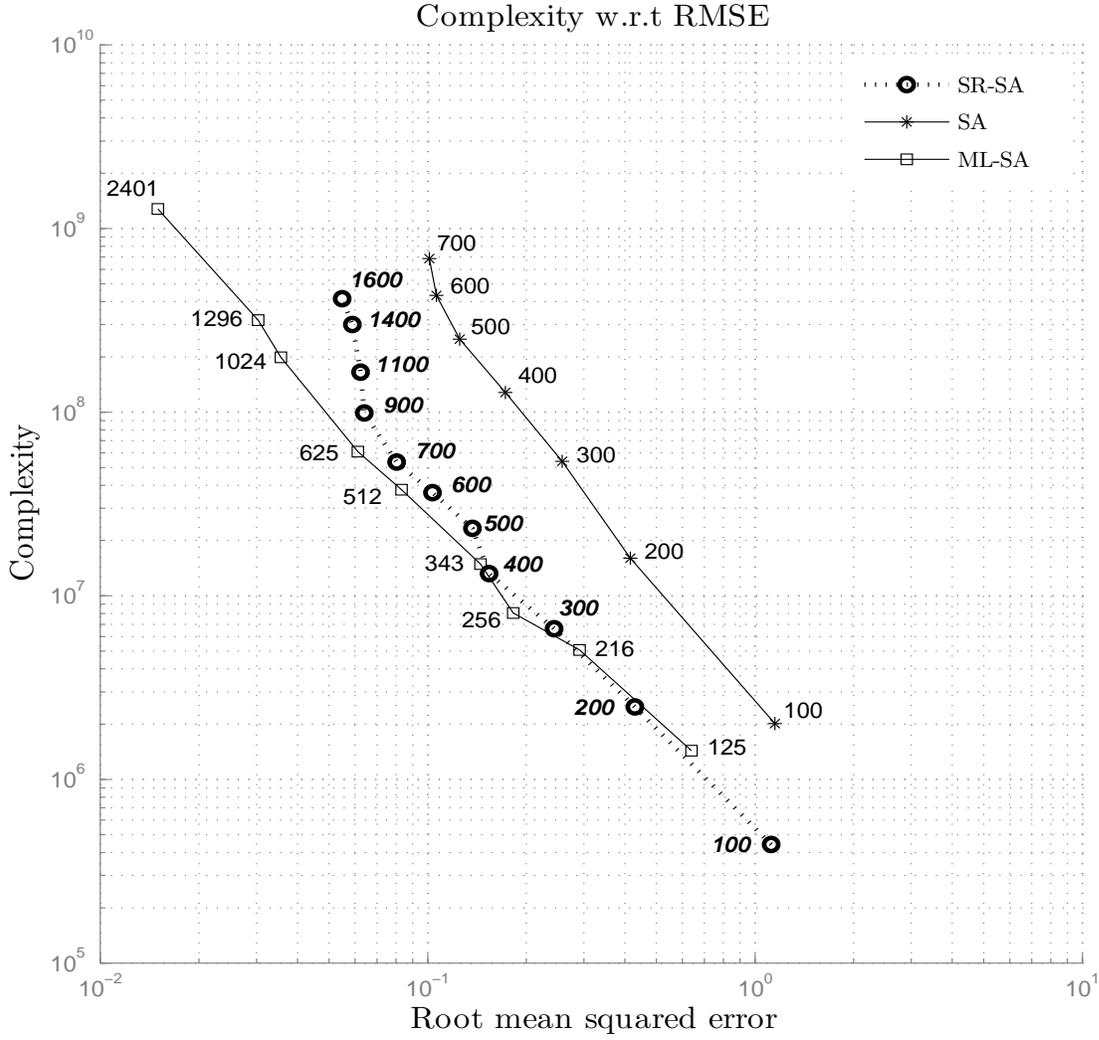}
\caption{Complexity with respect to RMSE.}
\label{Compl:RMSE}
\end{center}
 \end{figure}
 \begin{figure}
\begin{center}
\hspace*{-6.5cm}
\includegraphics[width=30cm,height=15cm]{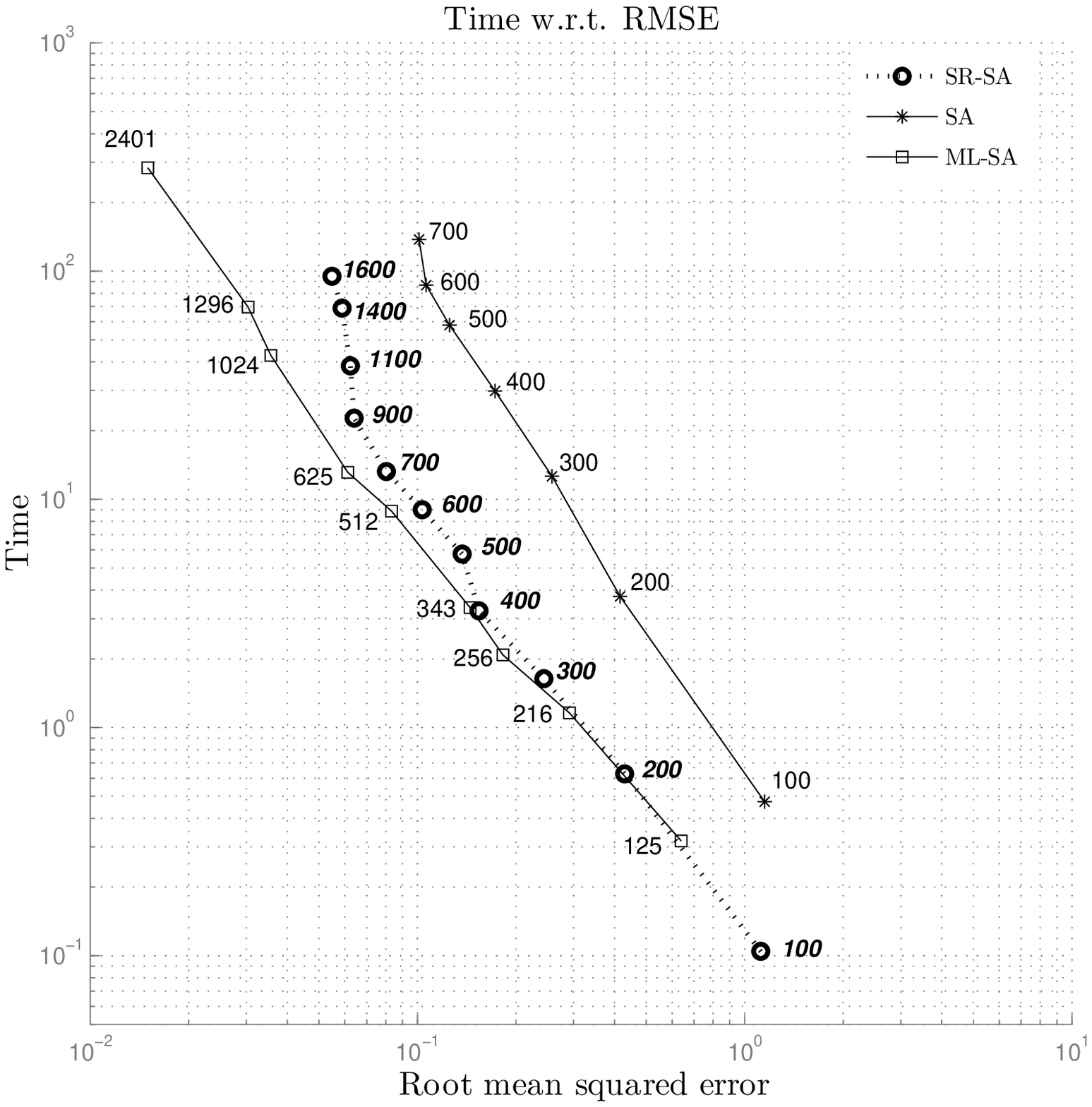}
\caption{Time in second (average time for one sample) with respect to RMSE.}
\label{Compl:RMSE}
\end{center}
 \end{figure}
 \section{Technical results}\label{technical:res:sec}
We provide here some useful technical results that are used repeatedly throughout the paper. When the exact value of a constant is not important we may repeat the same symbol for constants that may change from one line to next.

\begin{LEMME}
\label{stepseq:tech:lemme}
Let $H$ be a stable $d\times d$ matrix and denote by $\lambda_{min}$ its eigenvalue with the lowest real part. Let $(\gamma_n)_{n\geq1}$ be a sequence defined by $\gamma_n = \gamma(n)$, $n\geq1$, where $\gamma$ is a positive function defined on $[0,+\infty[$ decreasing to zero and such that $\sum_{n\geq1} \gamma(n)=+\infty$. Let $a,b>0$. We assume that $\gamma$ satisfies one of the following assumptions:
\begin{itemize}

\item $\gamma$ varies regularly with exponent $(-c)$, $c \in [0,1)$, that is for any $x>0$, $\lim_{t\rightarrow + \infty} \gamma(tx)/\gamma(t) = x^{-c}$. 

\item for $t\geq1$, $\gamma(t)=\gamma_0/t$ with $b \mathcal{R}e(\lambda_{min})\gamma_0>a$.

\end{itemize}

Let $(v_n)_{n\geq1}$ be a non-negative sequence. Then, for some positive constant $C$, one has
$$
\lim\sup_{n} \gamma^{-a}_{n} \sum_{k=1}^{n} \gamma^{1+a}_{k}  v_k \|\Pi_{k+1,n}\|^{b} \leq C \lim\sup_{n} v_n,
$$

\noindent where $\Pi_{k,n}:=\prod_{j=k}^{n}(I_d-\gamma_{j}H)$, with the convention $\Pi_{n+1,n}=I_d$.

\end{LEMME}

\begin{proof} First, from the stability of $H$, for all $0<\lambda < \mathcal{R}e(\lambda_{min})$, there exists a positive constant $C$ such that for any $k\leq n$,  $\|\Pi_{k+1,n}\| \leq C \prod_{j=k}^{n} (1-\lambda \gamma_j)$. Hence, we have  $ \sum_{k=1}^{n} \gamma^{1+a}_{k} v_k \|\Pi_{k+1,n}\|^{b} \leq C\sum_{k=1}^{n} \gamma^{1+a}_{k} v_k e^{-\lambda b (s_n-s_k)}$, $n\geq1$, with $s_n:=\sum_{k=1}^{n} \gamma_k$. We set $z_n:=\sum_{k=1}^{n} \gamma^{1+a}_{k} v_k e^{-\lambda b (s_n-s_k)}$. It can written in the recursive form
$$
z_{n+1} = e^{-\lambda b \gamma_{n+1}} z_n + \gamma^{a+1}_{n+1} v_{n+1}, \ n \geq 0.
$$

 Hence, a simple induction shows that for any $n>N$, $N\in \N^{*}$ 
\begin{align*}
z_{n} & = z_{N} \exp(-\lambda b (s_{n}-s_{N})) + \exp(-\lambda b s_{n}) \sum_{k=N+1}^{n} \exp(\lambda b s_{k}) \gamma^{a+1}_{k} v_{k} \\
& \leq z_{N} \exp(-\lambda b (s_{n}-s_{N})) + \left(\sup_{k>N} v_{k}\right) \exp(-\lambda b s_{n}) \sum_{k=N+1}^{n} \exp(\lambda b s_k) \gamma^{a+1}_k.
\end{align*}

\noindent We study now the impact of the step sequence $(\gamma_p)_{p\geq1}$ on the above estimate. We first assume that $\gamma_p= \gamma_0/p$ with $ b\mathcal{R}e(\lambda_{min}) \gamma_0 > a$. We select $\lambda>0$ such that $b\mathcal{R}e(\lambda_{min}) \gamma_0> b\lambda \gamma_0 > a$. Then, one has $s_{p}=\gamma_0 \log(p) + c_{1} + r_{p}$, $c_{1}>0$ and $r_{p}\rightarrow 0$ so that a comparison between the series and the integral yields
$$
 \exp(-\lambda b s_{n}) \sum_{k=N+1}^{n} \exp(\lambda b s_k) \gamma^{a+1}_k \leq C \frac{1}{n^{b \lambda \gamma_0}}\sum_{k=N+1}^{n} \frac{1}{k^{a-b\lambda \gamma_0 +1}} \leq \frac{C}{n^{a}}
$$

\noindent for some positive constant $C$ (independent of $N$) so that we clearly have
$$
\lim\sup_{n} \gamma^{-a}_n z_{n+1} \leq  C \sup_{k>N} v_{k}.
$$

\noindent and we conclude by passing to the limit $N\rightarrow + \infty$. 

We now assume that $\gamma$ varies regularly with exponent $-c$, $c \in [0,1)$. Let $s(t)=\int_0^t \gamma(s) ds$. We have 
\begin{align*}
\exp(-\lambda b s_n) \sum_{k=N}^{n} \exp(\lambda b s_{k}) \gamma^{a+1}_{k+1} & \sim \exp(-\lambda b s(n)) \int_0^{n} \exp(\lambda b s(t)) \gamma^{a+1}(t) dt \\
& \sim  \exp(-\lambda b s(n) ) \int_0^{s(n)} \exp(\lambda b t) \gamma^{a}(s^{-1}(t)) dt,
\end{align*}

\noindent so that for any $x$ such that $0<x<1$, since $t\mapsto \gamma^{a}(s^{-1}(t))$ is decreasing, we deduce
\begin{align*}
\int_0^{s(n)} \exp(\lambda b t) \gamma^{a}(s^{-1}(t)) dt & \leq \gamma^{a}(s^{-1}(0)) \int_0^{x s(n)} \exp(\lambda b t) dt + \gamma^{a}(s^{-1}(xs(n))) \int_{xs(n)}^{s(n)} \exp(\lambda b t) dt \\
& \leq \frac{\gamma^{a}(s^{-1}(0))}{\lambda b} \exp(\lambda b xs(n)) + \frac{\gamma^{a}(s^{-1}(xs(n)))}{\lambda b} \exp(\lambda b s(n)). 
\end{align*}

Hence it follows that
$$
\frac{\exp(-\lambda b s(n))}{\gamma^{a}(n)} \int_0^{s(n)} \exp(\lambda b t) \gamma^{a+1}(t) dt  \leq  \frac{\gamma(s^{-1}(0))}{\lambda \gamma^{a}(n)} \exp(-\lambda b(1-x) s(n)) +  \frac{\gamma^{a}(s^{-1}(xs(n)))}{\lambda b \gamma^{a}(n)},
$$

\noindent and since $t\mapsto \gamma^{a}(s^{-1}(t))$ varies regular with exponent $-a c /(1-c)$, and $\lim_{n \rightarrow + \infty} \frac{1}{\gamma^{a}(n)} \exp(-\lambda(1-x) s(n))=0$,
$$
\limsup_{n\rightarrow + \infty} \frac{\exp(-\lambda b s(n))}{\gamma^{a}(n)} \int_0^{s(n)} \exp(\lambda b t) \gamma^{a+1}(t) dt  \leq \frac{x^{-ac/(1-c)}}{\lambda b}.
$$

\noindent An argument similar to the previous case concludes the proof. 

\end{proof}

\begin{LEMME}
\label{sstrongerror:tech:lemme}
Let $(\theta^{n}_{p})_{p\geq0}$ be the procedure defined by \eqref{RM} where $\theta^{n}_0$ is independent of the innovation of the algorithm with $\sup_{n\geq1}\E|\theta^{n}_0|^2<+\infty$. Suppose that the assumption of theorem \ref{thm:conv:disc} are satisfied and that the mean-field function $h^{n}$ satisfies
\begin{equation}
\label{strongmeanrevert:assump}
\exists \underline{\lambda} >0, \ \forall n\in \N^{*}, \ \forall \theta \in \R^d, \ \langle  \theta-\theta^{*,n}, h^{n}(\theta) \rangle  \geq \underline{\lambda} |\theta-\theta^{*,n}|^2,
\end{equation}

\noindent where $\theta^{*,n}$ is the unique zero of $h^{n}$ satisfying $\sup_{n\geq1} |\theta^{*,n}| <+\infty$. Moreover, we assume that $\gamma$ satisfies one of the following assumptions:
\begin{itemize}

\item $\gamma$ varies regularly with exponent $(-c)$, $c \in [0,1)$, that is for any $x>0$, $\lim_{t\rightarrow + \infty} \gamma(tx)/\gamma(t) = x^{-c}$. 

\item for $t\geq1$, $\gamma(t)=\gamma_0/t$ with $2 \underline{\lambda} \gamma_0>1$.

\end{itemize}

Then, for some positive constant $C$ (independent of $p$ and $n$) one has:
$$
\forall p \geq 1, \ \ \sup_{n\geq1}\E[|\theta^{n}_p-\theta^{*,n}|^2] + \E[|\theta_p-\theta^*|^2] \leq C \gamma(p).
$$
\end{LEMME}
\begin{proof} From the dynamic of $(\theta^{n}_{p})_{p\geq1}$, we have
\begin{align*}
|\theta^{n}_{p+1}-\theta^{*,n}|^2 & = |\theta^{n}_{p} - \theta^{*,n}|^2 - 2 \gamma_{p+1}  \langle \theta^{n}_p-\theta^{*,n}, h^{n}(\theta^{n}_p) \rangle + 2 \gamma_{p+1} \langle \theta^{n}_p-\theta^{*,n},  \Delta M^{n}_{p+1} \rangle  \\
& + \gamma^{2}_{p+1} |H(\theta^{n}_p, (X^{n}_T)^{p+1})|^2,
\end{align*}

\noindent so that  taking expectation in the previous equality and using assumptions \eqref{Growth_Cond} and \eqref{strongmeanrevert:assump}, we easily derive
 $$
\E |\theta^{n}_{p+1}-\theta^{*,n}|^2 \leq (1-2\underline{\lambda}\gamma_{p+1} + C \gamma^{2}_{p+1}) \E|\theta^{n}_{p} - \theta^{*,n}|^2  + C \gamma^{2}_{p+1}.
 $$

Now a simple induction argument yields
\begin{align*}
\E |\theta^{n}_{p}-\theta^{*,n}|^2 \leq \E|\theta^{n}_0 - \theta^{*,n}|^2 \Pi_{1,p} + \sum_{k=1}^{p} \Pi_{k+1,p} \gamma^{2}_{k}
\end{align*}

\noindent where we set $\Pi_{k,p}:=\prod_{j=k}^{p}(1-2 \underline{\lambda} \gamma_{j} + C \gamma^{2}_{j})$ for sake of simplicity. Moreover, computations similar to the proof of Lemma \ref{stepseq:tech:lemme} imply
$$
\forall p \geq1, \ \ \E |\theta^{n}_{p}-\theta^{*,n}|^2 \leq C \gamma(p).
$$

In order to prove the similar bound for the sequence $(\theta_p)_{p\geq1}$ we first observe that since $(\theta_p)_{p\geq1}$ converges $a.s.$ to $\theta^*$ there exists a compact set $K$ (which depends on $w$) such that $\theta_p \in K$, for $p\geq0$. Then, Remark 2.3 shows that a mean reverting assumption is satisfied also for $h$ on $K$ with the same constant $\underline{\lambda}$. Finally we conclude using similar arguments as those used above. 
\end{proof}

\begin{PROP}\label{dec:aver:step}Assume that the assumptions of Theorem \ref{TWO:LVL:SA:AV} are satisfied. Then, for all $n \in \N$ there exist two sequences $(\tilde{\mu}^{n}_{p})_{p \in \leftB0,n\rightB}$ and $(\tilde{r}^{n}_p)_{p\in \leftB 0,n \rightB}$ with $\tilde{r}^{n}_0 = \theta^{n}_0 -\theta_0 - (\theta^* - \theta^{*,n})$ such that
$$
\forall p \in \leftB0,n\rightB, \ \ z^{n}_p = \theta^{n}_p - \theta^{*,n} - (\theta_p - \theta^*) = \tilde{\mu}^{n}_p + \tilde{r}^{n}_{p} 
$$

\noindent and satisfying for all $n \in \N$, for all $p \in \leftB 1,n \rightB$
$$
\sup_{p\geq1} \gamma^{-1/2}_{p} \E|\tilde{\mu}^{n}_p| < C n^{-\rho}, \ \  \ \sup_{n\geq1, p\geq0}\gamma^{-1}_{p}\E[ |\tilde{r}^{n}_{p}|] < +\infty. 
$$

%

\end{PROP}

\begin{proof}
Using \eqref{rec:two:level:sr}, we define the two sequences $(\tilde{\mu}^{n}_{p})_{p \in \leftB0,n\rightB}$ and $(\tilde{r}^{n}_p)_{p\in \leftB 0,n \rightB}$ by
\begin{align*}
\tilde{\mu}^{n}_p & =  \sum_{k=1}^{p} \gamma_k \Pi_{k+1,p} \Delta N^{n}_k + \sum_{k=1}^{p} \gamma_k \Pi_{k+1,p} (Dh(\theta^{*})-Dh^{n}(\theta^{*,n}))(\theta^{n}_{k-1}-\theta^{*,n}) \\
& +  \sum_{k=1}^{p} \gamma_k  \Pi_{k+1, n} ( h^{n}(\theta^{*,n})-h^{n}(\theta^{*}) - (H(\theta^{*,n},(U^{n})^{k+1})-H(\theta^{*},(U^{n})^{k+1}))  )
\end{align*}

\noindent and
\begin{align*}
\tilde{r}^{n}_{p} & = \Pi_{1,p} z^{n}_0 + \sum_{k=1}^{p} \gamma_k \Pi_{k+1,p} (\zeta^{n}_{k-1} - \zeta_{k-1}) +  \sum_{k=1}^{p} \gamma_k  \Pi_{k+1, p} ( h^{n}(\theta^{n}_k)-h^{n}(\theta^{*,n}) - (H(\theta^{n}_k,(U^{n})^{k+1})-H(\theta^{*,n},(U^{n})^{k+1}))   ) \\
& + \sum_{k=1}^{p} \gamma_k  \Pi_{k+1, p} ( H(\theta_k, U^{k+1}) - H(\theta^*, U^{k+1}) - (h(\theta_k) - h(\theta^*) ) ).
\end{align*}

We first focus on the sequence $(\tilde{\mu}^{n}_p)_{p\in \leftB 0, n \rightB}$. Moreover, by the definition of the sequence $(\Delta N^{n}_{k})_{k \in \leftB 1, n \rightB}$ and the Cauchy-Schwarz inequality we derive
$$
\E \left| \sum_{k=1}^{p} \gamma_k \Pi_{k+1,p} \Delta N^{n}_k \right| \leq C  (\E|H(\theta^*,U^{n}) - H(\theta^*, U)|^2)^{1/2} (\sum_{k=1}^{p} \gamma^{2}_k \|\Pi_{k+1,p}\|^2)^{1/2} = \O( \gamma^{1/2}_p  n^{-\rho}).
$$

Taking the expectation for the third term and following the lines of the proof of Lemma \ref{conv:lem:2step}, we obtain
\begin{align*}
\E \left| \sum_{k=1}^{p} \gamma_k \Pi_{k+1,p} (Dh(\theta^{*})-Dh^{n}(\theta^{*,n}))(\theta^{n}_{k-1}-\theta^{*,n}) \right| &  \leq C \sum_{k=1}^{p} \gamma^{3/2}_k \|\Pi_{k+1,p}\| (|\theta^{*,n} - \theta^*| + \|Dh(\theta^*) - Dh^{n}(\theta^*)\|) \\
& = \O(\gamma^{1/2}_p n^{-\rho}).
 \end{align*}
 
Finally we take the square of the $L^{2}$-norm of the last term and use Lemma \ref{stepseq:tech:lemme} to derive
\begin{align*}
\E \left|\sum_{k=1}^{p} \gamma_k  \Pi_{k+1, p} ( h^{n}(\theta^{*,n})-h^{n}(\theta^{*}) - (H(\theta^{*,n},(X^{n}_T)^{k+1}) - H(\theta^{*},(X^{n}_T)^{k+1}))  ) \right|^2  & \leq |\theta^*-\theta^{*,n}|^2 \sum_{k=1}^{p} \gamma^{2}_k  \|\Pi_{k+1, p}\|^2 \\
& = \O(\gamma_p n^{-2\rho}).
\end{align*}

We now prove the bound concerning the sequence $(\tilde{r}^{n}_{p})_{p\in \leftB 0, n \rightB}$. Under the assumption on the step sequence we have
$$
\E[| \Pi_{1,p} z^{n}_{0}|] \leq \| \Pi_{1,p}\| (1 + | \theta^{*}-\theta^{*,n}|) = \O(\gamma_p).
$$

By Lemma \ref{sstrongerror:tech:lemme}, we derive
\begin{align*}
\sup_{n\geq1}\E\left| \sum_{k=1}^{p} \gamma_k \Pi_{k+1,p} (\zeta^{n}_{k-1} - \zeta_{k-1}) \right|  &  \leq C  \sum_{k=1}^{p} \gamma^{2}_k \|\Pi_{k+1,p}\|  = \O(\gamma_p).
\end{align*}


Concerning the second term, following the lines of the proof of Lemma \ref{conv:lem:2step} we simply take the square of its $L^2(\P)$-norm to derive
\begin{align*}
\sup_{n\geq1} \E \left|  \sum_{k=1}^{p} \gamma_k  \Pi_{k+1, p} ( h^{n}(\theta^{n}_k)-h^{n}(\theta^{*,n}) - (H(\theta^{n}_k,(X^{n}_T)^{k+1})-H(\theta^{*,n},(X^{n}_T)^{k+1}))   ) \right|^2 & \leq C  \sum_{k=1}^{p} \gamma^3_k  \|\Pi_{k+1, p}\|^2  \\
& = \O(\gamma^{2}_p)
\end{align*}

\noindent and similarly $ \E \left|  \sum_{k=1}^{p} \gamma_k  \Pi_{k+1, p} \left( (H(\theta_k,(X_T)^{k+1})-H(\theta^{*},(X_T)^{k+1}))   - (h(\theta_k) - h(\theta^{*})) \right) \right|^2  = \O(\gamma^{2}_p)$.

%
%
\end{proof}

\bibliographystyle{alpha}
\bibliography{bibli}

\begin{thebibliography}{BFP09b}

\bibitem[AE78]{ald:eagl:78}
D.~J. Aldous and G.~K. Eagleson.
\newblock On mixing and stability of limit theorems.
\newblock {\em Ann. Probability}, 6(2):325--331, 1978.

\bibitem[AK12]{ben:keb:12}
M.~Ben Alaya and A.~Kebaier.
\newblock Central limit theorem for the multilevel monte carlo euler method and
  applications to asian options.
\newblock {\em Preprint}, 2012.

\bibitem[BFP09a]{bar:fri:pag:09}
O.~Bardou, N.~Frikha, and G.~Pag{\`e}s.
\newblock Computing {V}a{R} and {CV}a{R} using stochastic approximation and
  adaptive unconstrained importance sampling.
\newblock {\em Monte Carlo Methods Appl.}, 15(3):173--210, 2009.

\bibitem[BFP09b]{bar:fri:pag:09:2}
O.~Bardou, N.~Frikha, and G.~Pag{\`e}s.
\newblock Recursive computation of value-at-risk and conditional value-at-risk
  using {MC} and {QMC}.
\newblock In {\em Monte {C}arlo and quasi-{M}onte {C}arlo methods 2008}, pages
  193--208. Springer, Berlin, 2009.

\bibitem[BFP10]{bar:fri:pag:10}
O.~Bardou, N.~Frikha, and G.~Pag{\`e}s.
\newblock C{V}a{R} hedging using quantization based stochastic approximation
  algorithm.
\newblock {\em forthcoming in Mathematical Finance}, 2010.

\bibitem[BMP90]{ben:met:pri}
A.~Benveniste, M.~M{\'e}tivier, and P.~Priouret.
\newblock {\em Adaptive algorithms and stochastic approximations}, volume~22 of
  {\em Applications of Mathematics (New York)}.
\newblock Springer-Verlag, Berlin, 1990.
\newblock Translated from the French by Stephen S. Wilson.

\bibitem[Der11]{Dereich11}
S.~Dereich.
\newblock Multilevel {M}onte {C}arlo algorithms for {L}\'evy-driven {SDE}s with
  {G}aussian correction.
\newblock {\em Ann. Appl. Probab.}, 21(1):283--311, 2011.

\bibitem[DG95]{duff:glyn:1995}
D~Duffie and P.~Glynn.
\newblock Efficient monte carlo simulation of security prices.
\newblock {\em Ann. Appl. Probab.}, 5(4):897--905, 1995.

\bibitem[Duf96]{Duflo1996}
M.~Duflo.
\newblock {\em Algorithmes stochastiques}, volume~23 of {\em Math\'ematiques \&
  Applications (Berlin) [Mathematics \& Applications]}.
\newblock Springer-Verlag, Berlin, 1996.

\bibitem[FF13]{fat:fri:13}
M.~Fathi and N.~Frikha.
\newblock Transport-entropy inequalities and deviation estimates for stochastic
  approximation schemes.
\newblock {\em Electron. J. Probab.}, 18:no. 67, 1--36, 2013.

\bibitem[FM12]{fri:men:12}
N.~Frikha and S.~Menozzi.
\newblock Concentration bounds for stochastic approximations.
\newblock {\em Electron. Commun. Probab.}, 17:no. 47, 15, 2012.

\bibitem[Fri14]{fri:14}
N.~Frikha.
\newblock Shortfall {R}isk {M}inimization in {D}iscrete {T}ime {F}inancial
  {M}arket {M}odels.
\newblock {\em SIAM J. Financial Math.}, 5(1):384--414, 2014.

\bibitem[GHM09]{Giles:09}
M.~B. Giles, D.~J. Higham, and X.~Mao.
\newblock Analysing multi-level {M}onte {C}arlo for options with non-globally
  {L}ipschitz payoff.
\newblock {\em Finance Stoch.}, 13(3):403--413, 2009.

\bibitem[Gil08a]{Giles:08b}
M.~B. Giles.
\newblock Improved multilevel {M}onte {C}arlo convergence using the {M}ilstein
  scheme.
\newblock In {\em Monte {C}arlo and quasi-{M}onte {C}arlo methods 2006}, pages
  343--358. Springer, Berlin, 2008.

\bibitem[Gil08b]{Giles:08}
M.~B. Giles.
\newblock Multilevel {M}onte {C}arlo path simulation.
\newblock {\em Oper. Res.}, 56(3):607--617, 2008.

\bibitem[Hei01]{heinrich2001}
S.~Heinrich.
\newblock {\em Multilevel Monte Carlo methods}.
\newblock Springer, 2001.
\newblock In Large-scale scientific computing.

\bibitem[HH80]{Hall:Heyde:80}
P.~Hall and C.~C. Heyde.
\newblock {\em Martingale limit theory and its application}.
\newblock Academic Press Inc. [Harcourt Brace Jovanovich Publishers], New York,
  1980.
\newblock Probability and Mathematical Statistics.

\bibitem[Jac98]{Jacod98}
J.~Jacod.
\newblock Rates of convergence to the local time of a diffusion.
\newblock {\em Ann. Inst. H. Poincar Probab. Statist}, 34:505--544, 1998.

\bibitem[JP98]{Jac:Prot:98}
J.~Jacod and P.~Protter.
\newblock Asymptotic error distributions for the {E}uler method for stochastic
  differential equations.
\newblock {\em Ann. Probab.}, 26(1):267--307, 1998.

\bibitem[Keb05]{Keb:2005}
A.~Kebaier.
\newblock Statistical {R}omberg extrapolation: a new variance reduction method
  and applications to option pricing.
\newblock {\em Ann. Appl. Probab.}, 15(4):2681--2705, 2005.

\bibitem[KY03]{Kushner2003}
H.~J. Kushner and G.~G. Yin.
\newblock {\em Stochastic approximation and recursive algorithms and
  applications}, volume~35 of {\em Applications of Mathematics (New York)}.
\newblock Springer-Verlag, New York, second edition, 2003.
\newblock Stochastic Modelling and Applied Probability.

\bibitem[Pel98]{Pell:98}
M.~Pelletier.
\newblock Weak convergence rates for stochastic approximation with application
  to multiple targets and simulated annealing.
\newblock {\em Ann. Appl. Probab.}, 8(1):10--44, 1998.

\bibitem[PJ92]{Polyak1992}
B.~T. Polyak and A.~B. Juditsky.
\newblock Acceleration of stochastic approximation by averaging.
\newblock {\em SIAM J. Control Optim.}, 30(4):838--855, 1992.

\bibitem[R{\'e}n63]{ren:63}
A.~R{\'e}nyi.
\newblock On stable sequences of events.
\newblock {\em Sankhy\=a Ser. A}, 25:293 302, 1963.

\bibitem[RM51]{Robbins1951}
H.~Robbins and S.~Monro.
\newblock A stochastic approximation method.
\newblock {\em Ann. Math. Statistics}, 22:400--407, 1951.

\bibitem[Rup91]{Ruppert1991}
D.~Ruppert.
\newblock Stochastic approximation.
\newblock In {\em Handbook of sequential analysis}, volume 118 of {\em Statist.
  Textbooks Monogr.}, pages 503--529. Dekker, New York, 1991.

\end{thebibliography}

 \end{document}